\documentclass{article}
\usepackage{graphicx,array} 

\usepackage{amsmath,amssymb,fullpage,xcolor}
\usepackage{amsthm}
\definecolor{darkgreen}{RGB}{51,117,56}
\definecolor{boldpurple}{RGB}{46,37,113}
\definecolor{babyblue}{RGB}{148,203,236}
\definecolor{beige}{RGB}{220,205,125}
\definecolor{burgundy}{RGB}{126,041,084}
\definecolor{pinkcheeks}{RGB}{194,106,119}
\definecolor{realpurple}{RGB}{159,074,150} 
\definecolor{babyteal}{RGB}{093,168,153}
\usepackage{tikz}
\usetikzlibrary{decorations.pathreplacing}
\usetikzlibrary {shapes.geometric} 

\usepackage{ytableau, ifthen}
\usepackage{hyperref}
\usepackage{stmaryrd}
\usepackage{subcaption}
\usepackage{paralist}
\usepackage{multirow}
\usepackage{comment}

\ytableausetup{centertableaux, smalltableaux}

\newtheorem{thm}{Theorem}[section]
\newtheorem{lem}[thm]{Lemma}
\newtheorem{cor}[thm]{Corollary}
\newtheorem{prop}[thm]{Proposition}
\newtheorem{conj}[thm]{Conjecture}
\newtheorem{conjecture}[thm]{Conjecture}

\newtheorem*{thmAltRow_new}{Theorem~\ref{thm:AltRow_new}}
\newtheorem*{thmorbits}{Theorem \ref{thm:thmorbits}}
\newtheorem*{thmhomomesy}{Theorem \ref{thm:homomesy}}

\theoremstyle{definition}
\newtheorem{definition}[thm]{Definition}
\newtheorem{example}[thm]{Example}
\newtheorem{remark}[thm]{Remark}

\newcommand{\IC}{\mathcal{IC}}
\renewcommand{\O}{\mathcal{O}}
\newcommand{\rank}{\textrm{rk}}

\newcommand{\row}{\mathrm{Row}}
\newcommand{\Max}{\mathrm{Max}}
\newcommand{\Min}{\mathrm{Min}}

\newcommand{\inc}{\mathrm{Inc}}

\newcommand{\f}{\nabla}
\newcommand{\oi}{\Delta}

\newcommand{\ceil}[1]{\mathrm{Ceil}({#1})}

\newcommand{\sn}{\sigma_{(n\bmod{6})}}
\newcommand{\wn}{w_{(n\bmod{6})}}

\newcommand{\F}{\Min(I)\cap\oi\ceil{I}}
\newcommand{\arow}{\inc(I)\cup\Big(\oi\inc_{I}\big(\ceil{I}\big) -\big(I\cup\oi\ceil{I}\big)\Big)\cup\Big(\oi\ceil{I}-\oi(\F) \Big)}

\newcommand{\SC}{\mathrm{sc}}

\title{Interval-closed set rowmotion and homomesy \linebreak on products of two chains}
\author{Nadia Lafreni\`ere, Joel Brewster Lewis, Erin McNicholas, Jessica Striker, Amanda Welch}
\date{\today}

\begin{document}

\maketitle

\begin{abstract}
   We study rowmotion dynamics on interval-closed sets. Our first main result proves a simplification of the global definition of interval-closed set rowmotion from (Elder, Lafreni\`ere, McNicholas, Striker, and Welch 2024). We then completely describe the orbits of interval-closed set rowmotion on products of two chains $[2]\times[n]$ and use this understanding to prove a homomesy conjecture from (ELMSW 2024) involving the signed cardinality statistic. 
\end{abstract}

\section{Introduction}
   Rowmotion is a natural action     
   with surprisingly nice dynamical properties in certain cases.
   An interesting feature is that this action has both global and local definitions; globally, it may be seen as a convex closure operation, while locally, it is given as a composition of toggle operators.   
   Rowmotion was originally studied on order ideals of a finite poset (equivalently, distributive lattices)~\cite{SW2012} and has recently been extended to generalized realms~\cite{EinsteinPropp21,BSV21}, other combinatorial objects~\cite{Striker2018,Joseph19}, and other types of lattices~\cite{Barnard2021,RowmotionSlowmotion,Semidistrim}. Rowmotion has even been categorified~\cite{IyamaMarczinzik22}.  Many foundational results in dynamical algebraic combinatorics relate to this action; see the survey articles~\cite{Roby2016,Striker2017,HopkinsOPAC}.  
   
   Recently, Elder, Lafreni\`ere, McNicholas, Striker, and Welch \cite{ELMSW} defined rowmotion on the set $\IC(P)$ of \emph{interval-closed sets} (or convex subsets) of a poset, a natural superset of order ideals, giving both local and global characterizations.
   The local characterization is a natural toggle action for interval-closed sets, analogous to the definition for order ideals: for $x\in P$, the toggle $t_x$ acts on an interval-closed set $I$ by taking the symmetric difference of $\{x\}$ and $I$, provided this is an interval-closed set, and as the identity otherwise. The local definition of rowmotion applies all toggles in the poset from top to bottom, that is, in the reverse order of any linear extension. The global characterization of ICS rowmotion, though, is much more complicated than the analogous definition for order ideals. 
   Our first main theorem is a simplification of this global definition of rowmotion on interval-closed sets; see Section~\ref{sec:global} for the notation.
\begin{thm}\label{thm:AltRow_new}
Given an interval-closed set $I\in\IC(P)$, rowmotion on $I$ is given by
    \[ \row(I) = \Big( \inc(\ceil{I}) - I \Big)\cup\Big(\oi\ceil{I}-\oi(\F) \Big). \]
\end{thm}

The rest of our results focus on products of two chains $[m]\times[n]$. Order ideals of $[m]\times[n]$ have many nice dynamical properties: the order of rowmotion is $m+n$~\cite{AST2013,SW2012} and several statistics exhibit \emph{homomesy}~\cite{PR2015}, in which the average value of a statistic $f$ on each orbit of rowmotion equals the global average $c$. (In this case we also say $f$ is \emph{$c$-mesic}.)
 We prove some analogous theorems on $\IC([m]\times[n])$, though due to the increased complexity, we focus on the case $m=2$. The second theorem below settles \cite[Conjecture 4.12 ($m=2$)]{ELMSW}. 
\begin{thm}\label{thm:thmorbits}
    The orbits of rowmotion on $\IC([2]\times [n])$ are of sizes dividing $n+3$ or $n+5$, with the exception of a single orbit of size $2$ and the quadratic orbits described in Theorem \ref{large_orbits} and Table \ref{tab:orbit_lengths_small_posets}. In particular, the counts of orbits of each size are given in Tables \ref{tab:all_orbits} and \ref{tab:orbit_lengths_small_posets}.
\end{thm}

Some interesting features of the orbits are summarized in Remark \ref{rem:orbits_description}.

\begin{thm} \label{thm:homomesy}
    Suppose that $n$ is odd. Then the signed cardinality statistic on $\IC([2] \times [n])$ is $0$-mesic with respect to rowmotion. 
\end{thm}

This homomesy is surprising because the orbits of rowmotion on $\IC([2]\times[n])$ are much less well-behaved than in the case of order ideals. Most homomesies appearing in the literature involve actions with a small order, but this theorem proves homomesy on an action with wild orbits, like the homomesies found on noncrossing partitions in \cite{EinsteinFGJMPR16}.


We now give precise definitions of the objects and actions in the title of the paper and discuss relevant results of the prior paper \cite{ELMSW}. We also state analogous dynamical results for order ideals, for the sake of comparison.

Let $P$ be a partially ordered set (poset). All posets in this paper are finite.
A \emph{chain} of a poset $P$ is a totally ordered subset of $P$. An $n$-element \emph{chain poset} has elements $1<2<\cdots<n$ and is denoted as $[n]$. The \emph{Cartesian product} of two chains is given as $[m]\times [n]=\{(i,j) \ | \ 1\leq i\leq m, 1\leq j\leq n\}$ with partial order $(a,b)\leq (c,d)$ if and only if $a\leq c$ and $b\leq d$. An \emph{order ideal} of $P$ is a downward-closed subset, and an \emph{order filter} of $P$ is an upward-closed subset. The set of order ideals of $P$ is denoted $J(P)$. 
A \emph{linear extension} of $P$ is a total order that respects (or extends) the partial order $P$. 
See \cite[Ch.\ 3]{Stanley2011} for details and other standard poset terminology.

\begin{definition}  Let $P$ be a poset and $I$ a subset of $P$. We say that $I$ is an \emph{interval-closed set (ICS)} of $P$ if for all $x, y \in I$ and $z\in P$ such that $x \leq z \leq y$,  then $z \in I$. Let $\IC(P)$ denote the set of all interval-closed sets of $P$.
\end{definition}

The following definition contains notation that will be used in Section~\ref{sec:global}; see \cite[Examples 2.2 and 2.15]{ELMSW} for examples.

\begin{definition} 
\label{def:poset_stuff}
Given $I\in \IC(P)$, let $\oi(I)$ denote the smallest order ideal (downward-closed set) containing $I$ and $\f(I)$ denote the smallest order filter (upward-closed set) containing $I$.
Denote the minimal elements of $I$ as $\Min(I)$ and the maximal elements of $I$ as $\Max(I)$.  
Let $\inc(I)$ be the set of elements in $P$ that are incomparable with $I$ and $\inc_{I}(J)$ be the restriction of $\inc(J)$ to the subset $I$.
The \textit{ceiling} of $I$, denoted $\ceil{I}$, is the set of minimal elements of the order filter of $I$, not including $I$. That is, $\ceil{I}=\Min(\f(I)-I)$. 
\end{definition}

We now review the definition of generalized toggles from \cite{Striker2018} where we define toggles relative to any set of subsets $\mathcal{L}$ of a fixed ground set. In this paper, we usually take $\mathcal{L}=\IC(P)$, but we also compare to prior results on order ideals, in which case we use $\mathcal{L}=J(P)$.

\begin{definition}[\protect{\cite[Section 3.4]{Striker2018}}] Let $x\in P$ and $\mathcal{L}$ any subset of the power set $2^P$. 
Define the \emph{toggle} $t_x:\mathcal{L}\rightarrow\mathcal{L}$ as:
\[t_x(I) = 
\begin{cases} 
I \setminus \{x\} &\text{if } x\in I \text{ and } I \setminus\{x\} \in\mathcal{L},\\
I \cup \{x\} &\text{if } x\notin I \text{ and } I \cup\{x\} \in\mathcal{L},\\
I &\text{otherwise.}
\end{cases}\]
\end{definition}

The following states the condition in which order ideal toggles commute.
\begin{lem}[\protect{\cite[Sec.\ 2]{CF1995}}] \label{lem:oi_commute}
Given $x,y\in P$ and $\mathcal{L}=J(P)$,
$(t_x t_y)^2 = 1$ if and only if $x$ and $y$ do not share a covering relation in $P$.
\end{lem}

Below is the analogous lemma for ICS. Note that the condition under which toggles commute is weaker.
\begin{lem}[\protect{\cite[Lemma 3.17]{Striker2018}}] \label{lem:commute}
Given $x,y\in P$ and $\mathcal{L}=\IC(P)$,
$(t_x t_y)^2 = 1$ if and only if $x$ and $y$ are
incomparable in $P$ or (in the case $x < y$) if $y$ covers $x$ where $y$ is maximal and $x$ is minimal.
\end{lem}

Toggles may be composed to create interesting actions. One well-studied action on order ideals is the composition of toggles from top to bottom; this action was named rowmotion~\cite{SW2012}. In \cite{ELMSW}, this name was used for the analogous composition on interval-closed sets. We define this composition for any subset $\mathcal{L}$ of the power set $2^P$, noting that the case $\mathcal{L}=J(P)$ is \cite[Lemma 1]{CF1995} and $\mathcal{L}=\IC(P)$ is from \cite[Section 2.3]{ELMSW}.

\begin{definition}\label{def:Row_tog}
Given $\mathcal{L}$ any subset of the power set $2^P$ and $I\in\mathcal{L}$, the \emph{rowmotion} of $I$, $\row(I)$, is given by applying all toggles in the reverse order of any linear extension of $P$.
\end{definition}

In the case of order ideals, the action now called rowmotion first appeared in the literature as a global action~\cite{Duchet1974};  subsequently, \cite[Lemma 1]{CF1995} proved the equivalence of the toggling definition, given above. This original global definition of order ideal rowmotion is below, in our notation, and stated as a lemma, since we took the toggling characterization as our definition.
\begin{lem}[\cite{Duchet1974,CF1995}]
\label{eq:global_JP}
For $I\in J(P)$, $\row(I)=\oi(\ceil{I})$.
\end{lem}
Our first main result, Theorem \ref{thm:AltRow_new} (a simplification of \cite[Thm.\ 2.20]{ELMSW}), gives the analogous statement for interval-closed sets.

The order of order ideal rowmotion on various families of posets turns out to be nice. See \protect{\cite[Sec. 3.1 and 6.1]{SW2012}} for history of the result given below.

\begin{thm}[\cite{Brouwer1975,Stanley2009,SW2012}]
\label{thm:oi_prod_chains_row}
    The order of rowmotion on $J([m]\times [n])$ divides $m+n$.
\end{thm}
Our second main result, Theorem~\ref{thm:thmorbits}, finds the order of rowmotion on $\IC([2]\times[n])$ via an extensive analysis of the orbits. Such an analysis seems intractable for the case of general $m$.

Rowmotion on $J([m]\times[n])$ was found to exhibit another dynamical phenomenon called homomesy.

\begin{definition}[\cite{PR2015}]\label{def:homomesy}
We say that a statistic $f$ on a set $X$ exhibits \emph{homomesy} under an action $\varphi:X\rightarrow X$ when every orbit $\mathcal{O}$ of $\varphi$ has the same average value of the statistic $f$. That is, 
\[\frac{1}{|\mathcal{O}|}\sum_{I \in \mathcal{O}} f(I) =c\]
for all orbits $\mathcal{O}$ of $\varphi$. We say that the triple $(X,\varphi,f)$ is $c$\emph{-mesic}.
\end{definition}

For $I\subseteq P$, the cardinality statistic of $I$ is $|I|$, the number of elements of $I$.
\begin{thm}[\protect{\cite[Thm.\ 23]{PR2015}}]
    The cardinality statistic on $J([m]\times[n])$ is $\frac{mn}{2}$-mesic with respect to rowmotion.
\end{thm}

It was noted in \cite[Remark 3.26]{DHPP23} that the signed cardinality statistic, defined below, also exhibits homomesy on order ideals of $[m]\times[n]$ with respect to rowmotion. Our third main result, Theorem~\ref{thm:homomesy}, is an analogous homomesy result on $\IC([2]\times[n])$. 

\begin{definition}\label{def:sc}
For each $x\in P$, define the \emph{signed cardinality statistic} $\SC(x): P \rightarrow \{-1, 1\}$  as:
\[\SC(x) =
\begin{cases}
1 \textrm{ if } \rank(x) \textrm{ is even, } \\
-1 \textrm{ if } \rank(x) \textrm{ is odd. }
\end{cases}\]
For $I\subseteq P$, let $\SC(I) = \sum_{x\in I}\SC(x)$.
\end{definition}

In the above paragraphs, we  stated relevant theorems about rowmotion on order ideals and discussed how the results of this paper are analogous. As one may note by comparing the statements and the proofs, the results on interval-closed sets are much more complicated to state and to prove. So rather than jumping to these analogous theorems, in \cite{ELMSW}, the authors proved  theorems about ICS rowmotion on a simpler family of posets: ordinal sums of antichains.

\begin{thm}[\protect{\cite[Thm.\ 3.7]{ELMSW}}]
    The order of rowmotion on $\IC(\mathbf{a}_1\oplus\mathbf{a}_2\oplus\cdots\oplus\mathbf{a}_n)$ divides $2n(n + 2)$ when $n$ is odd
and $n(n + 2)/2$ when $n$ is even.
\end{thm}

For these posets, the global characterization of rowmotion simplified.

\begin{thm}
[\protect{\cite[Thm.\ 3.14]{ELMSW}}]\label{alt_def_ordinal}
Rowmotion on $I \in \IC(\mathbf{a}_1\oplus\mathbf{a}_2\oplus\cdots\oplus\mathbf{a}_n)$ is given as
\[\row(I) =
\begin{cases}
\overline{I} &\text{if } I=\emptyset, \ \mathbf{a}_n\subseteq I, \text{ or }  I\subseteq \mathbf{a}_n , \\
\oi\ceil{I}\setminus\oi\Min(I) &\text{otherwise}.
\end{cases}\]
\end{thm}

They also gave a homomesy result with the signed cardinality statistic on ordinal sums of antichains of the same size. 
\begin{thm}[\protect{\cite[Thm.\ 3.21]{ELMSW}}]
\label{thm:signed_card_ord_sum}
The signed cardinality statistic exhibits homomesy with respect to rowmotion on the poset $\bigoplus_{i=1}^n \mathbf{m}$ where $n$ is even.
\end{thm}

They found another homomesy and conjectured the analogous statement for $[m]\times[n]$.
\begin{thm}[\protect{\cite[Thm.\ 4.7]{ELMSW}}]
\label{thm:homomesy_2_by_n_max-min}
The number of maximal elements minus the number of minimal elements is $0$-mesic for rowmotion on $\IC([2] \times [n])$.
\end{thm}

Finally, they gave the following conjecture. Our Theorem~\ref{thm:homomesy} is the proof of the case $m=2$.
\begin{conj}[\protect{\cite[Conj.\ 4.12]{ELMSW}}]
\label{conj:signed_card} If $m = 2$ or $m = 3,$ then the signed cardinality
statistic is $0$-mesic under rowmotion on interval-closed sets of $[m]\times[n]$ whenever $m + n - 1$ is
even. 
\end{conj}

The paper is organized as follows. Section~\ref{sec:global} proves Theorem~\ref{thm:AltRow_new}, giving a simpler global characterization of rowmotion on interval-closed sets. Section~\ref{sec:2xn_orb_hom} proves Theorems~\ref{thm:thmorbits} and \ref{thm:homomesy} via a detailed analysis of the orbits of rowmotion on $\IC([2]\times[n])$. Section~\ref{sec:conj} discusses some directions for future research.

\section{A simpler global definition of rowmotion}
\label{sec:global}
In \cite{ELMSW}, rowmotion on interval-closed sets was defined in two ways. It was first defined as a composition of \textit{toggles} through which we add or remove individual elements in an interval-closed set, i.e.,~the \textit{local} definition. Rowmotion was then characterized \emph{globally} (\cite[Theorem 2.20]{ELMSW}) and it was shown that the two definitions are equivalent. The global definition sped up computations significantly, allowing the authors to search more effectively for dynamical phenomena.  We recall this result now; see Definition~\ref{def:poset_stuff} for necessary notations.

\begin{thm}[\protect{\cite[Theorem 2.20]{ELMSW}}]\label{thm:AltRow}
Given an interval-closed set $I\in\IC(P)$, rowmotion on $I$ is given by
\[
\row(I)=\arow.
\]
\end{thm}

In this section, we prove Theorem~\ref{thm:AltRow_new}, giving a simpler characterization of the global definition of rowmotion, presented as the union of two sets rather than three. 

\begin{thmAltRow_new}
Given an interval-closed set $I\in\IC(P)$, rowmotion on $I$ is given by
    \[ \row(I) = \Big( \inc(\ceil{I}) - I \Big)\cup\Big(\oi\ceil{I}-\oi(\F) \Big). \]
\end{thmAltRow_new}

\begin{proof}
    Recall from Theorem~\ref{thm:AltRow} above that 
    \[\row(I)=\arow.\]
    So we need to show the following (where we have labeled the sets in each union for brevity): 
    \begin{align*}
    \underbrace{\inc(I)}_{A}\cup&\underbrace{\Big(\oi\inc_{I}\big(\ceil{I}\big) -\big(I\cup\oi\ceil{I}\big)\Big)}_{B}\cup\underbrace{\Big(\oi\ceil{I}-\oi(\F) \Big)}_{C}\\
        &= \underbrace{\Big( \inc(\ceil{I}) - I \Big)}_{D}\cup\underbrace{\Big(\oi\ceil{I}-\oi(\F) \Big)}_{C}.
    \end{align*}
    We proceed by double inclusion. Since the last set of each union is $C$, we only need to show that $A\cup B \subseteq D \cup C$ and that $D \subseteq A \cup B \cup C$. We show below that  $D \subseteq A \cup B$, $A \subseteq D \cup C$ and that $B \subseteq D$. 
    
    We consider $x \in \inc(\ceil{I})-I$ (so $x \in D$). Then, $x$ is incomparable with the ceiling of $I$, and not in $I$. So, $x$ cannot be in the order filter spanned by $I$, so it is either incomparable with $I$, or in its order ideal (but not in $I$). In the first case, $x \in A$. In the second case, $x$ is in the order ideal of $I$ but not in $I$, and incomparable with the ceiling. Hence, $x \in B$.

    Now let $x$ be in $\inc(I)$. Because $x$ is incomparable with $I$, it cannot be in the order filter of $I$, so it is not in the ceiling of $I$ nor above. It can be either incomparable with the ceiling of $I$ (so in $D$), or below the ceiling of $I$ but not in the order ideal of $I$, so $x\in C$.

    The last case to consider is when $x$ is in $\Big(\oi\inc_{I}\big(\ceil{I}\big) -\big(I\cup\oi\ceil{I}\big)\Big)$ (so in $B$). We need to show that $x \in D\cup C$, and it suffices to show that $x \in D$, so this is what we do. In that case, $x$ is below an element of $I$ that is incomparable with the ceiling, but $x$ itself is neither in $I$ nor below the ceiling. Therefore, $x$ is incomparable with the ceiling and also not in $I$, so $x \in D$.
\end{proof}

\section{Rowmotion orbits and signed cardinality homomesy of interval-closed sets of $[2]\times[n]$}
\label{sec:2xn_orb_hom}
Our goals in this section are to prove Theorem~\ref{thm:thmorbits}, characterizing the rowmotion orbits on interval-closed sets of $[2]\times[n]$, and Theorem~\ref{thm:homomesy}, a homomesy result on these orbits with respect to the signed cardinality statistic, when $n$ is odd.

We begin in Subsection~\ref{sec:catalog} by giving a classification of the various types of ICS (Lemma~\ref{lem:types}) and showing the effect of rowmotion on each (Lemma~\ref{Lem: CatalogRowmotion2xn}).  In Subsection~\ref{ssec:calc_sc}, we discuss some basic properties of signed cardinality for these ICS. We then analyze orbits of size dividing $n+3$ in Subsection \ref{sec:n+3}, proving Theorem~\ref{thm:n+3} and Corollary~\ref{cor_n+3/2}, enumerating orbits of sizes dividing $n+3$, and Theorem~\ref{thm:sc_n+3}  and Corollary~\ref{cor:sc_n+3}, showing the average signed cardinality over each orbit is $0$. We study
orbits of size dividing $n+5$ in Subsection \ref{sec:n + 5}, proving Theorems~\ref{thm:n+5} and \ref{thm:sc_n+5} on orbit enumeration and homomesy, respectively.
We study all other orbits in Subsection \ref{sec:large}, proving Theorems \ref{large_orbits} and \ref{lem:sc-big-orbits} on orbit structure and homomesy. Subsection \ref{ssec:all_together} shows that all the orbits were listed in the previous subsections, allowing us to describe precisely the orbit structure and complete the proofs of Theorems~\ref{thm:thmorbits} and \ref{thm:homomesy}.

\subsection{Cataloging types of interval-closed sets of $[2]\times[n]$ and the effect of rowmotion on each}
\label{sec:catalog}

In this subsection, we describe all the ICS of $[2]
\times[n]$ and their image under rowmotion. We classify those into six types. To do so, we start by introducing notation.

The elements of the poset $[2] \times [n]$ may be divided into two subposets, $\{1\} \times [n]$ and $\{2\} \times [n]$, which we call the \emph{lower} and \emph{upper chain}, respectively. We parameterize each $I\in \IC([2] \times [n])$ by six numbers that record how many elements lie below, in, and above $I$ on these two chains.  Precisely, we denote the ICS $I$ by the tuple
 \[
[b_1,i_1,a_1:b_2,i_2,a_2] 
 \]
 where the parameters are given as follows:
 \begin{itemize}
     \item $i_1$ and $i_2$ are the number of elements of $I$ that belong to the lower and upper chain, respectively,
     \item $a_1$ and $a_2$ are the number of elements in the lower and upper chains, respectively, above $I$,
     \item $b_1$ and $b_2$ are the number of elements in the lower and upper chains, respectively,  below $I$, and
     \item we take the convention that if $I$ does not intersect the lower chain then $b_1 = n$ and $a_1 = 0$, while if $I$ does not intersect the upper chain then $b_2 = 0$ and $a_2 = n$.
 \end{itemize}
 For example, the full ICS $I = [2] \times [n]$ is parametrized by $[0, n, 0 : 0, n, 0]$ and the empty ICS $I = \emptyset$ by $[n, 0, 0 : 0, 0, n]$; a more generic ICS is illustrated in Figure~\ref{fig:Iparameterization}.

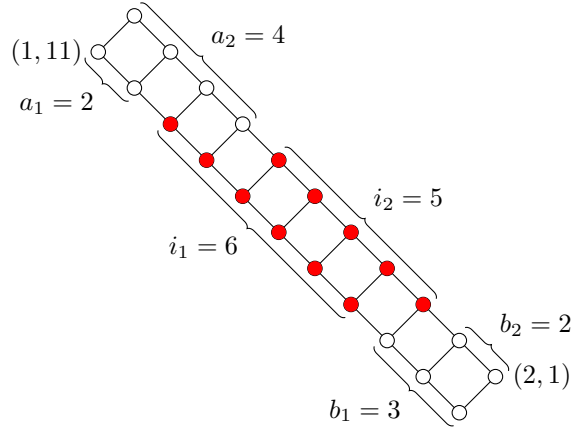
\begin{figure}[htbp]
    \centering
 \begin{tikzpicture}[scale=.48]
\foreach \x in {0,...,1}
	{\foreach \y in {0,...,10}
		{\ifthenelse{\x < 1}
			{\draw (\x - \y, \x + \y) -- (\x - \y + 1, \x + \y + 1);}{}
		\ifthenelse{\y < 10}
			{\draw (\x - \y, \x + \y) -- (\x - \y - 1, \x + \y+1);}{}
        \draw[fill=white] (\x - \y, \x + \y) circle (0.2cm) {};
		}
	}
\fill[red] (0 - 8, 0 + 8) circle (0.2cm) {};
\fill[red] (0 - 7, 0 + 7) circle (0.2cm) {};
\fill[red] (0 - 6, 0 + 6) circle (0.2cm) {};
\fill[red] (0 - 5, 0 + 5) circle (0.2cm) {};
\fill[red] (0 - 4, 0 + 4) circle (0.2cm) {};
\fill[red] (0 - 3, 0 + 3) circle (0.2cm) {};
\fill[red] (1 - 6, 1 + 6) circle (0.2cm) {};
\fill[red] (1 - 5, 1 + 5) circle (0.2cm) {};
\fill[red] (1 - 4, 1 + 4) circle (0.2cm) {};
\fill[red] (1 - 3, 1 + 3) circle (0.2cm) {};
\fill[red] (1 - 2, 1 + 2) circle (0.2cm) {};
\draw (0 - 10, 0 + 10) node[left=.25em] {$(1, 11)$};
\draw (1 - 0, 1 + 0) node[right=.3em] {$(2, 1)$};
\draw[decoration={brace, raise=.5em},decorate]
  (0 + .1, 0 - .1) -- node[below left=.5em] {$b_1 = 3$} (0 - 2.1, 0 + 2.1);
\draw[decoration={brace, raise=.5em},decorate]
  (0 - 2.9, 0 + 2.9) -- node[below left=.5em] {$i_1 = 6$} (0 - 8.1, 0 + 8.1);
\draw[decoration={brace, raise=.5em},decorate]
  (0 - 8.9, 0 + 8.9) -- node[below left=.5em] {$a_1 = 2$} (0 - 10.1, 0 + 10.1);
\draw[decoration={brace, raise=.5em, mirror},decorate]
  (1 + .1, 1  - .1) -- node[above right=.5em] {$b_2 = 2$} (1 - 1.1, 1 + 1.1);
\draw[decoration={brace, raise=.5em, mirror},decorate]
  (1 - 1.9, 1  + 1.9) -- node[above right=.5em] {$i_2 = 5$} (1 - 6.1, 1 + 6.1);
\draw[decoration={brace, raise=.5em, mirror},decorate]
  (1 - 6.9, 1  + 6.9) -- node[above right=.5em] {$a_2 = 4$} (1 - 10.1, 1 + 10.1);
\end{tikzpicture}   
    \caption{The ICS $I \subseteq [2] \times [11]$ parametrized by $[b_1, i_1, a_1 : b_2, i_2, a_2] = [3, 6, 2 : 2, 5, 4]$}
    \label{fig:Iparameterization}
\end{figure}

When $i_1$ and $i_2$ are nonzero, we have that $[b_1, i_1, a_1 : b_2, i_2, a_2]$ represents an interval-closed set of $[2] \times [n]$  if and only if the six entries are nonnegative integers such that $b_1 + i_1 + a_1 = b_2 + i_2 + a_2 = n$, $b_1  \geq b_2$, and $a_2 \geq a_1.$

In the case of ICS that belong entirely to either the upper or lower chain, we will abbreviate this notation further and write $[b, i, a : \varnothing]$ instead of $[b, i, a : 0, 0, n]$ and likewise $[\varnothing : b, i, a]$ instead of $[n, 0, 0 : b, i, a]$.

In what follows, it will be convenient to partition interval-closed sets that are neither the empty set nor the whole poset into six families, which we name now and illustrate in Tables~\ref{tab:lemma-RowEffectsPoints1} and~\ref{tab:lemma-RowEffectsPoints2}. 
For Types 1 to 4, we assume that $I$ intersects both the upper and lower chain (i.e., that $i_1, i_2 > 0$).
\begin{enumerate}[Type 1,]
    \item Double Hook: $b_2<b_1< b_2+i_2<b_1+i_1,$ 
    \item Disjoint:  $b_2+i_2\leq b_1,$ 
    \item First Hook: $b_1=b_2$ and $a_1<a_2,$
    \item Stacked Diagonals and Second Hooks: $a_1=a_2$ and $b_2 \leq b_1$,
    \item Low: $a_2 = n,$
    \item High: $b_1 = n.$
\end{enumerate}

\begin{lem}\label{lem:types}
    Any interval-closed set that is neither the empty set nor the full poset is in exactly one of the types $1$ through $6$.
\end{lem}
\begin{proof}
    Let $I$ be an ICS that is neither the empty set nor the full poset. If $I$ intersects only the lower chain, then it is of Type 5, Low. If it intersects only the upper chain, then it is of Type 6, High.
    Otherwise, $I = [b_1, i_1, a_1: b_2, i_2, a_2]$ with both $i_1$ and $i_2$ positive. There are different cases. Recall that for a subset of $P$ to be interval-closed, we need that $a_1 \leq a_2$ and $b_1 \geq b_2$. Here are all the cases:
    \begin{itemize}
        \item If $b_1 = b_2$:
        \begin{itemize}
            \item if $a_1 < a_2$, then $I$ is a First Hook.
            \item otherwise, $a_1 = a_2$, and $I$ is an interval-closed set of type Stacked Diagonals.
        \end{itemize}
        \item Otherwise, $b_1 > b_2$:
        \begin{itemize}
            \item If $a_1 = a_2$, then $I$ is a Second Hook.
            \item Otherwise, $ a_1 < a_2$, and there are two cases:
            \begin{itemize}
                \item If $b_1 \geq b_2 +i_2$, then $I$ is Disjoint.
                \item Otherwise, $b_1 < b_2 + i_2$, and $b_2 < b_1 <  b_2 +i_2 = n-a_2 < n-a_1 = b_1 + i_1$. Therefore, I is a Double Hook. \qedhere
            \end{itemize}
        \end{itemize}
    \end{itemize}
\end{proof}

\begin{lem} 
\label{Lem: CatalogRowmotion2xn}
The effects of rowmotion on each of the 
six configuration types are as shown in Tables~\ref{tab:lemma-RowEffectsPoints1} and~\ref{tab:lemma-RowEffectsPoints2}. 
\end{lem}
\begin{proof}
    This is a direct application of the Definition~\ref{def:Row_tog}, or of Theorem~\ref{thm:AltRow_new}.
\end{proof}

\begin{table}[!htbp]
    \centering
    \begin{tabular}{|p{3.5cm}|p{5.75cm}|p{5.75cm}|}\hline
        $[b_1,i_1, a_1:b_2,i_2,a_2]$  & 
        $a_1\not=0$ and $a_2\not=0$       
       &

       $a_1=0$ \\\hline\hline
       Type 1: Double Hook \newline

       $b_2<b_1<b_2+i_2<b_1+i_1$  
       
       &  Type 1a: $\row(I)=$

       $[b_1+1,i_1,a_1-1:b_2+1,i_2,a_2-1]$

       (Double Hook shifted up)
         \medskip
         \centering{\includegraphics[width=4cm]{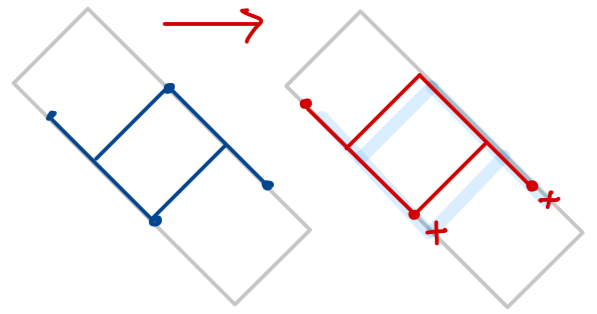}}

       & Type 1b:               
       $\row(I)=$
       
       $[b_1+1, i_1-a_2, a_2-1:b_2+1,i_2,a_2-1]$
       
       (Second Hook)
       
            \medskip   
       \centerline{\includegraphics[width=4cm]{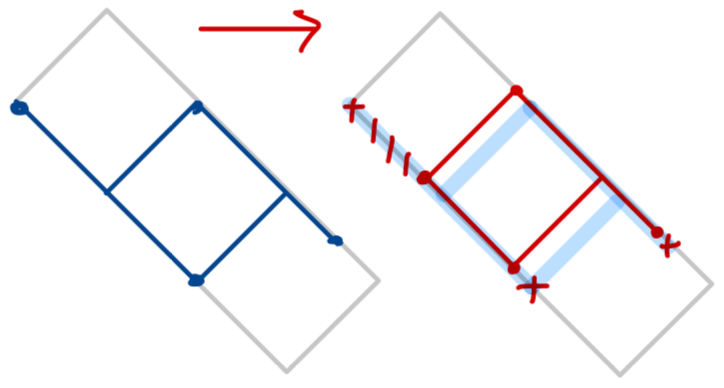}}
        \\\hline

     Type 2: Disjoint  \newline 
        
        $b_2+i_2\leq b_1$

              &  Type 2a:
      $\row(I)=$\newline
      $[b_1+1,i_1,a_1-1:b_2+1,i_2,a_2-1]$ 
      \newline
      (Disjoint shifted up)

        \centering{\includegraphics[width=4cm]{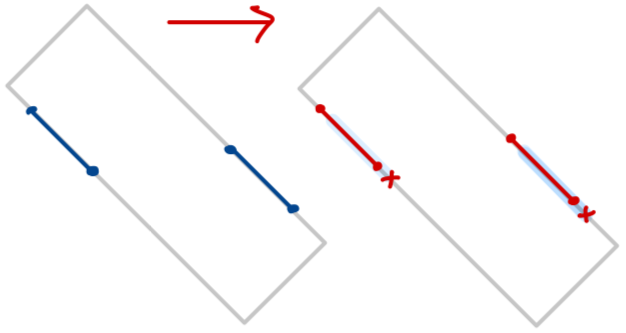}}
         
      & Type 2b: If $a_1=0$ and $ i_1<a_2$ in $I$ then
      $\row(I)=$\newline
      $[b_2+1,b_1-(b_2+1),i_1:b_2+1,i_2,a_2-1]$
      \newline
      (First Hook, or Stacked Diagonals if in $I$, $a_2=i_1+1$)
    \centerline{\includegraphics[width=4cm]{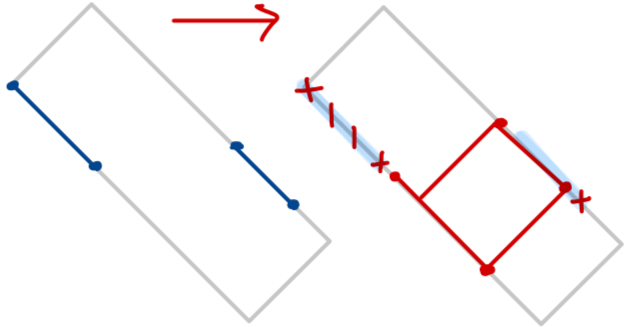}}  
    
    Type 2c: If $a_1=0$ and $i_1=a_2$ then  $\row(I)=$
 $[\varnothing:b_2+1,i_2,a_2-1]$\newline
 (High)

        \centerline{\includegraphics[width=4cm]{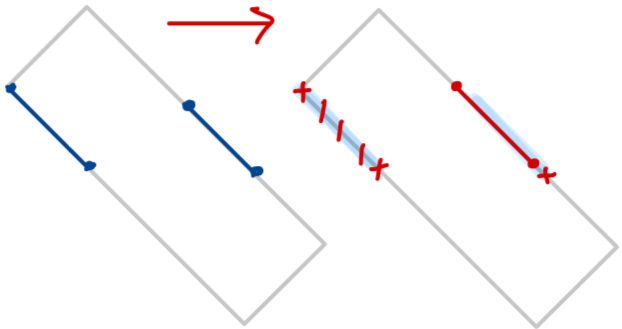}}    
   
    \\\hline

        Type 3: First Hook

        $b_1=b_2$ and 
        $a_1<a_2$

        & Type 3a: 
        $\row(I)=$
        \newline
        $[b_1+1,i_1,a_1-1:0,i_2+b_2+1,a_2-1]$

        (Double Hook)
        \medskip

        \centerline{\includegraphics[width=4cm]{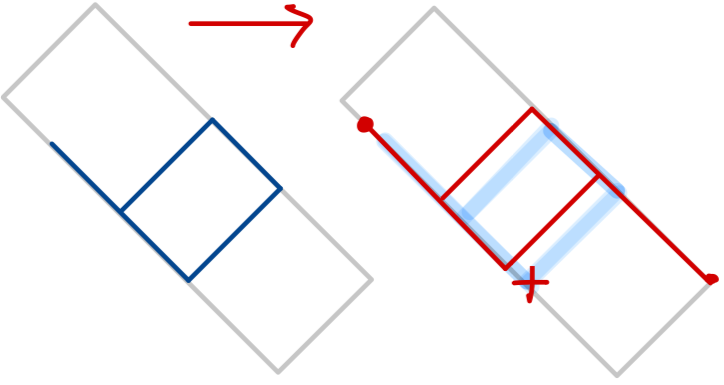}}

   & Type 3b:    $\row(I)=$
   \newline
   $[b_1+1,i_2,a_2-1:0,i_2+b_2+1,a_2-1]$

   (Second Hook)

   \medskip

        \centerline{\includegraphics[width=4cm]{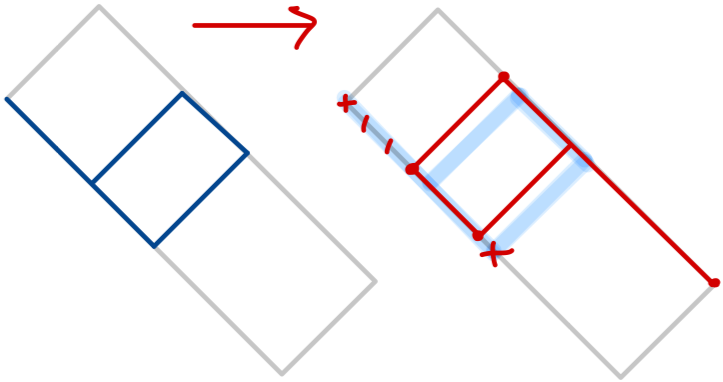}}

\\\hline

\end{tabular}
\caption{The effects of rowmotion on interval-closed sets of $[2]\times[n]$}
    \label{tab:lemma-RowEffectsPoints1}
\end{table}

\begin{small}
\begin{table}[p]
    \centering
    \begin{tabular}{|p{3.5cm}|p{5.75cm}|p{5.75cm}|}\hline
        $[b_1,i_1, a_1:b_2,i_2,a_2]$  & 
        $a_1\not=0$ and $a_2\not=0$     
       & 
       
       $a_1=0$ \\\hline\hline
        Type 4: Stacked Diagonal / Second Hook

        $a_1=a_2$ and $b_2 \leq b_1$

        & Type 4a: $\row(I)=$
        
        $[b_1+1,i_1,a_1-1:0,b_2,n-b_2]$
        
                (Disjoint)

                \medskip
        \centerline{\includegraphics[width=4cm]{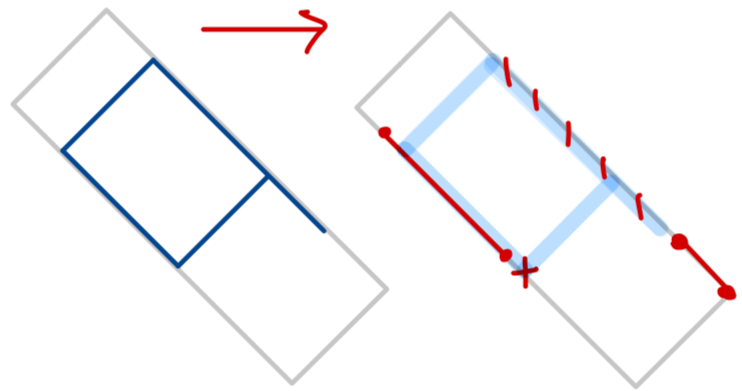}}

    Type 4b: If $b_2=0$ in $I$ then 
    
    $\row(I)=$
    $[b_1+1,i_1,a_1-1:\varnothing]$

    (Low)
    
    \medskip
        \centerline{\includegraphics[width=4cm]{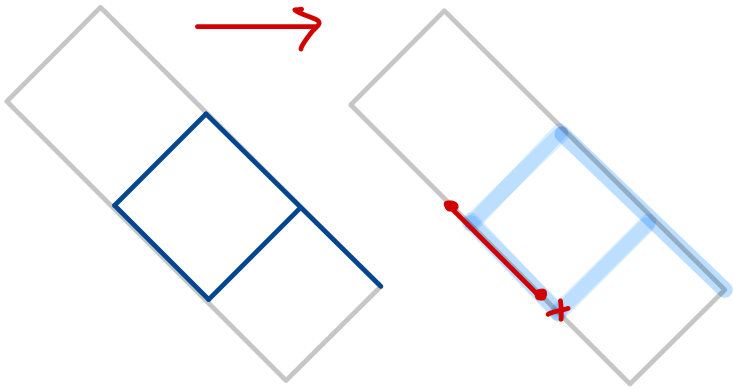}}

    & Type 4c: If $a_2\neq 0$,

    $\row(I)=[0,b_1,i_1:0,b_2,i_2]$
    
    (The complement; Stacked Diagonal or First Hook)
    
\medskip
        \centerline{\includegraphics[width=4cm]{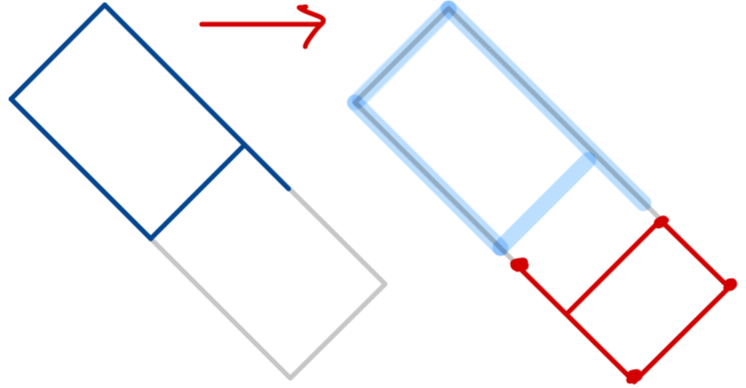}}
    
    Type 4d: If $b_1>b_2=0$,

    $\row(I)=[0,b_1,i_1:\varnothing]$
    
    (Low)
    
\medskip
        \centerline{\includegraphics[width=4cm]{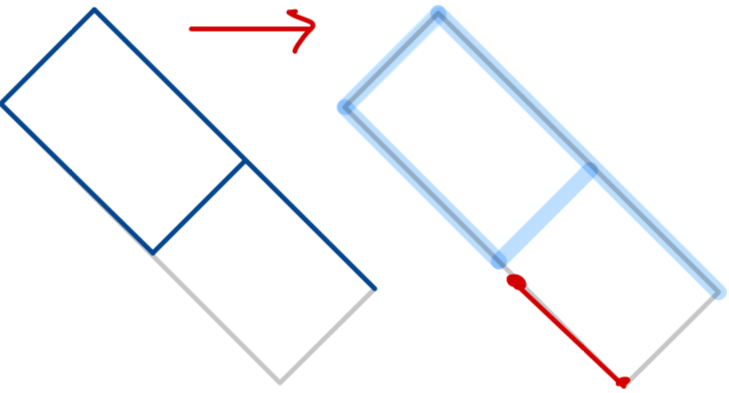}}
    
\\\hline
        
        Type 5: Low

         $a_2=n$ 
        
        & Type 5a: $\row(I)=$
         
         $[b_1+1,i_1,a_1-1: 0,b_1+1,n-b_1-1]$

         (Disjoint)

         \medskip

         \centerline{\includegraphics[width=4cm]{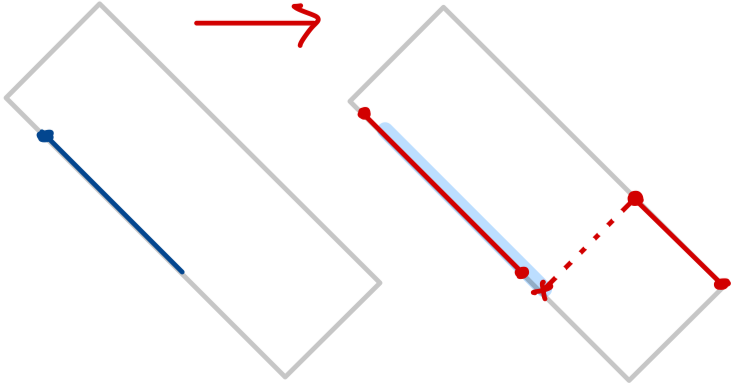}}

    & Type 5b: $\row(I)=$

    $[\varnothing: 0,b_1+1,i_1-1]$

    (High)

    \medskip
    \centerline{\includegraphics[width=4cm]{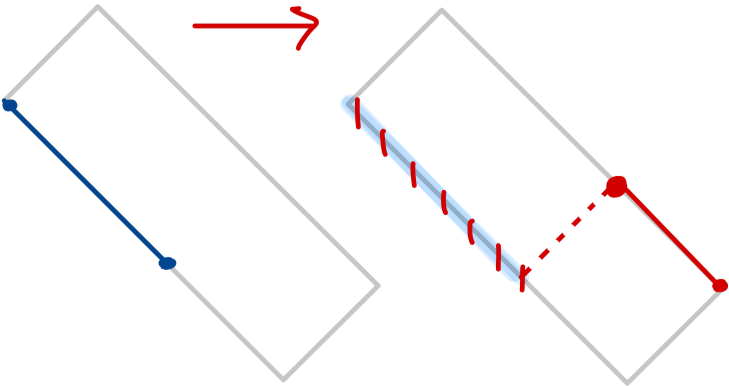}}

    \\\hline\hline
    
        $[b_1,i_1, a_1:b_2,i_2,a_2]$  & 
      $a_2\not=0$     
       & 
       
       $a_2=0$ \\\hline\hline  

        Type 6: High

        $b_1=n$
        
        & Type 6a: 
    $\row(I)=$ 
    
    $[b_2+1,i_2+a_2-1,0:b_2+1,i_2,a_2-1]$

    (First Hook) 

    \medskip
 
    \centerline{\includegraphics[width=4cm]{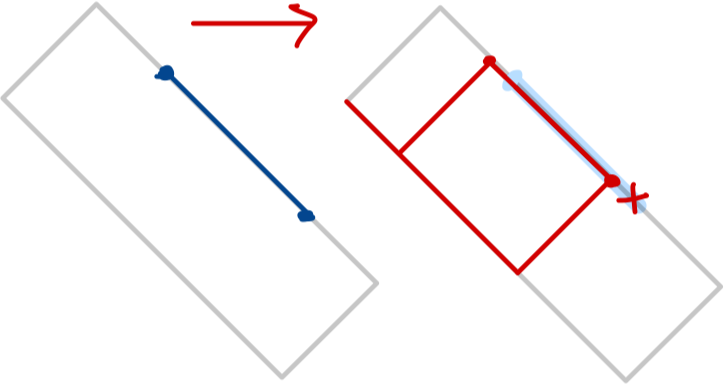}}
    
    Type 6b: If $a_2=1$ in $I$, then $\row(I)=$
        $[b_2+1,i_2,0:b_2+1,i_2,0]$

    (Stacked Diagonals)

    \medskip
     \centerline{\includegraphics[width=4cm]{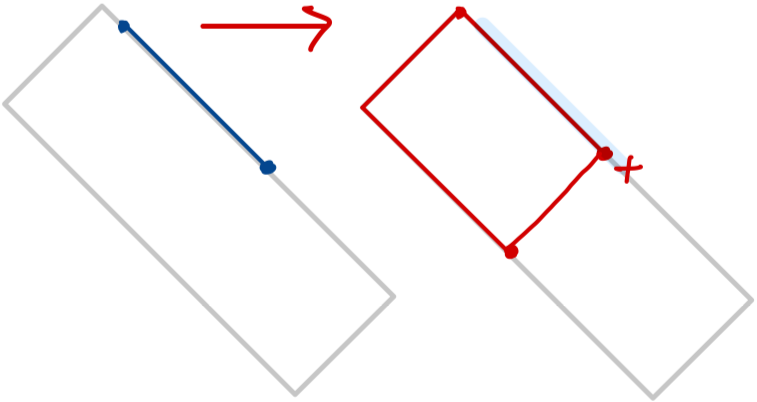}}

    &    
    Type 6c: If $b_2\neq 0$ in $I$, then  
    
        $\row(I)=[0,n,0:0,b_2,i_2]$

        (The complement; First Hook)

        \medskip
        \centerline{\includegraphics[width=4cm]{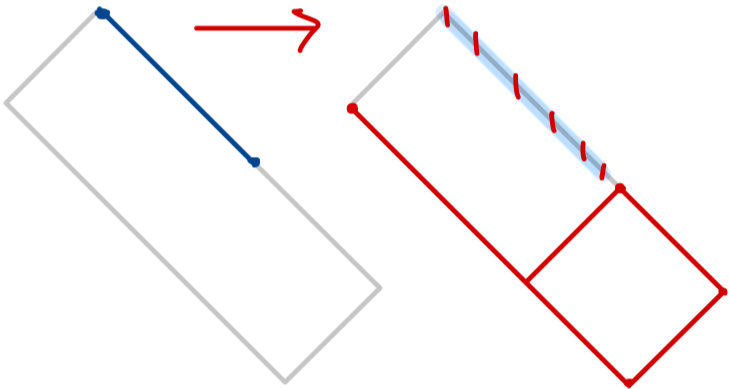}}

     Type 6d: If $b_2=0$, then

     $\row(I)=[0,n,0:\varnothing]$
     
     (The complement; Low) 
     
     \medskip
     \centerline{\includegraphics[width=4cm]{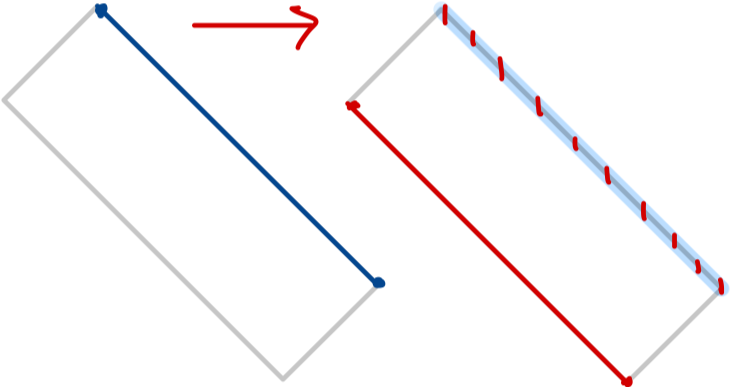}}

\\\hline
    \end{tabular}
    \normalsize
    \caption{The effects of rowmotion on interval-closed sets of $[2]\times[n]$}
    \label{tab:lemma-RowEffectsPoints2}
\end{table}
\end{small}
\normalsize

\subsection{Calculating the signed cardinality}\label{ssec:calc_sc}
Using the notation from Subsection \ref{sec:catalog}, we now compute the signed cardinality for the ICS of $[2]\times[n]$.
Recall from Definition \ref{def:sc}, that for all $x \in P$, the \emph{signed cardinality statistic} $\SC(x): P \rightarrow \{-1, 1\}$ is
\[\SC(x) =
\begin{cases}
1 & \textrm{if } \rank(x) \textrm{ is even, } \\
-1 & \textrm{if } \rank(x) \textrm{ is odd. }
\end{cases}\]
For $I\subseteq P$, let $\SC(I) = \sum_{x\in I}\SC(x)$.  
Using the \emph{parity indicator function}
\begin{equation}
\label{eq:parity delta}
\delta_x= \begin{cases} 
      1 & \text{if $x$ is odd,} \\
      0 & \text{otherwise}, 
      \end{cases}
\end{equation} we can directly calculate the signed cardinality for ICS in $P = [2] \times [n]$, as follows.

\begin{lem}\label{calc_sc(I)}
Let $P = [2] \times [n]$ with $n \geq 2$ and let $I = [b_1, i_1, a_1 : b_2, i_2, a_2]\in \IC(P)$. Then 
\begin{equation}
\SC(I) = \delta_{i_1}(-1)^{b_1}+ \delta_{i_2}(-1)^{b_2+1}.
\end{equation}
\end{lem}

\begin{proof}
The ICS $I := [b_1, i_1, a_1 : b_2, i_2, a_2]$ has entries on the lower diagonal in ranks $b_1, \ldots, b_1 + i_1 - 1$ and on the upper diagonal in ranks $b_2 + 1, \ldots, b_2 + i_2$.  Since $\sum_{k = r}^s (-1)^k = \frac{(-1)^r + (-1)^s}{2}$, it follows that the signed cardinality of $I$ is $\frac{(-1)^{b_1} + (-1)^{b_1 + i_1 - 1}}{2} + \frac{(-1)^{b_2 + 1} + (-1)^{b_2 + i_2}}{2}=\delta_{i_1}(-1)^{b_1}+ \delta_{i_2}(-1)^{b_2+1}.$ 
\end{proof}

Using Lemma~\ref{calc_sc(I)}, it is easy to prove the following propositions, which are integral in proving Theorem~\ref{thm:homomesy}. 

\begin{prop}\label{P_sc}
Let $P = [2] \times [n]$ with $n \geq 1$. Then $\SC(P) = 0$.
\end{prop}

\begin{prop}\label{comp_sc}
Let $P = [2] \times [n]$ with $n \geq 1$ and let $I \in \IC(P)$. Then $\SC(I) + \SC(\overline{I}) = 0$.
\end{prop}

\begin{prop}\label{shift_sc}
Let $P = [2] \times [n]$ with $n \geq 1$, let $I$ be an ICS in $P$ that does not contain the top element in either chain, and let $I'$ be the result of shifting $I$ up by one rank.  Then $\SC(I) + \SC(I') = 0$.
\end{prop}

\subsection{Orbits of size dividing $n + 3$}\label{sec:n+3}
In this section, we analyze orbits of size dividing $n+3$, proving the following result:

\begin{thm}\label{thm:n+3}

Let $P = [2] \times [n]$ with $n \geq 2$. The number of orbits of size $n + 3$ for rowmotion on $P = [2]\times [n]$ is at least

\begin{itemize}

\item $(n^3 - n^2 - 4n + 12)/12$ when $n \equiv 0 \pmod{6},$

\item $(n^3 - n^2 - 7n + 19)/12$ when $n \equiv 1 \pmod{6},$

\item $(n^3 - n^2 - 4n + 16)/12$ when $n \equiv 2 \pmod{6},$

\item $(n^3 - n^2 - 7n + 15)/12$ when $n \equiv 3 \pmod{6},$

\item $(n^3 - n^2 - 4n + 16)/12$ when $n \equiv 4 \pmod{6},$ and

\item $(n^3 - n^2 - 7n + 19)/12$ when $n \equiv 5 \pmod{6}.$

\end{itemize}
    
\end{thm}

Theorem \ref{thm:n+3} will be an immediate consequence of the following lemma, which categorizes the orbits of size $n+3$ into four cases. We use the terminology \emph{stacked ranks} to mean an ICS consisting of consecutive full ranks. This type of interval-closed set intersects each of the six types of ICS we defined in Section \ref{sec:catalog}. Specifically, if $I \in \IC(P)$ is an ICS in $P$ consisting of 2 or more stacked ranks,
then $I = [b_1, i_1, a_1 : b_2, i_2, a_2]$ 
has one of four forms:

\begin{itemize}
    \item If $I$ contains neither the maximal or minimal element of $P$, then $b_1 = b_2 + 1, 2 \leq i_1 = i_2, a_1 = a_2 - 1$. This falls under the category Double Hook. 
    \item If $I$ contains the maximal element of $P$ but not the minimal element, then $b_1 = b_2 + 1, 2 \leq i_1 = i_2 - 1, a_1 = a_2 = 0$. This falls under the category Second Hook.
    \item If $I$ contains both the maximal and minimal elements of $P$, then $b_1 = b_2 = 0 = a_1 = a_2, i_1 = i_2 = n $. This falls under the category Stacked Diagonals and is specifically when $I = P$.
    \item If $I$ contains the minimal element of $P$ but not the maximal element, then $b_1 = b_2 = 0, 2 \leq i_1 = i_2 + 1, a_1 = a_2 - 1$. This falls under the category First Hook.
\end{itemize}

If instead $I$ consists of only one stacked rank, then $I = [b_1, 1, a_1 : b_2, 1, a_2]$ 
has one of three forms:

\begin{itemize}
    \item If $I$ contains neither the maximal or minimal element of $P$, then $b_1 = b_2 + 1$, $i_1 = 1 = i_2$. This falls under the category Disjoint.
    \item If $I$ contains the maximal element of $P$ but not the minimal element, then $b_1 = n$. This falls under the category High.
    \item If $I$ contains the minimal element of $P$ but not the maximal element, then $a_2 = n$. This falls under the category Low.
\end{itemize}

\begin{lem}\label{n+3}
Let $P = [2] \times [n]$ with $n \geq 2$.  Consider the following ICS:
\begin{enumerate}
    \item Stacked ranks $[0, r, n - r : 0, r - 1, n + 1 - r]$ for $1 \leq r \leq n/2$;
    \item Stacked Diagonals of the form $I = [b, i, a : b, i, a]$ with $a, i, b > 0$, excepting the case $a = i = b = n/3$ when $n \equiv 0 \pmod{3}$;
    \item Second Hooks of the form $I = [b, i, 0  :  b - t, t + i, 0],$ with $i \geq 1, t \geq 2$, and $b - t \geq 1$; and
    \item Double Hooks of the form $[b, k + j, a :   0, b + k, a + j]$ where $j \geq 2$, $b \geq 2$, $k\geq 1$, and $a \geq 0$, such that either $j \neq b$ or $a + 1 \neq k$.
\end{enumerate}
Each of these ICS belongs to an orbit of size $n + 3$.  These orbits are distinct, with the following exceptions: in case 2, the ICS $[b, i, a : b, i, a]$, $[a, b, i : a, b, i]$, and $[i, a, b : i, a, b]$ all belong to the same orbit, and in case 4, the ICS $[b, k + j, a :   0, b + k, a + j]$ and $[j, a + b + 1, k - 1 :  0, j + a + 1, b + k - 1]$ belong to the same orbit.
The numbers of orbits containing the ICS in the four cases are respectively
\begin{enumerate}
    \item  $n/2$ when $n$ is even, $(n - 1)/2$ when $n$ is odd;
    \item  $(n^2 - 3n)/6$ when $n$ is divisible by $3$, $(n^2 - 3n + 2)/6$ otherwise;
    \item $(n^2 - 5n + 6)/2$; and
    \item $\frac{1}{2}\binom{n - 2}{3}$ ICS when $n$ is even, $\frac{1}{2}\left(\binom{n - 2}{3} - \frac{n - 3}{2}\right)$ when $n$ is odd. 
\end{enumerate}
\end{lem}

For examples of representative ICS in the four cases, see Figure \ref{n+3_reps}.


\begin{figure}[htbp]
\begin{center}
\begin{minipage}[c]{3cm}
\begin{tikzpicture}[scale = .4]
\draw [-, ultra thick] (-1,0) -- (-2,1);
\draw [-, ultra thick] (-2,1) -- (-3,2);
\draw [-, ultra thick] (-3,2) -- (-4,3);
\draw [-, ultra thick] (-4,3) -- (-5,4);
\draw [-, ultra thick] (-5,4) -- (-6,5);
\draw [-, ultra thick] (-3,0) -- (-4,1);
\draw [-, ultra thick] (-2,-1) -- (-3,0);
\draw [-, ultra thick] (-4,1) -- (-5,2);
\draw [-, ultra thick] (-5,2) -- (-6,3);
\draw [-, ultra thick] (-6,3) -- (-7,4);
\draw [-, ultra thick] (-7,4) -- (-8,5);
\draw [-, ultra thick] (-6,5) -- (-7,6);
\draw [-, ultra thick] (-1,0) -- (-2,-1);
\draw [-, ultra thick] (-2,1) -- (-3,0);
\draw [-, ultra thick] (-3,2) -- (-4,1);
\draw [-, ultra thick] (-4,3) -- (-5,2);
\draw [-, ultra thick] (-5,4) -- (-6,3);
\draw [-, ultra thick] (-6,5) -- (-7,4);
\draw [-, ultra thick] (-7,6) -- (-8,5);
\draw[fill=red, radius = .2] (-1,0) circle [radius = 0.2];
\draw[fill=red, radius = .2] (-2,1) circle [radius = 0.2];
\draw[fill=white, radius = .2] (-3,2) circle [radius = 0.2];
\draw[fill=white, radius = .2] (-4,3) circle [radius = 0.2];
\draw[fill=white, radius = .2] (-5,4) circle [radius = 0.2];
\draw[fill=white, radius = .2] (-6,5) circle [radius = 0.2];
\draw[fill=white, radius = .2] (-7,6) circle [radius = 0.2];
\draw[fill=red, radius = .2] (-2,-1) circle [radius = 0.2];
\draw[fill=red, radius = .2] (-3,0) circle [radius = 0.2];
\draw[fill=red, radius = .2] (-4,1) circle [radius = 0.2];
\draw[fill=white, radius = .2] (-5,2) circle [radius = 0.2];
\draw[fill=white, radius = .2] (-6,3) circle [radius = 0.2];
\draw[fill=white, radius = .2] (-7,4) circle [radius = 0.2];
\draw[fill=white, radius = .2] (-8,5) circle [radius = 0.2];
\draw[decoration={brace, raise=.5em,mirror},decorate]
  (-4, 1 + .1) -- node[left=1em] {$r = 3$} (-4, -1 - .1);
\end{tikzpicture}
  \end{minipage}
  \begin{minipage}[c]{4cm}
\begin{tikzpicture}[scale = .4]
\draw[decoration={brace, raise=.5em},decorate]
  (-2 + .1, -1 - .1) -- node[below left=.5em] {$b = 4$} (-5 +.1, 2 - .1);
\draw[decoration={brace, raise=.5em},decorate]
  (-6 + .1, 3 - .1) -- node[below left=.5em] {$i = 2$} (-7 +.1, 4 - .1);
\draw (-8 , 5 ) node[left=.5em]{$a = 1$};
\draw [-, ultra thick] (-1,0) -- (-2,1);
\draw [-, ultra thick] (-2,1) -- (-3,2);
\draw [-, ultra thick] (-3,2) -- (-4,3);
\draw [-, ultra thick] (-4,3) -- (-5,4);
\draw [-, ultra thick] (-5,4) -- (-6,5);
\draw [-, ultra thick] (-3,0) -- (-4,1);
\draw [-, ultra thick] (-2,-1) -- (-3,0);
\draw [-, ultra thick] (-4,1) -- (-5,2);
\draw [-, ultra thick] (-5,2) -- (-6,3);
\draw [-, ultra thick] (-6,3) -- (-7,4);
\draw [-, ultra thick] (-7,4) -- (-8,5);
\draw [-, ultra thick] (-6,5) -- (-7,6);
\draw [-, ultra thick] (-1,0) -- (-2,-1);
\draw [-, ultra thick] (-2,1) -- (-3,0);
\draw [-, ultra thick] (-3,2) -- (-4,1);
\draw [-, ultra thick] (-4,3) -- (-5,2);
\draw [-, ultra thick] (-5,4) -- (-6,3);
\draw [-, ultra thick] (-6,5) -- (-7,4);
\draw [-, ultra thick] (-7,6) -- (-8,5);
\draw[fill=white, radius = .2] (-1,0) circle [radius = 0.2];
\draw[fill=white, radius = .2] (-2,1) circle [radius = 0.2];
\draw[fill=white, radius = .2] (-3,2) circle [radius = 0.2];
\draw[fill=white, radius = .2] (-4,3) circle [radius = 0.2];
\draw[fill=red, radius = .2] (-5,4) circle [radius = 0.2];
\draw[fill=red, radius = .2] (-6,5) circle [radius = 0.2];
\draw[fill=white, radius = .2] (-7,6) circle [radius = 0.2];
\draw[fill=white, radius = .2] (-2,-1) circle [radius = 0.2];
\draw[fill=white, radius = .2] (-3,0) circle [radius = 0.2];
\draw[fill=white, radius = .2] (-4,1) circle [radius = 0.2];
\draw[fill=white, radius = .2] (-5,2) circle [radius = 0.2];
\draw[fill=red, radius = .2] (-6,3) circle [radius = 0.2];
\draw[fill=red, radius = .2] (-7,4) circle [radius = 0.2];
\draw[fill=white, radius = .2] (-8,5) circle [radius = 0.2];
\end{tikzpicture}
  \end{minipage}
\begin{minipage}[c]{4cm}
\begin{tikzpicture}[scale = .4]
\draw[decoration={brace, raise=.5em},decorate]
  (-2 + .1, -1 - .1) -- node[below left=.5em] {$b = 5$} (-6 +.1, 3 - .1);
\draw[decoration={brace, raise=.5em},decorate]
  (-7 + .1, 4 - .1) -- node[below left=.5em] {$i = 2$} (-8 +.1, 5 - .1);
\draw[decoration={brace, raise=.5em, mirror},decorate]
  (-3 + .1, 2 - .1) -- node[above right=.5em] {$t = 3$} (-5 +.1, 4 - .1);
\draw [-, ultra thick] (-1,0) -- (-2,1);
\draw [-, ultra thick] (-2,1) -- (-3,2);
\draw [-, ultra thick] (-3,2) -- (-4,3);
\draw [-, ultra thick] (-4,3) -- (-5,4);
\draw [-, ultra thick] (-5,4) -- (-6,5);
\draw [-, ultra thick] (-3,0) -- (-4,1);
\draw [-, ultra thick] (-2,-1) -- (-3,0);
\draw [-, ultra thick] (-4,1) -- (-5,2);
\draw [-, ultra thick] (-5,2) -- (-6,3);
\draw [-, ultra thick] (-6,3) -- (-7,4);
\draw [-, ultra thick] (-7,4) -- (-8,5);
\draw [-, ultra thick] (-6,5) -- (-7,6);
\draw [-, ultra thick] (-1,0) -- (-2,-1);
\draw [-, ultra thick] (-2,1) -- (-3,0);
\draw [-, ultra thick] (-3,2) -- (-4,1);
\draw [-, ultra thick] (-4,3) -- (-5,2);
\draw [-, ultra thick] (-5,4) -- (-6,3);
\draw [-, ultra thick] (-6,5) -- (-7,4);
\draw [-, ultra thick] (-7,6) -- (-8,5);
\draw[fill=white, radius = .2] (-1,0) circle [radius = 0.2];
\draw[fill=white, radius = .2] (-2,1) circle [radius = 0.2];
\draw[fill=red, radius = .2] (-3,2) circle [radius = 0.2];
\draw[fill=red, radius = .2] (-4,3) circle [radius = 0.2];
\draw[fill=red, radius = .2] (-5,4) circle [radius = 0.2];
\draw[fill=red, radius = .2] (-6,5) circle [radius = 0.2];
\draw[fill=red, radius = .2] (-7,6) circle [radius = 0.2];
\draw[fill=white, radius = .2] (-2,-1) circle [radius = 0.2];
\draw[fill=white, radius = .2] (-3,0) circle [radius = 0.2];
\draw[fill=white, radius = .2] (-4,1) circle [radius = 0.2];
\draw[fill=white, radius = .2] (-5,2) circle [radius = 0.2];
\draw[fill=white, radius = .2] (-6,3) circle [radius = 0.2];
\draw[fill=red, radius = .2] (-7,4) circle [radius = 0.2];
\draw[fill=red, radius = .2] (-8,5) circle [radius = 0.2];
\end{tikzpicture}
  \end{minipage}
  \begin{minipage}[c]{4cm}
\begin{tikzpicture}[scale = .4]
\draw[decoration={brace, raise=.5em},decorate]
  (-2 + .1, -1 - .1) -- node[below left=.5em] {$b = 2$} (-3 +.1, 0 - .1);
\draw[decoration={brace, raise=.5em},decorate]
  (-4 + .1, 1 - .1) -- node[below left=.5em] {$k = 2$} (-5 +.1, 2 - .1);
\draw[decoration={brace, raise=.5em},decorate]
  (-6 + .1, 3 - .1) -- node[below left=.5em] {$j = 2$} (-7 +.1, 4 - .1);
\draw (-8 , 5 ) node[left=.5em]{$a = 1$};
\draw [-, ultra thick] (-1,0) -- (-2,1);
\draw [-, ultra thick] (-2,1) -- (-3,2);
\draw [-, ultra thick] (-3,2) -- (-4,3);
\draw [-, ultra thick] (-4,3) -- (-5,4);
\draw [-, ultra thick] (-5,4) -- (-6,5);
\draw [-, ultra thick] (-3,0) -- (-4,1);
\draw [-, ultra thick] (-2,-1) -- (-3,0);
\draw [-, ultra thick] (-4,1) -- (-5,2);
\draw [-, ultra thick] (-5,2) -- (-6,3);
\draw [-, ultra thick] (-6,3) -- (-7,4);
\draw [-, ultra thick] (-7,4) -- (-8,5);
\draw [-, ultra thick] (-6,5) -- (-7,6);
\draw [-, ultra thick] (-1,0) -- (-2,-1);
\draw [-, ultra thick] (-2,1) -- (-3,0);
\draw [-, ultra thick] (-3,2) -- (-4,1);
\draw [-, ultra thick] (-4,3) -- (-5,2);
\draw [-, ultra thick] (-5,4) -- (-6,3);
\draw [-, ultra thick] (-6,5) -- (-7,4);
\draw [-, ultra thick] (-7,6) -- (-8,5);
\draw[fill=red, radius = .2] (-1,0) circle [radius = 0.2];
\draw[fill=red, radius = .2] (-2,1) circle [radius = 0.2];
\draw[fill=red, radius = .2] (-3,2) circle [radius = 0.2];
\draw[fill=red, radius = .2] (-4,3) circle [radius = 0.2];
\draw[fill=white, radius = .2] (-5,4) circle [radius = 0.2];
\draw[fill=white, radius = .2] (-6,5) circle [radius = 0.2];
\draw[fill=white, radius = .2] (-7,6) circle [radius = 0.2];
\draw[fill=white, radius = .2] (-2,-1) circle [radius = 0.2];
\draw[fill=white, radius = .2] (-3,0) circle [radius = 0.2];
\draw[fill=red, radius = .2] (-4,1) circle [radius = 0.2];
\draw[fill=red, radius = .2] (-5,2) circle [radius = 0.2];
\draw[fill=red, radius = .2] (-6,3) circle [radius = 0.2];
\draw[fill=red, radius = .2] (-7,4) circle [radius = 0.2];
\draw[fill=white, radius = .2] (-8,5) circle [radius = 0.2];
\end{tikzpicture}
  \end{minipage}
  \caption{The four types of representatives for orbits of size $n + 3$}
  \label{n+3_reps}
  \end{center}
\end{figure}
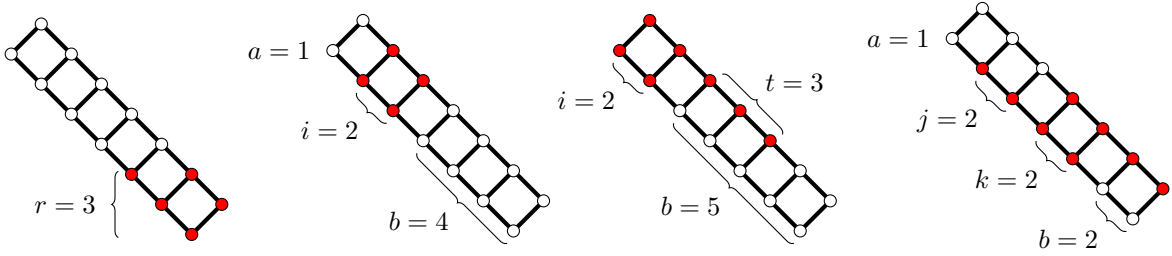


\begin{proof}[Proof of Lemma \ref{n+3}] Let $P = [2] \times [n]$.\\

\textit{Case One: Stacked Ranks.} 
For $r \leq n/2$, let $I = [0, r, n - r : 0, r - 1, n + 1 - r]$ be a stack of $r$ consecutive ranks, starting with rank $0$. First, we consider the case where $r > 1$. (This is illustrated in Example \ref{ex1_n+3} below.)  Then $I$ is a First Hook interval-closed set. Since necessarily $r < n$, by Lemma \ref{Lem: CatalogRowmotion2xn} Type 3a, a single application of rowmotion will return the Double Hook ICS $[1, r, n - r - 1 : 0, r, n - r]$. Then the next $n - r - 1$ applications of rowmotion belong to case Lemma \ref{Lem: CatalogRowmotion2xn} Type 1a, shifting the ICS up until it contains the maximal element of the lower chain. This requires results in the Double Hook interval-closed set $[n - r, r, 0 : n - r - 1, r, 1]$.

By Lemma \ref{Lem: CatalogRowmotion2xn} Type 1b, a single application of rowmotion will result in the Second Hook interval-closed set $[n - r + 1, r - 1, 0 : n - r, r, 0]$, which includes the maximum element of the poset. Thus, another single application of rowmotion returns the complement (Lemma \ref{Lem: CatalogRowmotion2xn} Type 4b), the First Hook interval-closed set $J = [0, n - r + 1, r-1 : 0, n - r, r]$. Altogether, this required $n - r + 2$ applications of rowmotion. 

If $J = I$, then $r = n - r + 1$, so $n$ must be odd and $r = \frac{n + 1}{2}$. However, this is excluded by our bounds on $r$.  (In this case the orbit of $I$ has size $\frac{n + 3}{2}$; it is recorded below in Corollary~\ref{cor_n+3/2}). Thus $J \neq I$.

Moreover, since $r > 1$, the same analysis that we applied to $I$ applies equally well to $J= [0, n - r + 1, r-1 : 0, n - r, r]$, and another $(r - 1) + 2 = r + 1$ applications of rowmotion will replace $J$ with a Double Hook, then shift this interval-closed set up until it reaches the top of the poset, and take the complement to get $[0, r, n - r: 0, r - 1, n - r + 1] = I.$ Obviously none of the intermediate steps in this process produced $I$, and therefore $I$ belongs to an orbit of size $(n - r + 2) + (r + 1) = n + 3$.  Moreover, we see that this orbit contains all interval-closed sets formed from either $r$ stacked ranks or $n - r + 1$ stacked ranks.

Next, consider the case when $r = 1 < n$. That is, $I$ contains only the minimum element of $P$ (the full rank $0$) and is a Low ICS.  By Lemma \ref{Lem: CatalogRowmotion2xn} Type 5a, applying rowmotion once returns the Disjoint ICS consisting of the full rank $1$. By Lemma \ref{Lem: CatalogRowmotion2xn} Type 2a, repeated application of rowmotion will shift this up one rank at a time until the ICS reaches the maximal element of the lower chain, and then by Lemma \ref{Lem: CatalogRowmotion2xn} Type 2c, a further application of rowmotion will return the High ICS consisting of only the maximum element (the full rank $n$). This requires a total of $n$ applications of rowmotion.  As this ICS contains the maximum element, a single application of rowmotion returns the complement, the full poset except for the maximum element (i.e., the stack of the $n$ bottom ranks in $P$), so it is a First Hook.  By Lemma \ref{Lem: CatalogRowmotion2xn} Type 3b, a single application of rowmotion returns the Second Hook ICS that includes every element of the poset except for the minimum element. As this contains the maximal element of the poset an application of rowmotion returns the complement, which is the Low ICS consisting of only the minimum element, which is our starting ICS.  Thus this orbit has size $n + 1 + 1 + 1 = n + 3$.
In total, the preceding analysis produces $\lfloor n/2 \rfloor$ distinct orbits of size $n + 3$.\\ 

\textit{Case Two: Stacked Diagonals.} In this section, to simplify notation, we denote by $[b, i, a]$ the ICS $I = [b_1, i_1, a_1 : b_2, i_2, a_2]$ with $a_1 = a_2 = a$, $i_1 = i_2 = i$, and $b_1 = b_2 = b$.  This case is illustrated in Example \ref{ex2_n+3}.

We claim if $n \equiv 0 \pmod{3}$ and $I = [n/3, n/3, n/3]$, then $I$ belongs to an orbit of size $\frac{n + 3}{3}.$ Otherwise, when $a, b, i > 0$, $I$ belongs to an orbit of size $n + 3$. 

Consider $I = [b, i, a]$ with $b, i, a > 0$. By Lemma \ref{Lem: CatalogRowmotion2xn} Type 4a, $\row(I)$ is the Disjoint interval-closed set $[b + 1, i, a-1: 0, b, i + a]$. By Lemma \ref{Lem: CatalogRowmotion2xn} Type 2a,
repeated applications of rowmotion shift the ICS up one rank at a time until the interval on the lower chain reaches the top of the poset. This requires $a-1$ applications and results in the Disjoint interval-closed set $[a+b, i, 0 : a - 1, b, i + 1].$
By Lemma \ref{Lem: CatalogRowmotion2xn} Type 2b, applying rowmotion 
once more will result in another Stacked Diagonal interval-closed set $[a, b, i].$ In total, this required $a + 1$ applications of rowmotion. If $[a, b, i] = I$, then $a = b = i = n/3$. Thus, $I = [n/3, n/3, n/3]$ belongs in an orbit of size $\frac{n + 3}{3}$. 

Otherwise, by a similar argument, an additional $i+1$ applications of rowmowtion will produce a $[i, a, b],$ and finally, an additional $b+1$ applications of rowmotion will return us to the original $[b, i, a]$. 
Thus, $[b, i, a], [a, b, i],$ and $[i, a, b]$ are all in the same orbit of size $a + i + b + 3 = n + 3$ when $a, b, i > 0$ and $[b, i, a] \neq [n/3, n/3, n/3]$, and these orbits contain no additional Stacked Diagonal ICS. As there are $\binom{n-1}{2}$ of these interval-closed sets when $n$ is congruent to $1$ or $2 \pmod3$ and $\binom{n-1}{2} - 1$ of them when $n$ is congruent to $0 \pmod3$, we have $\frac{\binom{n-1}{2}}{3} = \frac{n^2 - 3n + 2}{6}$ of these orbits when $n$ is congruent to $1$ or $2 \pmod3$ and $\frac{\binom{n-1}{2} - 1}{3} = \frac{n^2 - 3n}{6}$ when $n$ is congruent to $0 \pmod3$.  \\

\textit{Case Three: Second Hooks.} Let $I = [b, i, 0 :  b - t, t + i, 0],$ with $i \geq 1, t \geq 2$, and $b - t \geq 1$. In Lemma \ref{Lem: CatalogRowmotion2xn}, $I$ falls under the category of Second Hook with the additional restrictions that the maximum element of $P$ is in $I$, the minimal element of the upper chain is not in $I$, and the tail of the hook (the part of $I$ not constructed of consecutive diagonals) is at least two in length. We claim that $\row^{n+3}(I) = I$ and $\row^j(I)$ does not contain the maximum element of $P$ for $0 < j < n + 3.$   This case is illustrated in Example \ref{ex3_n+3}.

As $I$ contains the maximum element of $P$, by Lemma \ref{Lem: CatalogRowmotion2xn} Type 4b, $\row(I) = \bar{I} = [0, b, i: 0, b - t, i + t].$ By Lemma \ref{Lem: CatalogRowmotion2xn} Type 3a, a single application of rowmotion returns the Double Hook ICS $[1, b, i - 1 :  0, b - t + 1, i + t - 1].$
By Lemma \ref{Lem: CatalogRowmotion2xn} Type 1a, repeated applications of rowmotion shift this up until the interval on the lower chain reaches the top of the poset, resulting in $[i, b, 0 :  i - 1, b - t + 1, t].$ This requires $i - 1$ applications of rowmotion. Then by Lemma \ref{Lem: CatalogRowmotion2xn} Type 1b, a single application of rowmotion results in the Second Hook interval-closed set $[i + 1, b - t, t - 1 :  i, b - t + 1, t - 1]$, and by Lemma \ref{Lem: CatalogRowmotion2xn} Type 4a, another application of rowmotion results in the Disjoint interval-closed set $[i + 2, b - t, t - 2 :  0, i, b].$
By Lemma \ref{Lem: CatalogRowmotion2xn} Type 2a, repeated applications of rowmotion shift this up until the interval on the lower chain reaches the top of the poset, resulting in $[i + t, b - t, 0 :  t - 2, i, b - t + 2].$ This requires $t - 2$ applications of rowmotion. Then by Lemma \ref{Lem: CatalogRowmotion2xn} Type 2b, a single application of rowmotion results in the First Hook interval-closed set $[t - 1, i + 1 ,b - t :  t - 1, i, b - t + 1]$, and by Lemma \ref{Lem: CatalogRowmotion2xn} Type 3a, another application of rowmotion will result in the Double Hook interval-closed set $[t, i + 1, b - t - 1 :  0, i + t, b - t].$ By Lemma \ref{Lem: CatalogRowmotion2xn} Type 1a, repeated application of rowmotion will shift this up until the interval on the lower chain reaches the top of the poset. This requires $b - t - 1$ applications of rowmotion. Then, by Lemma \ref{Lem: CatalogRowmotion2xn} Type 1b, a final application of rowmotion will return to the Second Hook interval-closed set $I$. 

In total, this required $3 + i + b = n + 3$ applications of rowmotion, and none of the interval-closed sets other than $I$ included the  maximum element of $P$. Thus, $I$ is in an orbit of size $n + 3.$ For $i = 1, 2,  \ldots, n - 3$, there are $n - i - 2$ such ICS (with $t = 2, 3, \ldots, b - 2$), so there are $(n - 3) + (n - 4) + \dotsc + 1 = \frac{n^2-5n+6}{2}$ interval-closed sets of this type, and therefore the same number of orbits of this form. \\

\textit{Case Four: Double Hooks.} Let $I = [b, k + j, a :  0, b + k, a + j]$ where $j \geq 2$, $b \geq 2$, $k \geq 1$, and $a \geq 0$. In Lemma \ref{Lem: CatalogRowmotion2xn}, interval-closed sets of this type are a subset of Double Hooks. We claim that $a + b + 2$ applications of rowmotion to $I$ will return an interval-closed set (possibly $I$ itself) of this same form.  This case is illustrated in Example \ref{ex4_n+3}.

By Lemma \ref{Lem: CatalogRowmotion2xn} Type 1a, repeated applications of rowmotion will shift $I$ up one rank at a time until it reaches the top of the poset, that is, until it has the form $[a + b, k + j, 0 :  a, b + k, j]$. This requires $a$ applications of rowmotion. By Lemma \ref{Lem: CatalogRowmotion2xn} Type 1b, a further application of rowmotion will then result in the Second Hook interval-closed set $[a + 1 + b, k, j - 1 :  a + 1, b + k, j - 1]$ formed by removing all the elements above the stack of consecutive diagonals and shifting up by one. 

By Lemma \ref{Lem: CatalogRowmotion2xn} Type 4a, a single application of rowmotion will result in the Disjoint interval-closed set $[a + 2 + b, k, j - 2 :  0, a+1, b + k + j - 1]$. Then by Lemma \ref{Lem: CatalogRowmotion2xn} Type 2a, repeated applications of rowmotion will shift this up one diagonal at a time until it reaches the maximal element on the lower chain. This will require $j - 2$ applications of rowmotion and will result in $[b + a + j, k, 0 :  j - 2, a + 1, b + k  + 1]$. By Lemma \ref{Lem: CatalogRowmotion2xn} Type 2b, an additional application of rowmotion results in the First Hook interval-closed set $[j - 1, a + b + 1, k :  j - 1, a + 1, b + k].$

By Lemma \ref{Lem: CatalogRowmotion2xn} Type 3a, a single application of rowmotion will result in the Double Hook interval-closed set $[j, a + b + 1, k - 1 :  0, j + a + 1, b + k - 1].$ In total, this took  $(a + 1) + 1 + (j -2) + 1 + 1 = a + j + 2$ applications of rowmotion. If this is $I$, then $[j, a + b + 1, k - 1 :  0, j + a + 1, b + k - 1] = [b, k + j, a :  0, b + k, j + a]$, so $j = b$ and $a + 1 = k$. Thus, $I = [b, a + 1 + b, a :  0, a + b + 1, a + b]$ with $a\geq0, b\geq 2$ and $2a + 2b + 1 = n$. In this case, $I$ will be in an orbit of size $a + b + 2 =\frac{n + 3}{2}.$

If $I \neq [j, a + b + 1, k - 1 :  0, j + a + 1, b + k - 1]$, then a similar argument shows it will take another $k - 1 + b + 2$ applications of rowmotion to return to $[b, k + j, a :  0, b + k, j + a].$ Thus, $I$ is in an orbit of size $a + j + 2 + k - 1 + b + 2 = n +3.$

There are $\binom{n - 2}{3}$ ways to choose integers $a, b, k, j$ with sum $n$ such that $a \geq 0$, $b \geq 2$, $k \geq 2$ and $j \geq 1$, and when $n$ is odd there are $\frac{n - 3}{2}$ ways to choose integers $a, b$ such that $a \geq 0$, $b\geq 2$, and $2a + 2b + 1 = n$.  Thus, when $n$ is even, there are $\frac{1}{2}\binom{n - 2}{3}$ orbits of size $n + 3$ of this form, while when $n$ is odd, there are $\frac{1}{2} \left( \binom{n - 2}{3} - \frac{n - 3}{2}\right)$. 

Finally, we observe that the orbits described in the different cases are distinct: in Case One, all of the interval-closed sets in all the orbits are stacked ranks, while no other case contains any such ICS. In Cases Two and Four, no ICS in any of the orbits contains the maximum element of the poset, while in Case Three, there is exactly one interval-closed set of each orbit containing the maximum element. Thus there is no overlap between Cases Two and Four and Case Three. Lastly, as every orbit in Case Two contains exactly three interval-closed sets formed by stacks of consecutive diagonals, and the orbits in Case 4 contain none, there is no overlap between those cases either. 
\end{proof}

Theorem \ref{thm:n+3} follows immediately from Lemma \ref{n+3}. We also have the following corollary.

\begin{cor}\label{cor_n+3/2}
Let $P = [2] \times [n]$ with $n \geq 2$. Under rowmotion, $P$ has at least $1$ orbit of size $\frac{n + 3}{3}$ when $n$ is divisible by $3$ and $\frac{n - 1}{2}$ orbits of size $\frac{n + 3}{2}$ when $n$ is odd.
\end{cor}

\begin{proof}
Let $P = [2] \times [n]$. From the proof of Lemma \ref{n+3} we have the following:
\begin{itemize}
    \item If $n$ is odd, then the interval-closed set $[0, (n+1)/2, (n - 1)/2 : 0, (n-1)/2, (n+1)/2]$ consisting of $\frac{n + 1}{2}$ stacked ranks belongs to an orbit of size $\frac{n + 3}{2}$, containing all other ICS that consist of $(n + 1)/2$ stacked ranks.  (This is the sub-case of Case One in the proof of Lemma~\ref{n+3} when $J = I$.)
    \item If $n \equiv 0 \pmod{3}$ and $I = [n/3, n/3, n/3  :  n/3, n/3, n/3]$, then I belongs to an orbit of size $\frac{n + 3}{3}$. This is a special type of Stacked Diagonals interval-closed set.
    \item If $n$ is odd and $I = [b, a + 1 + b, a :  0, a + b + 1, a + b]$ with $a\geq0, b\geq 2$ and $2a + 2b + 1 = n$, then $I$ belongs to an orbit of size $\frac{n + 3}{2}.$ This is a special type of Double Hook interval-closed set contributing $\frac{n-1}{2} - 1$ orbits of size $\frac{n + 3}{2}.$  
    \qedhere
\end{itemize}
\end{proof}


Next, we consider the average value of signed cardinality over these orbits.

\begin{thm}\label{thm:sc_n+3}
Let $P = [2] \times [n]$ with $n$ odd. The average value of signed cardinality over the orbits of size $n + 3$ described in Lemma \ref{n+3} is $0$.
\end{thm}

We prove Theorem \ref{thm:sc_n+3} in cases based on the representatives for the orbits.

\begin{proof}[Proof of Case One: Stacked Ranks.]

Consider the orbit $\mathcal{O}$ containing $I = [0, r, n - r : 0, r - 1, n + 1 - r]$, the bottom $r$ stacked ranks, for $1 \leq r < n/2$.  The proof of Lemma \ref{n+3} shows that this orbit also contains the order ideal $J = [0, n - r + 1, r - 1 : 0, n - r, r]$, their complementary order filters $\overline{I}$ and $\overline{J}$, $n - r$ ICS that consist of $r$ stacked ranks and do not include the maximum or minimum elements of $P$, and $r - 1$ ICS that consist of $n + 1 - r$ stacked ranks and do not include the maximum or minimum elements of $P$.  By Proposition \ref{comp_sc}, $\SC(I) + \SC(\overline{I}) = \SC(J) + \SC(\overline{J}) = 0$.  Now consider the other ICS in $\mathcal{O}$.  If $r$ is even, so is $n + 1 - r$, and so the other ICS in $\mathcal{O}$ all have signed cardinality $0$, and we're done.  Otherwise, $n - r$ and $r - 1$ are even, so the other ICS $\row^i(I)$ for $i = 1, \ldots, r - 1$ and $\row^j(J)$ for $j = 1, \ldots, n - r$ can be split up into pairs in which one ICS is the result of shifting the other up by one rank.  Thus in this case the signed cardinalities sum to zero by Proposition \ref{shift_sc}.
\end{proof}

\begin{proof}[Proof of Case Two: Stacked Diagonals.]

The second case of Lemma \ref{n+3} consisted of those orbits that contained a Stacked Diagonal ICS $I = [b, i, a : b, i, a]$ where $b + i + a = n$ and $b, i, a > 0$.  Then $\SC(I) = 0$ by Lemma \ref{calc_sc(I)}.  Consider the $a$ ICS that follow $I$ in its orbit, all of which are shifts of the Disjoint ICS $\row(I) = [b + 1, i, a-1 :  0, b, i + a]$. If $b$ and $i$ have the same parity, then $\SC(\row(I)) = \SC(\row^j(I)) = 0$ for $j = 1, \ldots, a$ by Lemma \ref{calc_sc(I)}.  Otherwise, $b$ and $i$ have different parities, so $a = n - b - i$ is even.  In this case, we can divide these $a$ ICS into pairs that differ by a single upwards shift, so their signed cardinalities sum to $0$ by Proposition \ref{shift_sc}.  Thus in either case the total signed cardinality over this portion of the orbit is $0$.  The same analysis applies to the other two thirds of the orbit (beginning with $[i, a, b]$ and $[a, b, i]$ instead of $[b, i, a]$), so the average signed cardinality over the whole orbit is $0$.
\end{proof}

\begin{proof}[Proof of Case Three: Second Hooks.]

In Lemma \ref{n+3}, we showed that the orbit containing $I = [b, i, 0 :  b - t, t + i, 0]$ with $i \geq 1$, $t \geq 2$, and $b - t \geq 1$, consists of the following $n + 3$ ICS:
\begin{itemize}
    \item $I$ and $\row(I) = \overline{I},$
    \item the Double Hook interval-closed set $[1, b, i - 1 :  0, b - t + 1, i + t - 1]$, and its $i - 1$ upward shifts,
    \item the Second Hook interval-closed set $[i + 1, b - t, t - 1 :  i, b - t + 1, t - 1],$
    \item the Disjoint interval-closed set $[i + 2, b - t, t - 2 :  0, i, b]$, and its $t - 2$ upward shifts,
    \item the First Hook interval-closed set $[t - 1, i + 1, b - t :  t - 1, i, b - t + 1],$ and
    \item the Double Hook interval-closed set $[t, i + 1, b - t - 1 :  0, i + t, b - t]$, and its $b - t - 1$ upward shifts.
\end{itemize}

First, we consider the signed cardinality of the interval-closed sets involved in the orbit that are not shifted upward through rowmotion. By Proposition \ref{comp_sc}, the complement of $I$ has signed cardinality that cancels with the signed cardinality of $I$. By Lemma \ref{calc_sc(I)}, $\SC([i + 1, b - t, t - 1 :  i, b - t + 1, t - 1])$ is $-1$ if $i$ is even and $+1$ if $i$ is odd, and $\SC([t - 1, i + 1 ,b - t :  t - 1, i, b - t + 1])$ is $+1$ if $b-t$ is even and  $-1$ if $b - t$ is odd.  The same kinds of calculations apply in the other cases; we combine them with Proposition \ref{shift_sc} to deal with the cases where the ICS is shifted upward.  The results of this calculation (grouping the shifts together, and using the fact that $b + i = n$, so $b$ and $i$ always have opposite parities) are recorded in the following table.
\begin{center}
\begin{tabular}[H]{ | m{29em} | m{1cm}| m{1cm} |  m{1cm}| m{1cm} |} 
  \hline
   & $i$ odd, $t$ even & $i$ odd, $t$ odd & $i$ even, $t$ even& $i$ even, $t$ odd \\ 
  \hline
$I$ and $\row(I) = \overline{I}$ & $0$ & $0$ & $0$ & $0$ \\ 
  \hline
Double Hook $[1, b, i - 1 :  0, b - t + 1, i + t - 1]$, shifted $i - 1$ times & $-1$ & $0$ & $0$ & $0$ \\ 
  \hline
  Second Hook $[i + 1, b - t, t - 1 :  i, b - t + 1, t - 1]$ & $+1$ & $+1$ & $-1$ & $-1$ \\ 
  \hline
   Disjoint $[i + 2, b - t, t - 2 :  0, i, b]$, shifted $t - 2$ times & $-1$ & $0$ & $+1$ & $0$ \\ 
  \hline
First Hook $[t - 1, i + 1 ,b - t :  t - 1, i, b - t + 1]$ & $+1$ & $-1$ & $-1$ & $+1$ \\ 
  \hline
Double Hook $[t, i + 1, b - t - 1 :  0, i + t, b - t]$, shifted $b - t - 1$ times & $0$ & $0$ & $+1$ & $0$ \\ 
  \hline
\end{tabular}
\end{center}
Thus, in all cases, signed cardinality over the orbit sums to $0$, as claimed.
\end{proof}

\begin{proof}[Proof of Case Four: Double Hooks]
Given integers $a, b, j, k$ such that $a + b + j + k = n$, let $I(a, b, j, k) = [b, k + j, a : 0, b + k, a + j]$ when that makes sense.  In Lemma \ref{n+3}, we showed that when $a \geq 0$, $b \geq 2$, $j \geq 2$, and $k \geq 1$, the orbit containing the Double Hook ICS $I(a, b, j, k)$ consists of the following ICS:
\begin{itemize}
    \item $I(a, b, j, k)$ and its $a$ upward shifts,
    \item the Second Hook ICS $[a + 1 + b, k, j-1 :  a+1, b+k, j-1]$,
    \item the Disjoint ICS $I(j - 2, a + b + 2, b + k + 1, -b - 1)$ and its $j - 2$ upward shifts,
    \item the First Hook ICS $[j-1, a + b + 1, k :  j-1, a + 1, b + k],$ and
    \item the ICS that we get by making the substitution $a \mapsto k - 1$, $b \mapsto j$, $j \mapsto b$, and $k \mapsto a + 1$ in the preceding four bullet-points.
\end{itemize}
By Lemma~\ref{shift_sc}, the total signed cardinality of an ICS $J$ and $m$ of its upward shifts is $\delta_{m + 1} \cdot \SC(J)$.  Therefore, by Lemma \ref{calc_sc(I)}, the total contributions to signed cardinality from each of the cases in the list above are respectively
\begin{itemize}
    \item $\left(\delta_{k + j}\cdot (-1)^b - \delta_{b + k}\right) \cdot \delta_{a + 1}$,
    \item $\delta_k \cdot (-1)^{a + b + 1} + \delta_{b + k} \cdot (-1)^a$,
    \item $\left(\delta_{k}\cdot (-1)^{a + b} - \delta_{a + 1}\right) \cdot \delta_{j - 1}$,
    \item $\delta_{a + b + 1} \cdot (-1)^{j - 1} + \delta_{a + 1} \cdot (-1)^j$, and
    \item respectively \quad $\left(\delta_{a + b + 1}\cdot (-1)^j - \delta_{a + j + 1}\right) \cdot \delta_{k}$, \qquad
        $\delta_{a + 1} \cdot (-1)^{j + k} + \delta_{a + j + 1} \cdot (-1)^{k - 1}$, \\
        $\left(\delta_{a + 1}\cdot (-1)^{j + k - 1} - \delta_{k}\right) \cdot \delta_{b - 1}$, \quad
        and \quad $\delta_{j + k} \cdot (-1)^{b - 1} + \delta_{k} \cdot (-1)^b$.
\end{itemize}
The following chart records each of these eight values in order, depending on the parities of $a$, $b$, $j$, and $k$, subject to the requirement that $a + b + j + k = n$ is odd:
\begin{center}
\begin{tabular}[H]{|c|c|c|c||c|c|c|c|c|c|c|c|} 
  \hline
$a$ & $b$ & $j$ & $k$ &&&&&&&& \\ 
  \hline\hline
even & even & even &  odd & $0$ & $0$ & $0$ & $0$ & $0$ & $0$ & $0$ & $0$ \\\hline
even & even &  odd & even & $+1$& $0$ & $0$ & $0$ & $0$ & $-1$& $+1$& $-1$\\\hline
even &  odd & even & even & $-1$& $+1$& $-1$& $+1$& $0$ & $0$ & $0$ & $0$ \\\hline
even &  odd &  odd &  odd & $0$ & $+1$& $0$ & $-1$& $0$ & $+1$& $0$ & $-1$\\\hline
 odd & even & even & even & $0$ & $0$ & $0$ & $0$ & $0$ & $0$ & $0$ & $0$ \\\hline
 odd & even &  odd &  odd & $0$ & $0$ & $0$ & $0$ & $-1$& $+1$& $-1$& $+1$\\\hline
 odd &  odd & even &  odd & $0$ & $-1$& $+1$& $-1$& $+1$& $0$ & $0$ & $0$ \\\hline
 odd &  odd &  odd & even & $0$ & $-1$& $0$ & $+1$& $0$ & $-1$& $0$ & $+1$\\\hline
\end{tabular}
\end{center}
In all eight cases, we see that the total signed cardinality (the sum across each row) is $0$, as needed.
\end{proof}

\begin{cor}\label{cor:sc_n+3}
Let $P = [2] \times [n]$ with $n$ odd. The average signed cardinality over the orbits of size $\frac{n+3}{3}$ or $\frac{n+3}{2}$ described in Corollary \ref{cor_n+3/2} is $0$.
\end{cor}

\begin{proof}
When $n$ is congruent to $0 \pmod{3}$, $P$ has a single orbit of size $\frac{n + 3}{3}$ with representative $I = [n/3,n/3,n/3 : n/3, n/3, n/3].$ In the proof of Theorem \ref{thm:sc_n+3} we showed that orbits of this type have signed cardinality 0.

When $n$ is odd, $P$ has $\frac{n-1}{2}$ orbits of size $\frac{n+3}{2}.$ One of these consists of interval-closed sets made up of stacked ranks. In the proof of Theorem \ref{thm:sc_n+3} we showed that orbits of this type have signed cardinality 0. The other $\frac{n-3}{2}$ orbits have representative $I = [b, a + 1 + b, a :  0, a + b + 1, a + b]$ with $a\geq0, b\geq 2$ and $2a + 2b + 1 = n$. This is given in the fourth case of Theorem \ref{thm:sc_n+3}, specifically when $j$ and $b$ have the same parity, $a$ and $i$ have differing parity, and the orbit only goes through one rotation of Double Hook ICS, Second Hook ICS, Disjoint ICS, First Hook ICS (which is why the orbit size is smaller). In the proof of Theorem \ref{thm:sc_n+3}, this corresponds to only the first two columns of each table and only looking at the first half of the tables. In each case, we can see the average signed cardinality over the orbit is 0.
\end{proof}


To illustrate the four types of orbits discussed in Lemma \ref{n+3}, we include the following examples. The boxed values in the figures are the signed cardinality for the given interval-closed sets.

\begin{example}\label{ex1_n+3} Let $n = 7$ and $P = [2] \times [n]$. Then the following illustrates an orbit of size $n+3$ falling under Case One starting with $I = [0, 3, 4: 0, 2, 5]$. 
\begin{center}
\begin{minipage}[c]{2.5cm}
\begin{tikzpicture}[scale = .45]
\draw [-, ultra thick] (-1,0) -- (-2,1);
\draw [-, ultra thick] (-2,1) -- (-3,2);
\draw [-, ultra thick] (-3,2) -- (-4,3);
\draw [-, ultra thick] (-4,3) -- (-5,4);
\draw [-, ultra thick] (-5,4) -- (-6,5);
\draw [-, ultra thick] (-3,0) -- (-4,1);
\draw [-, ultra thick] (-2,-1) -- (-3,0);
\draw [-, ultra thick] (-4,1) -- (-5,2);
\draw [-, ultra thick] (-5,2) -- (-6,3);
\draw [-, ultra thick] (-6,3) -- (-7,4);
\draw [-, ultra thick] (-7,4) -- (-8,5);
\draw [-, ultra thick] (-6,5) -- (-7,6);
\draw [-, ultra thick] (-1,0) -- (-2,-1);
\draw [-, ultra thick] (-2,1) -- (-3,0);
\draw [-, ultra thick] (-3,2) -- (-4,1);
\draw [-, ultra thick] (-4,3) -- (-5,2);
\draw [-, ultra thick] (-5,4) -- (-6,3);
\draw [-, ultra thick] (-6,5) -- (-7,4);
\draw [-, ultra thick] (-7,6) -- (-8,5);
\draw[fill=red, radius = .2] (-1,0) circle [radius = 0.2];
\draw[fill=red, radius = .2] (-2,1) circle [radius = 0.2];
\draw[fill=white, radius = .2] (-3,2) circle [radius = 0.2];
\draw[fill=white, radius = .2] (-4,3) circle [radius = 0.2];
\draw[fill=white, radius = .2] (-5,4) circle [radius = 0.2];
\draw[fill=white, radius = .2] (-6,5) circle [radius = 0.2];
\draw[fill=white, radius = .2] (-7,6) circle [radius = 0.2];
\draw[fill=red, radius = .2] (-2,-1) circle [radius = 0.2];
\draw[fill=red, radius = .2] (-3,0) circle [radius = 0.2];
\draw[fill=red, radius = .2] (-4,1) circle [radius = 0.2];
\draw[fill=white, radius = .2] (-5,2) circle [radius = 0.2];
\draw[fill=white, radius = .2] (-6,3) circle [radius = 0.2];
\draw[fill=white, radius = .2] (-7,4) circle [radius = 0.2];
\draw[fill=white, radius = .2] (-8,5) circle [radius = 0.2];
\node[draw,align=left] at (-5,0){$+1$};
\end{tikzpicture}
  \end{minipage}
\begin{minipage}[c]{.1cm}$\quad\xrightarrow{\text{Row}}$\end{minipage}
  \begin{minipage}[c]{2.5cm}
\begin{tikzpicture}[scale = .45]
\draw [-, ultra thick] (-1,0) -- (-2,1);
\draw [-, ultra thick] (-2,1) -- (-3,2);
\draw [-, ultra thick] (-3,2) -- (-4,3);
\draw [-, ultra thick] (-4,3) -- (-5,4);
\draw [-, ultra thick] (-5,4) -- (-6,5);
\draw [-, ultra thick] (-3,0) -- (-4,1);
\draw [-, ultra thick] (-2,-1) -- (-3,0);
\draw [-, ultra thick] (-4,1) -- (-5,2);
\draw [-, ultra thick] (-5,2) -- (-6,3);
\draw [-, ultra thick] (-6,3) -- (-7,4);
\draw [-, ultra thick] (-7,4) -- (-8,5);
\draw [-, ultra thick] (-6,5) -- (-7,6);
\draw [-, ultra thick] (-1,0) -- (-2,-1);
\draw [-, ultra thick] (-2,1) -- (-3,0);
\draw [-, ultra thick] (-3,2) -- (-4,1);
\draw [-, ultra thick] (-4,3) -- (-5,2);
\draw [-, ultra thick] (-5,4) -- (-6,3);
\draw [-, ultra thick] (-6,5) -- (-7,4);
\draw [-, ultra thick] (-7,6) -- (-8,5);
\draw[fill=red, radius = .2] (-1,0) circle [radius = 0.2];
\draw[fill=red, radius = .2] (-2,1) circle [radius = 0.2];
\draw[fill=red, radius = .2] (-3,2) circle [radius = 0.2];
\draw[fill=white, radius = .2] (-4,3) circle [radius = 0.2];
\draw[fill=white, radius = .2] (-5,4) circle [radius = 0.2];
\draw[fill=white, radius = .2] (-6,5) circle [radius = 0.2];
\draw[fill=white, radius = .2] (-7,6) circle [radius = 0.2];
\draw[fill=white, radius = .2] (-2,-1) circle [radius = 0.2];
\draw[fill=red, radius = .2] (-3,0) circle [radius = 0.2];
\draw[fill=red, radius = .2] (-4,1) circle [radius = 0.2];
\draw[fill=red, radius = .2] (-5,2) circle [radius = 0.2];
\draw[fill=white, radius = .2] (-6,3) circle [radius = 0.2];
\draw[fill=white, radius = .2] (-7,4) circle [radius = 0.2];
\draw[fill=white, radius = .2] (-8,5) circle [radius = 0.2];
\node[draw,align=left] at (-5,0){$-2$};
\end{tikzpicture}
  \end{minipage}
\begin{minipage}[c]{.1cm}$\quad\xrightarrow{\text{Row}}$\end{minipage}
\begin{minipage}[c]{2.5cm}
\begin{tikzpicture}[scale = .45]
\draw [-, ultra thick] (-1,0) -- (-2,1);
\draw [-, ultra thick] (-2,1) -- (-3,2);
\draw [-, ultra thick] (-3,2) -- (-4,3);
\draw [-, ultra thick] (-4,3) -- (-5,4);
\draw [-, ultra thick] (-5,4) -- (-6,5);
\draw [-, ultra thick] (-3,0) -- (-4,1);
\draw [-, ultra thick] (-2,-1) -- (-3,0);
\draw [-, ultra thick] (-4,1) -- (-5,2);
\draw [-, ultra thick] (-5,2) -- (-6,3);
\draw [-, ultra thick] (-6,3) -- (-7,4);
\draw [-, ultra thick] (-7,4) -- (-8,5);
\draw [-, ultra thick] (-6,5) -- (-7,6);
\draw [-, ultra thick] (-1,0) -- (-2,-1);
\draw [-, ultra thick] (-2,1) -- (-3,0);
\draw [-, ultra thick] (-3,2) -- (-4,1);
\draw [-, ultra thick] (-4,3) -- (-5,2);
\draw [-, ultra thick] (-5,4) -- (-6,3);
\draw [-, ultra thick] (-6,5) -- (-7,4);
\draw [-, ultra thick] (-7,6) -- (-8,5);
\draw[fill=white, radius = .2] (-1,0) circle [radius = 0.2];
\draw[fill=red, radius = .2] (-2,1) circle [radius = 0.2];
\draw[fill=red, radius = .2] (-3,2) circle [radius = 0.2];
\draw[fill=red, radius = .2] (-4,3) circle [radius = 0.2];
\draw[fill=white, radius = .2] (-5,4) circle [radius = 0.2];
\draw[fill=white, radius = .2] (-6,5) circle [radius = 0.2];
\draw[fill=white, radius = .2] (-7,6) circle [radius = 0.2];
\draw[fill=white, radius = .2] (-2,-1) circle [radius = 0.2];
\draw[fill=white, radius = .2] (-3,0) circle [radius = 0.2];
\draw[fill=red, radius = .2] (-4,1) circle [radius = 0.2];
\draw[fill=red, radius = .2] (-5,2) circle [radius = 0.2];
\draw[fill=red, radius = .2] (-6,3) circle [radius = 0.2];
\draw[fill=white, radius = .2] (-7,4) circle [radius = 0.2];
\draw[fill=white, radius = .2] (-8,5) circle [radius = 0.2];
\node[draw,align=left] at (-5,0){$+2$};
\end{tikzpicture}
  \end{minipage}
\begin{minipage}[c]{.1cm}$\quad\xrightarrow{\text{Row}}$\end{minipage}
  \begin{minipage}[c]{2.5cm}
\begin{tikzpicture}[scale = .45]
\draw [-, ultra thick] (-1,0) -- (-2,1);
\draw [-, ultra thick] (-2,1) -- (-3,2);
\draw [-, ultra thick] (-3,2) -- (-4,3);
\draw [-, ultra thick] (-4,3) -- (-5,4);
\draw [-, ultra thick] (-5,4) -- (-6,5);
\draw [-, ultra thick] (-3,0) -- (-4,1);
\draw [-, ultra thick] (-2,-1) -- (-3,0);
\draw [-, ultra thick] (-4,1) -- (-5,2);
\draw [-, ultra thick] (-5,2) -- (-6,3);
\draw [-, ultra thick] (-6,3) -- (-7,4);
\draw [-, ultra thick] (-7,4) -- (-8,5);
\draw [-, ultra thick] (-6,5) -- (-7,6);
\draw [-, ultra thick] (-1,0) -- (-2,-1);
\draw [-, ultra thick] (-2,1) -- (-3,0);
\draw [-, ultra thick] (-3,2) -- (-4,1);
\draw [-, ultra thick] (-4,3) -- (-5,2);
\draw [-, ultra thick] (-5,4) -- (-6,3);
\draw [-, ultra thick] (-6,5) -- (-7,4);
\draw [-, ultra thick] (-7,6) -- (-8,5);
\draw[fill=white, radius = .2] (-1,0) circle [radius = 0.2];
\draw[fill=white, radius = .2] (-2,1) circle [radius = 0.2];
\draw[fill=red, radius = .2] (-3,2) circle [radius = 0.2];
\draw[fill=red, radius = .2] (-4,3) circle [radius = 0.2];
\draw[fill=red, radius = .2] (-5,4) circle [radius = 0.2];
\draw[fill=white, radius = .2] (-6,5) circle [radius = 0.2];
\draw[fill=white, radius = .2] (-7,6) circle [radius = 0.2];
\draw[fill=white, radius = .2] (-2,-1) circle [radius = 0.2];
\draw[fill=white, radius = .2] (-3,0) circle [radius = 0.2];
\draw[fill=white, radius = .2] (-4,1) circle [radius = 0.2];
\draw[fill=red, radius = .2] (-5,2) circle [radius = 0.2];
\draw[fill=red, radius = .2] (-6,3) circle [radius = 0.2];
\draw[fill=red, radius = .2] (-7,4) circle [radius = 0.2];
\draw[fill=white, radius = .2] (-8,5) circle [radius = 0.2];
\node[draw,align=left] at (-5,0){$-2$};
\end{tikzpicture}
  \end{minipage}
\begin{minipage}[c]{.1cm}$\quad\xrightarrow{\text{Row}}$\end{minipage}
  \begin{minipage}[c]{2.5cm}
\begin{tikzpicture}[scale = .45]
\draw [-, ultra thick] (-1,0) -- (-2,1);
\draw [-, ultra thick] (-2,1) -- (-3,2);
\draw [-, ultra thick] (-3,2) -- (-4,3);
\draw [-, ultra thick] (-4,3) -- (-5,4);
\draw [-, ultra thick] (-5,4) -- (-6,5);
\draw [-, ultra thick] (-3,0) -- (-4,1);
\draw [-, ultra thick] (-2,-1) -- (-3,0);
\draw [-, ultra thick] (-4,1) -- (-5,2);
\draw [-, ultra thick] (-5,2) -- (-6,3);
\draw [-, ultra thick] (-6,3) -- (-7,4);
\draw [-, ultra thick] (-7,4) -- (-8,5);
\draw [-, ultra thick] (-6,5) -- (-7,6);
\draw [-, ultra thick] (-1,0) -- (-2,-1);
\draw [-, ultra thick] (-2,1) -- (-3,0);
\draw [-, ultra thick] (-3,2) -- (-4,1);
\draw [-, ultra thick] (-4,3) -- (-5,2);
\draw [-, ultra thick] (-5,4) -- (-6,3);
\draw [-, ultra thick] (-6,5) -- (-7,4);
\draw [-, ultra thick] (-7,6) -- (-8,5);
\draw[fill=white, radius = .2] (-1,0) circle [radius = 0.2];
\draw[fill=white, radius = .2] (-2,1) circle [radius = 0.2];
\draw[fill=white, radius = .2] (-3,2) circle [radius = 0.2];
\draw[fill=red, radius = .2] (-4,3) circle [radius = 0.2];
\draw[fill=red, radius = .2] (-5,4) circle [radius = 0.2];
\draw[fill=red, radius = .2] (-6,5) circle [radius = 0.2];
\draw[fill=white, radius = .2] (-7,6) circle [radius = 0.2];
\draw[fill=white, radius = .2] (-2,-1) circle [radius = 0.2];
\draw[fill=white, radius = .2] (-3,0) circle [radius = 0.2];
\draw[fill=white, radius = .2] (-4,1) circle [radius = 0.2];
\draw[fill=white, radius = .2] (-5,2) circle [radius = 0.2];
\draw[fill=red, radius = .2] (-6,3) circle [radius = 0.2];
\draw[fill=red, radius = .2] (-7,4) circle [radius = 0.2];
\draw[fill=red, radius = .2] (-8,5) circle [radius = 0.2];
\node[draw,align=left] at (-5,0){$+2$};
\end{tikzpicture}
  \end{minipage}
  \begin{minipage}[c]{.1cm}$\quad\xrightarrow{\text{Row}}$\end{minipage}
\end{center}


\begin{center}
\begin{minipage}[c]{2.5cm}
\begin{tikzpicture}[scale = .45]
\draw [-, ultra thick] (-1,0) -- (-2,1);
\draw [-, ultra thick] (-2,1) -- (-3,2);
\draw [-, ultra thick] (-3,2) -- (-4,3);
\draw [-, ultra thick] (-4,3) -- (-5,4);
\draw [-, ultra thick] (-5,4) -- (-6,5);
\draw [-, ultra thick] (-3,0) -- (-4,1);
\draw [-, ultra thick] (-2,-1) -- (-3,0);
\draw [-, ultra thick] (-4,1) -- (-5,2);
\draw [-, ultra thick] (-5,2) -- (-6,3);
\draw [-, ultra thick] (-6,3) -- (-7,4);
\draw [-, ultra thick] (-7,4) -- (-8,5);
\draw [-, ultra thick] (-6,5) -- (-7,6);
\draw [-, ultra thick] (-1,0) -- (-2,-1);
\draw [-, ultra thick] (-2,1) -- (-3,0);
\draw [-, ultra thick] (-3,2) -- (-4,1);
\draw [-, ultra thick] (-4,3) -- (-5,2);
\draw [-, ultra thick] (-5,4) -- (-6,3);
\draw [-, ultra thick] (-6,5) -- (-7,4);
\draw [-, ultra thick] (-7,6) -- (-8,5);
\draw[fill=white, radius = .2] (-1,0) circle [radius = 0.2];
\draw[fill=white, radius = .2] (-2,1) circle [radius = 0.2];
\draw[fill=white, radius = .2] (-3,2) circle [radius = 0.2];
\draw[fill=white, radius = .2] (-4,3) circle [radius = 0.2];
\draw[fill=red, radius = .2] (-5,4) circle [radius = 0.2];
\draw[fill=red, radius = .2] (-6,5) circle [radius = 0.2];
\draw[fill=red, radius = .2] (-7,6) circle [radius = 0.2];
\draw[fill=white, radius = .2] (-2,-1) circle [radius = 0.2];
\draw[fill=white, radius = .2] (-3,0) circle [radius = 0.2];
\draw[fill=white, radius = .2] (-4,1) circle [radius = 0.2];
\draw[fill=white, radius = .2] (-5,2) circle [radius = 0.2];
\draw[fill=white, radius = .2] (-6,3) circle [radius = 0.2];
\draw[fill=red, radius = .2] (-7,4) circle [radius = 0.2];
\draw[fill=red, radius = .2] (-8,5) circle [radius = 0.2];
\node[draw,align=left] at (-5,0){$-1$};
\end{tikzpicture}
  \end{minipage}
\begin{minipage}[c]{.1cm}$\quad\xrightarrow{\text{Row}}$\end{minipage}
  \begin{minipage}[c]{2.5cm}
\begin{tikzpicture}[scale = .45]
\draw [-, ultra thick] (-1,0) -- (-2,1);
\draw [-, ultra thick] (-2,1) -- (-3,2);
\draw [-, ultra thick] (-3,2) -- (-4,3);
\draw [-, ultra thick] (-4,3) -- (-5,4);
\draw [-, ultra thick] (-5,4) -- (-6,5);
\draw [-, ultra thick] (-3,0) -- (-4,1);
\draw [-, ultra thick] (-2,-1) -- (-3,0);
\draw [-, ultra thick] (-4,1) -- (-5,2);
\draw [-, ultra thick] (-5,2) -- (-6,3);
\draw [-, ultra thick] (-6,3) -- (-7,4);
\draw [-, ultra thick] (-7,4) -- (-8,5);
\draw [-, ultra thick] (-6,5) -- (-7,6);
\draw [-, ultra thick] (-1,0) -- (-2,-1);
\draw [-, ultra thick] (-2,1) -- (-3,0);
\draw [-, ultra thick] (-3,2) -- (-4,1);
\draw [-, ultra thick] (-4,3) -- (-5,2);
\draw [-, ultra thick] (-5,4) -- (-6,3);
\draw [-, ultra thick] (-6,5) -- (-7,4);
\draw [-, ultra thick] (-7,6) -- (-8,5);
\draw[fill=red, radius = .2] (-1,0) circle [radius = 0.2];
\draw[fill=red, radius = .2] (-2,1) circle [radius = 0.2];
\draw[fill=red, radius = .2] (-3,2) circle [radius = 0.2];
\draw[fill=red, radius = .2] (-4,3) circle [radius = 0.2];
\draw[fill=white, radius = .2] (-5,4) circle [radius = 0.2];
\draw[fill=white, radius = .2] (-6,5) circle [radius = 0.2];
\draw[fill=white, radius = .2] (-7,6) circle [radius = 0.2];
\draw[fill=red, radius = .2] (-2,-1) circle [radius = 0.2];
\draw[fill=red, radius = .2] (-3,0) circle [radius = 0.2];
\draw[fill=red, radius = .2] (-4,1) circle [radius = 0.2];
\draw[fill=red, radius = .2] (-5,2) circle [radius = 0.2];
\draw[fill=red, radius = .2] (-6,3) circle [radius = 0.2];
\draw[fill=white, radius = .2] (-7,4) circle [radius = 0.2];
\draw[fill=white, radius = .2] (-8,5) circle [radius = 0.2];
\node[draw,align=left] at (-5,0){$+1$};
\end{tikzpicture}
  \end{minipage}
\begin{minipage}[c]{.1cm}$\quad\xrightarrow{\text{Row}}$\end{minipage}
\begin{minipage}[c]{2.5cm}
\begin{tikzpicture}[scale = .45]
\draw [-, ultra thick] (-1,0) -- (-2,1);
\draw [-, ultra thick] (-2,1) -- (-3,2);
\draw [-, ultra thick] (-3,2) -- (-4,3);
\draw [-, ultra thick] (-4,3) -- (-5,4);
\draw [-, ultra thick] (-5,4) -- (-6,5);
\draw [-, ultra thick] (-3,0) -- (-4,1);
\draw [-, ultra thick] (-2,-1) -- (-3,0);
\draw [-, ultra thick] (-4,1) -- (-5,2);
\draw [-, ultra thick] (-5,2) -- (-6,3);
\draw [-, ultra thick] (-6,3) -- (-7,4);
\draw [-, ultra thick] (-7,4) -- (-8,5);
\draw [-, ultra thick] (-6,5) -- (-7,6);
\draw [-, ultra thick] (-1,0) -- (-2,-1);
\draw [-, ultra thick] (-2,1) -- (-3,0);
\draw [-, ultra thick] (-3,2) -- (-4,1);
\draw [-, ultra thick] (-4,3) -- (-5,2);
\draw [-, ultra thick] (-5,4) -- (-6,3);
\draw [-, ultra thick] (-6,5) -- (-7,4);
\draw [-, ultra thick] (-7,6) -- (-8,5);
\draw[fill=red, radius = .2] (-1,0) circle [radius = 0.2];
\draw[fill=red, radius = .2] (-2,1) circle [radius = 0.2];
\draw[fill=red, radius = .2] (-3,2) circle [radius = 0.2];
\draw[fill=red, radius = .2] (-4,3) circle [radius = 0.2];
\draw[fill=red, radius = .2] (-5,4) circle [radius = 0.2];
\draw[fill=white, radius = .2] (-6,5) circle [radius = 0.2];
\draw[fill=white, radius = .2] (-7,6) circle [radius = 0.2];
\draw[fill=white, radius = .2] (-2,-1) circle [radius = 0.2];
\draw[fill=red, radius = .2] (-3,0) circle [radius = 0.2];
\draw[fill=red, radius = .2] (-4,1) circle [radius = 0.2];
\draw[fill=red, radius = .2] (-5,2) circle [radius = 0.2];
\draw[fill=red, radius = .2] (-6,3) circle [radius = 0.2];
\draw[fill=red, radius = .2] (-7,4) circle [radius = 0.2];
\draw[fill=white, radius = .2] (-8,5) circle [radius = 0.2];
\node[draw,align=left] at (-5,0){$-2$};
\end{tikzpicture}
  \end{minipage}
\begin{minipage}[c]{.1cm}$\quad\xrightarrow{\text{Row}}$\end{minipage}
  \begin{minipage}[c]{2.5cm}
\begin{tikzpicture}[scale = .45]
\draw [-, ultra thick] (-1,0) -- (-2,1);
\draw [-, ultra thick] (-2,1) -- (-3,2);
\draw [-, ultra thick] (-3,2) -- (-4,3);
\draw [-, ultra thick] (-4,3) -- (-5,4);
\draw [-, ultra thick] (-5,4) -- (-6,5);
\draw [-, ultra thick] (-3,0) -- (-4,1);
\draw [-, ultra thick] (-2,-1) -- (-3,0);
\draw [-, ultra thick] (-4,1) -- (-5,2);
\draw [-, ultra thick] (-5,2) -- (-6,3);
\draw [-, ultra thick] (-6,3) -- (-7,4);
\draw [-, ultra thick] (-7,4) -- (-8,5);
\draw [-, ultra thick] (-6,5) -- (-7,6);
\draw [-, ultra thick] (-1,0) -- (-2,-1);
\draw [-, ultra thick] (-2,1) -- (-3,0);
\draw [-, ultra thick] (-3,2) -- (-4,1);
\draw [-, ultra thick] (-4,3) -- (-5,2);
\draw [-, ultra thick] (-5,4) -- (-6,3);
\draw [-, ultra thick] (-6,5) -- (-7,4);
\draw [-, ultra thick] (-7,6) -- (-8,5);
\draw[fill=white, radius = .2] (-1,0) circle [radius = 0.2];
\draw[fill=red, radius = .2] (-2,1) circle [radius = 0.2];
\draw[fill=red, radius = .2] (-3,2) circle [radius = 0.2];
\draw[fill=red, radius = .2] (-4,3) circle [radius = 0.2];
\draw[fill=red, radius = .2] (-5,4) circle [radius = 0.2];
\draw[fill=red, radius = .2] (-6,5) circle [radius = 0.2];
\draw[fill=white, radius = .2] (-7,6) circle [radius = 0.2];
\draw[fill=white, radius = .2] (-2,-1) circle [radius = 0.2];
\draw[fill=white, radius = .2] (-3,0) circle [radius = 0.2];
\draw[fill=red, radius = .2] (-4,1) circle [radius = 0.2];
\draw[fill=red, radius = .2] (-5,2) circle [radius = 0.2];
\draw[fill=red, radius = .2] (-6,3) circle [radius = 0.2];
\draw[fill=red, radius = .2] (-7,4) circle [radius = 0.2];
\draw[fill=red, radius = .2] (-8,5) circle [radius = 0.2];
\node[draw,align=left] at (-5,0){$+2$};
\end{tikzpicture}
  \end{minipage}
\begin{minipage}[c]{.1cm}$\quad\xrightarrow{\text{Row}}$\end{minipage}
  \begin{minipage}[c]{2.5cm}
\begin{tikzpicture}[scale = .45]
\draw [-, ultra thick] (-1,0) -- (-2,1);
\draw [-, ultra thick] (-2,1) -- (-3,2);
\draw [-, ultra thick] (-3,2) -- (-4,3);
\draw [-, ultra thick] (-4,3) -- (-5,4);
\draw [-, ultra thick] (-5,4) -- (-6,5);
\draw [-, ultra thick] (-3,0) -- (-4,1);
\draw [-, ultra thick] (-2,-1) -- (-3,0);
\draw [-, ultra thick] (-4,1) -- (-5,2);
\draw [-, ultra thick] (-5,2) -- (-6,3);
\draw [-, ultra thick] (-6,3) -- (-7,4);
\draw [-, ultra thick] (-7,4) -- (-8,5);
\draw [-, ultra thick] (-6,5) -- (-7,6);
\draw [-, ultra thick] (-1,0) -- (-2,-1);
\draw [-, ultra thick] (-2,1) -- (-3,0);
\draw [-, ultra thick] (-3,2) -- (-4,1);
\draw [-, ultra thick] (-4,3) -- (-5,2);
\draw [-, ultra thick] (-5,4) -- (-6,3);
\draw [-, ultra thick] (-6,5) -- (-7,4);
\draw [-, ultra thick] (-7,6) -- (-8,5);
\draw[fill=white, radius = .2] (-1,0) circle [radius = 0.2];
\draw[fill=white, radius = .2] (-2,1) circle [radius = 0.2];
\draw[fill=red, radius = .2] (-3,2) circle [radius = 0.2];
\draw[fill=red, radius = .2] (-4,3) circle [radius = 0.2];
\draw[fill=red, radius = .2] (-5,4) circle [radius = 0.2];
\draw[fill=red, radius = .2] (-6,5) circle [radius = 0.2];
\draw[fill=red, radius = .2] (-7,6) circle [radius = 0.2];
\draw[fill=white, radius = .2] (-2,-1) circle [radius = 0.2];
\draw[fill=white, radius = .2] (-3,0) circle [radius = 0.2];
\draw[fill=white, radius = .2] (-4,1) circle [radius = 0.2];
\draw[fill=red, radius = .2] (-5,2) circle [radius = 0.2];
\draw[fill=red, radius = .2] (-6,3) circle [radius = 0.2];
\draw[fill=red, radius = .2] (-7,4) circle [radius = 0.2];
\draw[fill=red, radius = .2] (-8,5) circle [radius = 0.2];
\node[draw,align=left] at (-5,0){$-1$};
\end{tikzpicture}
  \end{minipage}
  \begin{minipage}[c]{.1cm}$\quad\xrightarrow{\text{Row}}$\end{minipage}
\end{center}
\end{example}


\begin{example}\label{ex2_n+3} Let $n = 7$ and $P = [2] \times [n]$. Then the following illustrates an orbit of size $n+3$ falling under Case Two starting with $I = [3, 3, 1: 3, 3, 1]$.

\begin{center}
\begin{minipage}[c]{2.5cm}

  \end{minipage}
  \begin{minipage}[c]{.1cm}$\quad\xrightarrow{\text{Row}}$\end{minipage}
\end{center}
\end{example}


\begin{example}\label{ex3_n+3} Let $n = 7$ and $P = [2] \times [n]$. Then the following illustrates an orbit of size $n+3$ falling under Case Three starting with $I = [5, 2, 0: 1, 6, 0].$
\begin{center}
\begin{minipage}[c]{2.5cm}
%
  \end{minipage}
  \begin{minipage}[c]{.1cm}$\quad\xrightarrow{\text{Row}}$\end{minipage}
\end{center}

\end{example}


\begin{example}\label{ex4_n+3} Let $n = 7$ and $P = [2] \times [n]$. Then the following illustrates an orbit of size $n+3$ falling under Case Four starting with $I = [3, 3, 1: 0, 4, 3]$.
\begin{center}
  \begin{minipage}[c]{2.5cm}
%
  \end{minipage}
\begin{minipage}[c]{.1cm}$\quad\xrightarrow{\text{Row}}$\end{minipage}
 
\end{center}
\end{example}

\subsection{Orbits of size dividing $n + 5$}\label{sec:n + 5}
In this section, we analyze orbits of size dividing $n+5$, giving representatives for them and computing the average signed cardinality over each such orbit.

\begin{lem}\label{lem:n + 5} 
  Fix any integers $x, y, z$ such that $x, y, z \geq 4$ and $x + y + z = n + 5$. Let $I$ be the ICS $[x - 1, y, z - 4 : \varnothing]$.  Then the following hold.
   \begin{enumerate}
   \item Applying rowmotion $z$ times to $I$
   results in the ICS  $[z - 1, x, y - 4 : \varnothing]$.
   \item None of the $z - 1$ intermediate ICS is a Low ICS.
   \end{enumerate}
\end{lem}
\begin{proof}
Suppose $x, y, z \geq 4$, $x + y + z = n + 5$, and $[x - 1, y, z - 4 : \varnothing]$.  We first claim that after applying $z - 3$ applications of rowmotion to $I$, we arrive at the ICS $J = [\varnothing : z - 4, x, y - 1]$ (a High ICS).  In the case that $z = 4$, this is because of a single application of Lemma \ref{Lem: CatalogRowmotion2xn} Type 5b. In the case that $z > 4$, it instead is a three-step process: after one application of rowmotion we get $[x, y, z - 5 : 0, x, n - x]$ by Lemma \ref{Lem: CatalogRowmotion2xn} Type 5a, then by Lemma \ref{Lem: CatalogRowmotion2xn} Type 2a this marches up one diagonal at a time for $z - 5$ steps, leading to $[n - y, y, 0 : z - 5, x, y]$, and then by Lemma \ref{Lem: CatalogRowmotion2xn} Type 2c we end up at $J$, as claimed.

From $J$, we have that three successive applications of rowmotion produce $[z - 3, n - z + 3, 0 : z - 3, x, y - 2]$ (by Lemma \ref{Lem: CatalogRowmotion2xn} Type 6a), $[z - 2, x, y - 3 : 0, x + z - 2, y - 3]$ (by Lemma \ref{Lem: CatalogRowmotion2xn} Type 3b), and finally $[z - 1, x, y - 4 : \varnothing]$ (by Lemma \ref{Lem: CatalogRowmotion2xn} Type 4b), as claimed.
\end{proof}

\begin{example}\label{ex:n+5} Let $n = 8$ and $P = [2] \times [n]$. Then the following illustrates part of an orbit of size $n+5$ satisfying Lemma \ref{lem:n + 5} with $x = 4, y = 4, z = 5$ and $I = [3, 4, 1 : \varnothing]$.

\begin{center}
\begin{minipage}[c]{2.2cm}
\begin{tikzpicture}[scale = .35]
\draw [-, ultra thick] (-1,0) -- (-2,1);
\draw [-, ultra thick] (-2,1) -- (-3,2);
\draw [-, ultra thick] (-3,2) -- (-4,3);
\draw [-, ultra thick] (-4,3) -- (-5,4);
\draw [-, ultra thick] (-5,4) -- (-6,5);
\draw [-, ultra thick] (-3,0) -- (-4,1);
\draw [-, ultra thick] (-2,-1) -- (-3,0);
\draw [-, ultra thick] (-4,1) -- (-5,2);
\draw [-, ultra thick] (-5,2) -- (-6,3);
\draw [-, ultra thick] (-6,3) -- (-7,4);
\draw [-, ultra thick] (-7,4) -- (-8,5);
\draw [-, ultra thick] (-6,5) -- (-7,6);
\draw [-, ultra thick] (-8,5) -- (-9,6);
\draw [-, ultra thick] (-7,6) -- (-8, 7);
\draw [-, ultra thick] (-1,0) -- (-2,-1);
\draw [-, ultra thick] (-2,1) -- (-3,0);
\draw [-, ultra thick] (-3,2) -- (-4,1);
\draw [-, ultra thick] (-4,3) -- (-5,2);
\draw [-, ultra thick] (-5,4) -- (-6,3);
\draw [-, ultra thick] (-6,5) -- (-7,4);
\draw [-, ultra thick] (-7,6) -- (-8,5);
\draw [-, ultra thick] (-8,7) -- (-9,6);
\draw[fill=white, radius = .2] (-1,0) circle [radius = 0.2];
\draw[fill=white, radius = .2] (-2,1) circle [radius = 0.2];
\draw[fill=white, radius = .2] (-3,2) circle [radius = 0.2];
\draw[fill=white, radius = .2] (-4,3) circle [radius = 0.2];
\draw[fill=white, radius = .2] (-5,4) circle [radius = 0.2];
\draw[fill=white, radius = .2] (-6,5) circle [radius = 0.2];
\draw[fill=white, radius = .2] (-7,6) circle [radius = 0.2];
\draw[fill=white, radius = .2] (-8,7) circle [radius = 0.2];
\draw[fill=white, radius = .2] (-2,-1) circle [radius = 0.2];
\draw[fill=white, radius = .2] (-3,0) circle [radius = 0.2];
\draw[fill=white, radius = .2] (-4,1) circle [radius = 0.2];
\draw[fill=red, radius = .2] (-5,2) circle [radius = 0.2];
\draw[fill=red, radius = .2] (-6,3) circle [radius = 0.2];
\draw[fill=red, radius = .2] (-7,4) circle [radius = 0.2];
\draw[fill=red, radius = .2] (-8,5) circle [radius = 0.2];
\draw[fill=white, radius = .2] (-9,6) circle [radius = 0.2];
\end{tikzpicture}
  \end{minipage}
\begin{minipage}[c]{.1cm}$\quad\xrightarrow{\text{Row}}$\end{minipage}
  \begin{minipage}[c]{2.2cm}
\begin{tikzpicture}[scale = .35]
\draw [-, ultra thick] (-1,0) -- (-2,1);
\draw [-, ultra thick] (-2,1) -- (-3,2);
\draw [-, ultra thick] (-3,2) -- (-4,3);
\draw [-, ultra thick] (-4,3) -- (-5,4);
\draw [-, ultra thick] (-5,4) -- (-6,5);
\draw [-, ultra thick] (-3,0) -- (-4,1);
\draw [-, ultra thick] (-2,-1) -- (-3,0);
\draw [-, ultra thick] (-4,1) -- (-5,2);
\draw [-, ultra thick] (-5,2) -- (-6,3);
\draw [-, ultra thick] (-6,3) -- (-7,4);
\draw [-, ultra thick] (-7,4) -- (-8,5);
\draw [-, ultra thick] (-6,5) -- (-7,6);
\draw [-, ultra thick] (-8,5) -- (-9,6);
\draw [-, ultra thick] (-7,6) -- (-8, 7);
\draw [-, ultra thick] (-1,0) -- (-2,-1);
\draw [-, ultra thick] (-2,1) -- (-3,0);
\draw [-, ultra thick] (-3,2) -- (-4,1);
\draw [-, ultra thick] (-4,3) -- (-5,2);
\draw [-, ultra thick] (-5,4) -- (-6,3);
\draw [-, ultra thick] (-6,5) -- (-7,4);
\draw [-, ultra thick] (-7,6) -- (-8,5);
\draw [-, ultra thick] (-8,7) -- (-9,6);
\draw[fill=red, radius = .2] (-1,0) circle [radius = 0.2];
\draw[fill=red, radius = .2] (-2,1) circle [radius = 0.2];
\draw[fill=red, radius = .2] (-3,2) circle [radius = 0.2];
\draw[fill=red, radius = .2] (-4,3) circle [radius = 0.2];
\draw[fill=white, radius = .2] (-5,4) circle [radius = 0.2];
\draw[fill=white, radius = .2] (-6,5) circle [radius = 0.2];
\draw[fill=white, radius = .2] (-7,6) circle [radius = 0.2];
\draw[fill=white, radius = .2] (-8,7) circle [radius = 0.2];
\draw[fill=white, radius = .2] (-2,-1) circle [radius = 0.2];
\draw[fill=white, radius = .2] (-3,0) circle [radius = 0.2];
\draw[fill=white, radius = .2] (-4,1) circle [radius = 0.2];
\draw[fill=white, radius = .2] (-5,2) circle [radius = 0.2];
\draw[fill=red, radius = .2] (-6,3) circle [radius = 0.2];
\draw[fill=red, radius = .2] (-7,4) circle [radius = 0.2];
\draw[fill=red, radius = .2] (-8,5) circle [radius = 0.2];
\draw[fill=red, radius = .2] (-9,6) circle [radius = 0.2];
\end{tikzpicture}
  \end{minipage}
\begin{minipage}[c]{.1cm}$\quad\xrightarrow{\text{Row}}$\end{minipage}
\begin{minipage}[c]{2.2cm}
\begin{tikzpicture}[scale = .35]
\draw [-, ultra thick] (-1,0) -- (-2,1);
\draw [-, ultra thick] (-2,1) -- (-3,2);
\draw [-, ultra thick] (-3,2) -- (-4,3);
\draw [-, ultra thick] (-4,3) -- (-5,4);
\draw [-, ultra thick] (-5,4) -- (-6,5);
\draw [-, ultra thick] (-3,0) -- (-4,1);
\draw [-, ultra thick] (-2,-1) -- (-3,0);
\draw [-, ultra thick] (-4,1) -- (-5,2);
\draw [-, ultra thick] (-5,2) -- (-6,3);
\draw [-, ultra thick] (-6,3) -- (-7,4);
\draw [-, ultra thick] (-7,4) -- (-8,5);
\draw [-, ultra thick] (-6,5) -- (-7,6);
\draw [-, ultra thick] (-8,5) -- (-9,6);
\draw [-, ultra thick] (-7,6) -- (-8, 7);
\draw [-, ultra thick] (-1,0) -- (-2,-1);
\draw [-, ultra thick] (-2,1) -- (-3,0);
\draw [-, ultra thick] (-3,2) -- (-4,1);
\draw [-, ultra thick] (-4,3) -- (-5,2);
\draw [-, ultra thick] (-5,4) -- (-6,3);
\draw [-, ultra thick] (-6,5) -- (-7,4);
\draw [-, ultra thick] (-7,6) -- (-8,5);
\draw [-, ultra thick] (-8,7) -- (-9,6);
\draw[fill=white, radius = .2] (-1,0) circle [radius = 0.2];
\draw[fill=red, radius = .2] (-2,1) circle [radius = 0.2];
\draw[fill=red, radius = .2] (-3,2) circle [radius = 0.2];
\draw[fill=red, radius = .2] (-4,3) circle [radius = 0.2];
\draw[fill=red, radius = .2] (-5,4) circle [radius = 0.2];
\draw[fill=white, radius = .2] (-6,5) circle [radius = 0.2];
\draw[fill=white, radius = .2] (-7,6) circle [radius = 0.2];
\draw[fill=white, radius = .2] (-8,7) circle [radius = 0.2];
\draw[fill=white, radius = .2] (-2,-1) circle [radius = 0.2];
\draw[fill=white, radius = .2] (-3,0) circle [radius = 0.2];
\draw[fill=white, radius = .2] (-4,1) circle [radius = 0.2];
\draw[fill=white, radius = .2] (-5,2) circle [radius = 0.2];
\draw[fill=white, radius = .2] (-6,3) circle [radius = 0.2];
\draw[fill=white, radius = .2] (-7,4) circle [radius = 0.2];
\draw[fill=white, radius = .2] (-8,5) circle [radius = 0.2];
\draw[fill=white, radius = .2] (-9,6) circle [radius = 0.2];
\end{tikzpicture}
  \end{minipage}
\begin{minipage}[c]{.1cm}$\quad\xrightarrow{\text{Row}}$\end{minipage}
  \begin{minipage}[c]{2.2cm}
\begin{tikzpicture}[scale = .35]
\draw [-, ultra thick] (-1,0) -- (-2,1);
\draw [-, ultra thick] (-2,1) -- (-3,2);
\draw [-, ultra thick] (-3,2) -- (-4,3);
\draw [-, ultra thick] (-4,3) -- (-5,4);
\draw [-, ultra thick] (-5,4) -- (-6,5);
\draw [-, ultra thick] (-3,0) -- (-4,1);
\draw [-, ultra thick] (-2,-1) -- (-3,0);
\draw [-, ultra thick] (-4,1) -- (-5,2);
\draw [-, ultra thick] (-5,2) -- (-6,3);
\draw [-, ultra thick] (-6,3) -- (-7,4);
\draw [-, ultra thick] (-7,4) -- (-8,5);
\draw [-, ultra thick] (-6,5) -- (-7,6);
\draw [-, ultra thick] (-8,5) -- (-9,6);
\draw [-, ultra thick] (-7,6) -- (-8, 7);
\draw [-, ultra thick] (-1,0) -- (-2,-1);
\draw [-, ultra thick] (-2,1) -- (-3,0);
\draw [-, ultra thick] (-3,2) -- (-4,1);
\draw [-, ultra thick] (-4,3) -- (-5,2);
\draw [-, ultra thick] (-5,4) -- (-6,3);
\draw [-, ultra thick] (-6,5) -- (-7,4);
\draw [-, ultra thick] (-7,6) -- (-8,5);
\draw [-, ultra thick] (-8,7) -- (-9,6);
\draw[fill=white, radius = .2] (-1,0) circle [radius = 0.2];
\draw[fill=white, radius = .2] (-2,1) circle [radius = 0.2];
\draw[fill=red, radius = .2] (-3,2) circle [radius = 0.2];
\draw[fill=red, radius = .2] (-4,3) circle [radius = 0.2];
\draw[fill=red, radius = .2] (-5,4) circle [radius = 0.2];
\draw[fill=red, radius = .2] (-6,5) circle [radius = 0.2];
\draw[fill=white, radius = .2] (-7,6) circle [radius = 0.2];
\draw[fill=white, radius = .2] (-8, 7) circle [radius = 0.2];
\draw[fill=white, radius = .2] (-2,-1) circle [radius = 0.2];
\draw[fill=white, radius = .2] (-3,0) circle [radius = 0.2];
\draw[fill=red, radius = .2] (-4,1) circle [radius = 0.2];
\draw[fill=red, radius = .2] (-5,2) circle [radius = 0.2];
\draw[fill=red, radius = .2] (-6,3) circle [radius = 0.2];
\draw[fill=red, radius = .2] (-7,4) circle [radius = 0.2];
\draw[fill=red, radius = .2] (-8,5) circle [radius = 0.2];
\draw[fill=red, radius = .2] (-9,6) circle [radius = 0.2];
\end{tikzpicture}
  \end{minipage}
\begin{minipage}[c]{.1cm}$\quad\xrightarrow{\text{Row}}$\end{minipage}
  \begin{minipage}[c]{2.2cm}
\begin{tikzpicture}[scale = .35]
\draw [-, ultra thick] (-1,0) -- (-2,1);
\draw [-, ultra thick] (-2,1) -- (-3,2);
\draw [-, ultra thick] (-3,2) -- (-4,3);
\draw [-, ultra thick] (-4,3) -- (-5,4);
\draw [-, ultra thick] (-5,4) -- (-6,5);
\draw [-, ultra thick] (-3,0) -- (-4,1);
\draw [-, ultra thick] (-2,-1) -- (-3,0);
\draw [-, ultra thick] (-4,1) -- (-5,2);
\draw [-, ultra thick] (-5,2) -- (-6,3);
\draw [-, ultra thick] (-6,3) -- (-7,4);
\draw [-, ultra thick] (-7,4) -- (-8,5);
\draw [-, ultra thick] (-6,5) -- (-7,6);
\draw [-, ultra thick] (-8,5) -- (-9,6);
\draw [-, ultra thick] (-7,6) -- (-8, 7);
\draw [-, ultra thick] (-1,0) -- (-2,-1);
\draw [-, ultra thick] (-2,1) -- (-3,0);
\draw [-, ultra thick] (-3,2) -- (-4,1);
\draw [-, ultra thick] (-4,3) -- (-5,2);
\draw [-, ultra thick] (-5,4) -- (-6,3);
\draw [-, ultra thick] (-6,5) -- (-7,4);
\draw [-, ultra thick] (-7,6) -- (-8,5);
\draw [-, ultra thick] (-8,7) -- (-9,6);
\draw[fill=red, radius = .2] (-1,0) circle [radius = 0.2];
\draw[fill=red, radius = .2] (-2,1) circle [radius = 0.2];
\draw[fill=red, radius = .2] (-3,2) circle [radius = 0.2];
\draw[fill=red, radius = .2] (-4,3) circle [radius = 0.2];
\draw[fill=red, radius = .2] (-5,4) circle [radius = 0.2];
\draw[fill=red, radius = .2] (-6,5) circle [radius = 0.2];
\draw[fill=red, radius = .2] (-7,6) circle [radius = 0.2];
\draw[fill=white, radius = .2] (-8, 7) circle [radius = 0.2];
\draw[fill=white, radius = .2] (-2,-1) circle [radius = 0.2];
\draw[fill=white, radius = .2] (-3,0) circle [radius = 0.2];
\draw[fill=white, radius = .2] (-4,1) circle [radius = 0.2];
\draw[fill=red, radius = .2] (-5,2) circle [radius = 0.2];
\draw[fill=red, radius = .2] (-6,3) circle [radius = 0.2];
\draw[fill=red, radius = .2] (-7,4) circle [radius = 0.2];
\draw[fill=red, radius = .2] (-8,5) circle [radius = 0.2];
\draw[fill=white, radius = .2] (-9,6) circle [radius = 0.2];
\end{tikzpicture}
  \end{minipage}
  \begin{minipage}[c]{.1cm}$\quad\xrightarrow{\text{Row}}$\end{minipage}
  \begin{minipage}[c]{2.2cm}
\begin{tikzpicture}[scale = .35]
\draw [-, ultra thick] (-1,0) -- (-2,1);
\draw [-, ultra thick] (-2,1) -- (-3,2);
\draw [-, ultra thick] (-3,2) -- (-4,3);
\draw [-, ultra thick] (-4,3) -- (-5,4);
\draw [-, ultra thick] (-5,4) -- (-6,5);
\draw [-, ultra thick] (-3,0) -- (-4,1);
\draw [-, ultra thick] (-2,-1) -- (-3,0);
\draw [-, ultra thick] (-4,1) -- (-5,2);
\draw [-, ultra thick] (-5,2) -- (-6,3);
\draw [-, ultra thick] (-6,3) -- (-7,4);
\draw [-, ultra thick] (-7,4) -- (-8,5);
\draw [-, ultra thick] (-6,5) -- (-7,6);
\draw [-, ultra thick] (-8,5) -- (-9,6);
\draw [-, ultra thick] (-7,6) -- (-8, 7);
\draw [-, ultra thick] (-1,0) -- (-2,-1);
\draw [-, ultra thick] (-2,1) -- (-3,0);
\draw [-, ultra thick] (-3,2) -- (-4,1);
\draw [-, ultra thick] (-4,3) -- (-5,2);
\draw [-, ultra thick] (-5,4) -- (-6,3);
\draw [-, ultra thick] (-6,5) -- (-7,4);
\draw [-, ultra thick] (-7,6) -- (-8,5);
\draw [-, ultra thick] (-8,7) -- (-9,6);
\draw[fill=white, radius = .2] (-1,0) circle [radius = 0.2];
\draw[fill=white, radius = .2] (-2,1) circle [radius = 0.2];
\draw[fill=white, radius = .2] (-3,2) circle [radius = 0.2];
\draw[fill=white, radius = .2] (-4,3) circle [radius = 0.2];
\draw[fill=white, radius = .2] (-5,4) circle [radius = 0.2];
\draw[fill=white, radius = .2] (-6,5) circle [radius = 0.2];
\draw[fill=white, radius = .2] (-7,6) circle [radius = 0.2];
\draw[fill=white, radius = .2] (-8,7) circle [radius = 0.2];
\draw[fill=white, radius = .2] (-2,-1) circle [radius = 0.2];
\draw[fill=white, radius = .2] (-3,0) circle [radius = 0.2];
\draw[fill=white, radius = .2] (-4,1) circle [radius = 0.2];
\draw[fill=white, radius = .2] (-5,2) circle [radius = 0.2];
\draw[fill=red, radius = .2] (-6,3) circle [radius = 0.2];
\draw[fill=red, radius = .2] (-7,4) circle [radius = 0.2];
\draw[fill=red, radius = .2] (-8,5) circle [radius = 0.2];
\draw[fill=red, radius = .2] (-9,6) circle [radius = 0.2];
\end{tikzpicture}
  \end{minipage}
 \end{center}
\end{example}

Lemma~\ref{lem:n + 5} leads immediately to the classification of orbits of size $n + 5$ and $\frac{n + 5}{3}$.

\begin{thm}\label{thm:n+5}
    When $n \equiv 0, 2 \pmod{3}$ and $n \geq 7$, there is a collection of $\frac{1}{3} \binom{n - 5}{2}$ orbits, each of size $n + 5$, that together include all ICS of the form $[x - 1, y, n-x - y + 1: \varnothing]$    with $x, y \geq 4$.

    When $n \equiv 1 \pmod{3}$ and $n \geq 7$, there is a collection consisting of one orbit of size $\frac{n + 5}{3}$ and $\frac{(n - 4)(n - 7)}{3}$ orbits of size $n + 5$ that together include all ICS of the form  $[ x - 1, y, n - x - y + 1: \varnothing]$ with $x, y \geq 4$.
\end{thm}
\begin{proof}
    For this proof, define a \emph{relevant} ICS to be an interval of the form
    $[x-1,y, n- x - y +1 : \varnothing]$ with $x, y \geq 4$ and $x + y - 1 \leq n$.  In this case, defining $z := n + 5 - x - y$, we have that relevant ICS are in natural correspondence with integer solutions $(x, y, z)$ to the equation $x + y + z = n + 5$ with $x, y, z \geq 4$.  By standard enumerative machinery (``balls and boxes''), the number of such solutions is $\binom{n - 5}{2}$.  Fix one such solution $(x, y, z)$, and let $I$ be the associated relevant interval.

    If $x = y = z$ then it follows from  Lemma~\ref{lem:n + 5} (1) that applying rowmotion $z = \frac{n + 5}{3}$ times to $I$ returns $I$, and from Lemma~\ref{lem:n + 5} (2) that none of the intermediate $\frac{n + 5}{3} - 1$ ICS are equal to $I$.  Thus $I$ belongs to an orbit of size $\frac{n + 5}{3}$.

    Otherwise (if it is not the case that $x = y = z$), we have by Lemma~\ref{lem:n + 5} (1) that applying rowmotion $z$ times to $I$ returns 
    $[z-1, x, y-4 : \varnothing]$ (corresponding to the solution $(z, x, y)$).  Again by Leamma~\ref{lem:n + 5} (1), we have that applying rowmotion to this interval $y$ times returns 
    $[y - 1, z, x - 4 : \varnothing]$ (corresponding to the solution $(y, z, x)$), and applying rowmotion a further $x$ times returns 
    $[x - 1, y, z - 4: \varnothing] = I$.  Thus, $I$ belongs to an orbit of order dividing $x + y + z = n + 5$.  Since it is not the case that $x = y = z$, the three relevant ICS 
    $[x - 1, y, z - 4: \varnothing]$, $[z - 1, x, y - 4: \varnothing]$, and $[y - 1, z, x - 4: \varnothing]$
    are distinct.  Moreover, by Lemma~\ref{lem:n + 5} (2), none of the other $n + 2$ ICS that arise when applying rowmotion $n + 5$ times to $I$ are relevant.  Thus in fact the orbit of $I$ has size $n + 5$.

    The case that $x = y = z$ arises if and only if $n + 5$ is a multiple of $3$, i.e., if $n \equiv 1 \pmod{3}$.  In this case, the preceding argument shows that the solution $x = y = z= \frac{n + 5}{3}$ corresponds to an ICS in an orbit of size $\frac{n + 5}{3}$, while the remaining $\binom{n - 5}{2} - 1 = \frac{(n - 4)(n - 7)}{2}$ relevant ICS all belong to orbits of size $n + 5$, with three relevant ICS in each orbit, and thus $\frac{(n - 4)(n - 7)}{6}$ distinct such orbits.  In the case that $n \not\equiv 1 \pmod{3}$, we instead conclude that the $\binom{n - 5}{2}$ relevant ICS all belong to orbits of size $n + 5$, with three relevant ICS in each orbit, and thus $\frac{1}{3} \binom{n - 5}{2}$ distinct such orbits.  This completes the proof.
\end{proof}

\begin{thm}\label{thm:sc_n+5}
 Suppose that $n$ is odd and that $\O$ is an orbit of interval-closed sets of $[2] \times [n]$ of size $n + 5$ or $(n + 5)/3$.  Then the average value $\frac{1}{\#\O}\sum_{I \in \O} sc(I)$ of the signed cardinality over $\O$ is equal to $0$.
\end{thm}
\begin{proof}
We continue with the notation from the proof of Theorem~\ref{thm:n+5}.  Let's consider the $z$ ICS
\begin{multline*}
[x - 1 , y, z - 4 : \varnothing], \row([x - 1 , y, z - 4 : \varnothing]), \row^2([x - 1 , y, z - 4 : \varnothing]), \ldots, \\ \row^{z - 1}([x - 1 , y, z - 4 : \varnothing]) = \row^{-1}([z - 1, x, y - 4 : \varnothing]).
\end{multline*}
In the case $z = 4$, by the proof of Lemma~\ref{lem:n + 5} , these are
\[
[x - 1, y, 0 : \varnothing], \quad
[\varnothing : 0, x, y - 1], \quad
[1, n - 1, 0 : 1, x, y - 2], \quad
[2, x, y - 3 : 0, x + 2, y - 3].
\]
By Lemma \ref{calc_sc(I)}, the signed cardinalities of these ICS are respectively
\[
-\delta_y (-1)^{x} + 0, \quad
0 - \delta_x, \quad 
-\delta_{n - 1} + \delta_x, \quad
\text{and} \quad
\delta_x - \delta_x,
\]
and since $n$ is odd they have sum $-\delta_y (-1)^{x}$.

If instead $z > 4$, we have (again by the proof of Lemma~\ref{lem:n + 5} ) that the $z$ ICS are
\begin{multline*}
[x - 1, y, z - 4: \varnothing], \\
[x, y, z - 5: 0, x, n - x], \quad
[x + 1, y, z - 6: 1, x, n - x - 1], \quad
\ldots, \quad
[n - y, y, 0: z - 5, x, y], \\
[\varnothing : z - 4, x, y - 1], \quad
[z - 3, n - z + 3, 0: z - 3, x, y - 2], \quad
[z - 2, x, y - 3: 0, x + z - 2, y - 3].
\end{multline*}
Again by Lemma \ref{calc_sc(I)}, the first of these has signed cardinality $-\delta_y (-1)^{x}$, while the second has signed cardinality $\delta_y (-1)^x - \delta_x$.
Since the second through $(z - 3)$rd are just shifts, combining Lemma \ref{calc_sc(I)} and Proposition \ref{shift_sc}, we have that the total contribution of those $z - 4$ ICS is $(\delta_y (-1)^x - \delta_x) \cdot \delta_z$.   
Since $x + y + z = n - 5$ and $n$ is odd, one can see by considering cases for the parities that this can be simplified to $\delta_z (-1)^x$.
Finally the contributions of the last three are respectively
\[
0 - \delta_x (-1)^z, \quad
-\delta_{z} (-1)^z + \delta_x (-1)^z = \delta_z + \delta_x(-1)^z, \quad
\text{and} \quad
\delta_x (-1)^z - \delta_{x + z} = -\delta_z.
\]
Summing these various contributions and simplifying (using again that $x + y + z$ is even)
yields
\[
(\delta_z - \delta_y)(-1)^x = \frac{(-1)^z - (-1)^y}{2}.
\]
Observe (comparing with the earlier analysis) that this formula is also valid when $z = 4$, so we need not keep track of a separate case.  Then it follows that\footnote{in the $n + 5$ case; the $\frac{n + 5}{3}$ case is slightly different} the total signed cardinality of the whole orbit is
\[
\frac{(-1)^{z} - (-1)^y}{2} +
\frac{(-1)^{y} - (-1)^x}{2} +
\frac{(-1)^{x} - (-1)^z}{2} 
= 0,
\]
as claimed.
\end{proof}

\subsection{The quadratic orbits}\label{sec:large}
In this section, we examine the third class of orbits under rowmotion, those that have large size that grows quadratically with $n$. These orbits are unlike other orbit classes, which are either of size 2 or grow linearly with $n.$  Most orbits in this third class have lower-chain singleton representatives, i.e., interval-closed sets of the form $I=[x - 1,1,n-x:\varnothing] = \{(1,x)\}.$   
We give a characterization of these orbits in Theorem \ref{large_orbits}, while in Theorem \ref{lem:sc-big-orbits}, we prove the average value of the signed cardinality statistic for each of these orbits is zero when $n$ is odd.  
To prove Theorems \ref{large_orbits} and \ref{lem:sc-big-orbits}, we study commonly appearing orbit substructures, ultimately focusing our attention on substructures between consecutive appearances of lower-chain singleton ICS.   In particular, we find that for $1<x\leq n-6,$ the next lower-chain singleton appearing in an orbit after $\{(1,x)\}$ is $\{(1,x+6)\}$ (Lemma \ref{prop:big orbits 2n+12}).  A more complicated process determines the next lower-chain singleton appearing in an orbit after $\{(1,n-y)\}$ for $y<6$ (Lemma \ref{prop:big orbits edge cases}).  These two results allow us to use the ordering of lower-chain singleton representatives to stitch together the larger orbits based on values of $x$ and $n\bmod{6}.$ In particular, we are able to determine the size and lower-chain representatives of these orbits from subcycles of a weighted permutation.

The first substructures we observe give rules for mapping between low and high interval-closed sets under applications of rowmotion, Propositions \ref{prop:Low-to-High} and \ref{prop:High-to-Low}.  While these results follow directly from Lemma \ref{Lem: CatalogRowmotion2xn}, they provide a useful framework for examining more complex substructures, as seen in Lemmas \ref{prop:big orbits 2n+12}, \ref{prop:big orbit 2phase} and \ref{prop:big orbits edge cases}.

\begin{prop}\label{prop:Low-to-High} A Low (type 5) ICS of the form $I=[b,i,a:\varnothing]$ is mapped to a High (type 6) ICS under $a+1$ applications of rowmotion, with no High or Low ICS occurring in between. In particular, 
$$\row^{a+1}(I)=[\varnothing:a,b+1,i-1].$$
\end{prop}
\begin{proof}
Let $I=[b,i,a:\varnothing]$ with $a\geq 1.$  By Lemma \ref{Lem: CatalogRowmotion2xn} for type 5a ICS, $\row(I)$ is the type 2 Disjoint ICS $[b+1,i,a-1:0,b+1,n-b-1].$  Again by Lemma \ref{Lem: CatalogRowmotion2xn}, $\row^{a-1}([b+1,i,a-1:0,b+1,n-b-1])$ equals the type 2c Disjoint ICS $[n-i,i,0:a-1,b+1,i]$, and one more application of rowmotion returns the type 6 High ICS $[\varnothing:a,b+1,i-1].$  When $a=0$ the result follows trivially from the action of rowmotion on type 5b Low ICS.  
\end{proof}

\begin{prop}\label{prop:High-to-Low} Let $I=[\varnothing:b,i,a]$ be a High (type 6) ICS with $i\not=n.$ $I$ is mapped to a Low (type 5) ICS under three applications of rowmotion with no High or Low ICS in between. In particular, 
\begin{itemize} 
    \item[a.] 
    If $a\geq 3,$ $\row^3(I)=[b+3,i,a-3:\varnothing]$
    \item[b.] 
    If $0<a<3,$ $\row^3(I)=[2-a,n-i,i-2+a:\varnothing]$
    \item[c.] 
    If $a=0$ and
    \begin{itemize}
        \item[i.] $1<i<n,$ then $\row^3(I)=[2,n-i,i-2:\varnothing],$
        \item[ii.] $i=1,$ then $\row^3(I)=[0,1,n-1:\varnothing].$
    \end{itemize}
\end{itemize}
\end{prop}

\begin{proof}
By Lemma \ref{Lem: CatalogRowmotion2xn}, when $a=0$ and $i=n,$ one application of rowmotion returns the complement Low ICS $\row(I)=[0,n,0:\varnothing].$  For all other possible values of $a$ and $i$ we proceed by cases.

\medskip

\noindent (Case 1: $a\geq 3$) In this case, the High ICS $[\varnothing:b,i,a]$ with $a\geq 3$ is of type 6a, so rowmotion makes it a First Hook ICS of type 3b, $[b+1, i+a-1,0; b+1,i, a-1]$. Applying rowmotion again gives a Second Hook of type 4b, namely $[b+2,i,a-2: 0,n-a+2, a-2]$. One further application of rowmotion gives the Low ICS $[b+3, i, a-3: \varnothing]$, as desired.  

\medskip
\noindent(Case 2: $a=2$) Starting with the High ICS of type 6a $[\varnothing: n-i-2, i,2]$, applying rowmotion gives a First Hook of type 3b $[n-i-1, i+1,0:n-i-1,i,1]$. Applying rowmotion again gives a Second Hook of type 4d, $[n-i,i,0:0,n,0]$ which in turn gives its complement, the Low interval-closet set $[0,n-i,i:\varnothing]$.

\medskip
\noindent(Case 3: $a=1$) Starting with a High ICS of type 6b, $[\varnothing:n-i-1,i,1]$, and applying rowmotion gives the Stacked Diagonals $[n-i,i,0:n-i,i,0]$ of type 4c. Rowmotion gives the complement stacked diagonal $[0,n-i,i:0,n-i,i]$ of type 4b, which in turn gives the Low ICS $[1,n-i,i-1:\varnothing]$.

\medskip
\noindent(Case 4: $a=0, 1<i<n$)  If we start with a High ICS of type 6c, $[\varnothing:n-i,i,0]$, with $1<i<n$, rowmotion gives the complement, namely the First Hook of type 3b $[0,n,0:0,n-i,i]$. Rowmotion again gives the Second Hook $[1,n-i,i-1: 0,n-i+1,i-1]$, of type 4b. A final application of rowmotion gives its complement, the Low ICS $[2,n-i,i-2: \varnothing]$.

\medskip
\noindent(Case 5: $a=0, i=1$)  Finally, we start with the High ICS  $[\varnothing: n-1,1,0]$, which is the singleton top element. Applying rowmotion gives the complement ICS, which is a First Hook of type 3b, $[0,n,0:0,n-1,1]$. Rowmotion on this latter ICS gives the Second Hook $[1,n-1,0:0,n,0]$ of type 4d, which in turn gives the Low ICS $[0,1,n-1: \varnothing]$.
\end{proof}

Mirroring our approach to orbit structure, we examine the signed cardinality statistic on frequently appearing substructures to piece together the average signed cardinality statistic for large orbits.  As we are ultimately interested in the signed cardinality statistic when $n$ is odd, we will focus on that case alone.
We start by finding the sum of the signed cardinality statistic over the partial orbits examined in Propositions \ref{prop:Low-to-High} and \ref{prop:High-to-Low} when $n$ is odd.  Because we will be adding these partial sums together, we include the signed cardinality statistic for the starting ICS of a given substructure in our sum, but not the ending ICS.  For example, starting with a High (type 6) ICS examined in Proposition \ref{prop:High-to-Low}, we calculate the sum of the signed cardinality statistics $\SC(I)+\SC(\row(I))+\SC(\row^2(I)).$  One more application of rowmotion gives the Low (type 5) ICS $\row^3(I).$  

Recall from \eqref{eq:parity delta} the parity indicator function $\delta_x$, defined to be $1$ when $x$ is odd and $0$ otherwise.  This allows us to express useful signed cardinality statistic results for High and Low interval-closed sets as follows.
\begin{lem}\label{lem:sc_High-and-Low} If $n$ is odd and $I$ is a High (type 6) ICS of the form $[\varnothing:b,i,a]$ with $i\not=n$, then
\begin{equation*}
\SC(I)+\SC(\row(I))+\SC(\row^2(I))=
            \begin{cases}
        0  & \text{if } a\geq2,\\
        \delta_i & \text{if } a=1,\quad\\
        -1  & \text{if } a=0.
    \end{cases}
\end{equation*}
If $n$ is odd and $I$ is a Low (type 5) ICS of the form $[b,i,a:\varnothing],$ then 
\begin{equation*}
\sum_{k=0}^a\SC(\row^k(I))=\delta_a(-\delta_{i+1})+\delta_{a+1}\delta_i=\delta_{a+i}(-1)^{a}.
\end{equation*}
\end{lem}

\begin{proof}

\noindent(Case 1: High (type 6) ICS) 
 Recall $b=n-i-a$ and consider the ICS $I=[\varnothing:n-i-a,i,a].$ 
 
 If $a\geq 2$
then by Cases 1 and 2 in the proof of Proposition
\ref{prop:High-to-Low}
$$[\varnothing:n-i-a,i,a]\xrightarrow{\row}[n-i-a+1, i+a-1,0; n-i-a+1,i, a-1]\xrightarrow{\row}[n-i-a+2,i,a-2: 0,n-a+2, a-2].$$
Using Lemma \ref{calc_sc(I)} to calculate the signed cardinality statistic we find $\SC(I)+\SC(\row(I))+\SC(\row^2(I))=$ $$\delta_i(-1)^{n-i-a+1}+\left(\delta_{i+a-1}(-1)^{n-i-a+1}+\delta_i(-1)^{n-i-a}\right)+\left(\delta_i(-1)^{n-i-a}+\delta_{n-a}(-1)\right).$$  Noting that $\delta_{i+a-1}$ is $1$ when $i$ and $a$ have the same parity, and zero otherwise, we simplify this expression to find 
\begin{equation}\label{eq:sc_sum1}\SC(I)+\SC(\row(I))+\SC(\row^2(I))=\delta_{i+a-1}(-1)^{n+1}+\delta_i(-1)^{n-i-a}-\delta_{n-a}.
\end{equation}
Given $n$ is odd, Equation \eqref{eq:sc_sum1} becomes $$\SC(I)+\SC(\row(I))+\SC(\row^2(I))=\delta_{i+a-1}+\delta_i(-1)^{i+a+1}-\delta_{a+1}.$$  If $a$ and $i$ have the same parity this equals $1-(\delta_i+\delta_{a+1})=1-1=0,$ while if $a$ and $i$ have opposite parity we get $0+(\delta_i-\delta_{a+1})=0.$  

If $a=1,$ or equivalently $I=[\varnothing: n-i-1,i,1],$ we have by Case 3 in the proof of Proposition \ref{prop:High-to-Low}
$$[\varnothing:n-i-1,i,1]\xrightarrow{\row}[n-i, i,0; n-i,i,0]\xrightarrow{\row}[0,n-i,i: 0,n-i,i].$$
Using Lemma \ref{calc_sc(I)} to calculate the signed cardinality statistic we find $\SC(I)+\SC(\row(I))+\SC(\row^2(I))$ is equal to $$\delta_i(-1)^{n-i}+\left(\delta_{i}(-1)^{n-i}+\delta_i(-1)^{n-i+1}\right)+\left(\delta_{n-i}(-1)^0+\delta_{n-i}(-1)\right)=\delta_i(-1)^{i+1}=\delta_i.$$  
Here the second to last equality makes use of the assumption $n$ is odd, and the last equality makes use of the fact that $\delta_{i}$ is zero unless $i$ is odd.  Going forward, we will frequently use similar deductions to simplify expressions involving the parity delta function, often without explanation.

Finally, if $a=0,$ our ICS $I$ is of the form $[\varnothing:n-i,i,0]$ and by Cases 4 and 5 of the proof of Proposition \ref{prop:High-to-Low}, 
$$[\varnothing:n-i,i,0]\xrightarrow{\row}[0,n,0:0, n-i,i]\xrightarrow{\row}[1,n-i,i-1: 0,n-i+1,i-1].$$
Once again using Lemma \ref{calc_sc(I)} to calculate the signed cardinality statistic we find $$\SC(I)+\SC(\row(I))+\SC(\row^2(I))=\delta_i(-1)^{n-i+1}+\left(\delta_{n}+\delta_{n-i}(-1)\right)+\left(\delta_{n-i}(-1)+\delta_{n-i+1}(-1)\right).$$
Given $n$ is odd, this simplifies to 
$$-\delta_i+(1-\delta_{i+1})+(-\delta_{i+1}-\delta_i)=-1.$$

\medskip

\noindent(Case 2: Low (type 5) ICS) Consider the ICS $I=[n-i-a,i,a:\varnothing].$  

If $a=0,$ then by Lemma \ref{Lem: CatalogRowmotion2xn}, $\row(I)$ is a High ICS, and we need only to calculate $\SC(I).$ By Lemma \ref{calc_sc(I)}, $\SC(I)=\delta_{i}(-1)^{n-i},$ which is equal to $\delta_{a+i}(-1)^a$ when $n$ is odd and $a=0.$ 

If $a\ge 1,$ then by the proof of Proposition \ref{prop:Low-to-High}, 
$$[n-i-a,i,a:\varnothing]\xrightarrow{\row}[n-i-a+1,i,a-1: 0,n-i-a+1,i+a-1]\xrightarrow{\row^{a-1}}[n-i,i,0: a-1,n-i-a+1, i]$$ where each of the $a-1$ applications of rowmotion shifts the ICS up one rank. One additional application of rowmotion returns the High ICS $\row^{a+1}(I)=[\varnothing:a,n-i-a+1,i-1].$  By Proposition \ref{shift_sc}, if $\row(I)$ is $I$ shifted up one rank then $\SC(I)+\SC(\row(I))=0.$ Thus, applying Lemma \ref{calc_sc(I)} and summing the ICS appearing in the partial orbit 
we find 
\begin{equation*}\label{eqn:sc_2}
 \sum_{k=0}^{a}\SC(\row^k(I))=\delta_i(-1)^{n-i-a}+\left(\delta_i(-1)^{n-i-a+1}+\delta_{n-i-a+1}(-1)\right)(a\bmod{2}).
\end{equation*}
Using the fact $n$ is odd and examining the cases when $a$ is even and odd separately, we find 
\begin{equation*}
\sum_{k=0}^{a}\SC(\row^k(I))=
\begin{cases}
        \delta_{i}  & \text{if } a\text{ is even,}\\
        -\delta_{i+1}  & \text{if } a \text{ is odd},
    \end{cases}
\end{equation*}
 or equivalently 
 $\sum_{k=0}^a\SC(\row^k(I))=\delta_{a+i}(-1)^a.$
\end{proof}

Building on Propositions  \ref{prop:Low-to-High} and 
 \ref{prop:High-to-Low} and Lemma \ref{lem:sc_High-and-Low}, the following lemma explains the dependence on the values of $x$ and $n\pmod{6}$ in the characterization of orbits with singleton representatives $\{(1,x)\}.$

\begin{lem}\label{prop:big orbits 2n+12}  Fix positive integers $n, x$ with $n\geq 8$ and $1<
 x \leq n-6$.  Applying rowmotion $2n + 12$ times to $I = \{ (1, x) \}$ results in the ICS $\{ (1, x + 6)\},$ and among the intermediate ICS, there are no other singletons on the lower chain and only one singleton, $\{(2,x+3)\}$, on the upper chain.   When $n$ is odd, the sum of the signed cardinality statistic over this substructure is $$\sum_{k=0}^{2n+11}\SC(\row^k(I))=(-1)^x.$$ 
\end{lem}

\begin{proof}
By Propositions \ref{prop:Low-to-High} and \ref{prop:High-to-Low} and Lemma \ref{lem:sc_High-and-Low}, when $n$ is odd and $x=2$, we have the following orbit substructures:
      
     $$\{(1,2)\}=\underbrace{[1,1,n-2:\varnothing]\xrightarrow[\ref{prop:Low-to-High}]{\row^{n-1}}}_{\delta_{n-1}(-1)^{n-2}}\underbrace{[\varnothing:n-2,2,0]\xrightarrow[\ref{prop:High-to-Low}c(i)]{\row^3}}_{-1}\underbrace{[2,n-2,0:\varnothing]\xrightarrow[\ref{prop:Low-to-High}]{\row}}_{\delta_{n-2}(-1)^0}\underbrace{[\varnothing:0,3,n-3]\xrightarrow[\ref{prop:High-to-Low}a]{\row^3}}_{0}$$ 
     $$\underbrace{[3,3,n-6:\varnothing]\xrightarrow[\ref{prop:Low-to-High}]{\row^{n-5}}}_{\delta_{n-3}(-1)^{n-6}}\underbrace{[\varnothing:n-6,4,2]\xrightarrow[\ref{prop:High-to-Low}b]{\row^{3}}}_{0}\underbrace{[0,n-4,4:\varnothing]\xrightarrow[\ref{prop:Low-to-High}]{\row^{5}}}_{\delta_n(-1)^{4}}\underbrace{[\varnothing:4,1,n-5]\xrightarrow[\ref{prop:High-to-Low}a]{\row^{3}}}_{0}\{(1,8)\}.$$

    Here we have recorded the partial sums of the signed cardinality statistic, given by Lemma \ref{lem:sc_High-and-Low} for odd $n,$ under the corresponding orbit substructure, with $\sum_{k=0}^{n-2}\SC(\row^{k}(\{(1,2)\}))=\delta_{n-1}(-1)^{n-2}$ while $\sum_{k=n-1}^{n+1}\SC(\row^{k}(\{(1,2)\}))=-1,$ and so on.
Adding these partial sums of the signed cardinality statistic over the entire substructure, and using the assumption that $n$ is odd and the definition of the parity indicator function $\delta$ to simplify, we find 
$$\sum_{k=0}^{2n+11}\SC(\row^k(\{(1,2)\}))=0-1+1+0+0+0+1+0=1  $$
which matches the formula in the Proposition evaluated at $x=2.$

For $x>2,$ Propositions \ref{prop:Low-to-High} and 
 \ref{prop:High-to-Low} and Lemma \ref{lem:sc_High-and-Low} give the following orbit substructures and signed cardinality statistic sums:
    
 $$\{(1,x)\}=\underbrace{[x-1,1,n-x:\varnothing]\xrightarrow[\ref{prop:Low-to-High}]{\row^{n-x+1}}}_{\delta_{n-x+1}(-1)^{n-x}}\underbrace{[\varnothing:n-x,x,0]\xrightarrow[\ref{prop:High-to-Low}c(i)]{\row^{3}}}_{-1}\underbrace{[2,n-x,x-2:\varnothing]\xrightarrow[\ref{prop:Low-to-High}]{\row^{x-1}}}_{\delta_{n-2}(-1)^{x-2}}$$
    $$\underbrace{[\varnothing:x-2,3,n-x-1]\xrightarrow[\ref{prop:High-to-Low}a]{\row^3}}_{0}\underbrace{[x+1,3,n-x-4:\varnothing]\xrightarrow[\ref{prop:Low-to-High}]{\row^{n-x-3}}}_{\delta_{n-x-1}(-1)^{n-x-4}}\underbrace{[\varnothing: n-x-4,x+2,2]\xrightarrow[\ref{prop:High-to-Low}b]{\row^3}}_{0}$$
    $$\underbrace{[0,n-x-2,x+2:\varnothing]\xrightarrow[\ref{prop:Low-to-High}]{\row^{x+3}}}_{\delta_n(-1)^{x+2}}\underbrace{[\varnothing:x+2,1,n-x-3]\xrightarrow[\ref{prop:High-to-Low}a]{\row^3}}_{0}\{(1,x+6)\}.$$

Adding the partial sums of the signed cardinality statistic together, and once again using the assumption $n$ is odd and the definition of $\delta$ to simplify, we find 
$$\sum_{k=0}^{2n+11}\SC(\row^k(\{1,x\}))=\delta_x-1+(-1)^x+0+\delta_x+0+(-1)^x+0,$$
which is equal to $1-1-1+0+1+0-1+0=-1$ when $x$ is odd and $0-1+1+0+0+0+1+0=1$ when $x$ is even.  Thus, in all cases, \[\sum_{k=0}^{2n+11}\SC(\row^k(\{1,x\}))=(-1)^x.\qedhere\]
 \end{proof} 

To completely characterize the orbits with singleton representatives, we establish how singletons are mapped to singletons under rowmotion for $I=\{(1,n-y)\}$ with $0\leq y\leq 5.$ This is done in Lemma \ref{prop:big orbits edge cases}, the proof of which uses the following substructure results for ICS consisting of two consecutive ranks in the lower chain.  

\begin{lem}\label{prop:big orbit 2phase} Let $n\geq 6$ and consider ICS consisting of two consecutive ranks in the lower chain, i.e., ICS of the form $I=[n-a-2,2,a:\varnothing]=\{(1,n-a-1),(1,n-a)\}.$ 
\begin{enumerate}
    \item[i.] For $3\leq a\leq n-2$, no lower-chain singletons appear in the orbit of $I$ between $I$ and 
    $$\row^{n+6}(I)=[n-a+1,2,a-3:\varnothing]=\{(1,n-a+2),(1,n-a+3)\}.$$
    \item[ii.] For $a=1$ and $n\equiv 1\bmod{3}$, there are no lower-chain singletons in the orbit of $I$ and $$|\O(I)|=\frac{n^2+2n-9}{3}.$$
    \item[iii.] For all other values of $a$ and $n$, the first appearance of a lower-chain singleton in the orbit of $I$  is at \begin{equation*}
\row^m(I)=
\begin{cases}
        \{(1,2)\}  & \text{where $m=4$ when $a=0,$ and $m=\frac{n^2+3n+9}{3}$ when $a=1$ and $n\equiv 0\bmod{3}$},\\
        \{(1,6)\} & \text{where $m=n+12$ when $a=2,$ and $m=\frac{n^2+4n+21}{3}$ when $a=1$ and $n\equiv 2\bmod{3}$}.
    \end{cases}
\end{equation*}
\end{enumerate}

\noindent In almost all cases, the sum of the signed cardinality statistic over the examined substructure is $0$ when $n$ is odd.  In particular, if $\row^m$ is the next lower-chain ICS consisting of two consecutive ranks, or the first lower-chain singleton, appearing in the orbit of $I,$ then for odd $n$
\begin{equation*}
        \sum_{k=0}^{m-1}\SC(\row^k(I))=
            \begin{cases}  
             1  & \text{if $a=2$,}\\
             1 & \text{if $a=1$ and $n\equiv 2\bmod{3}$},\\
             0 & \text{otherwise}.
    \end{cases}
\end{equation*}
 
 \end{lem}

\begin{proof} For $I=[n-a-2,2,a:\varnothing]$ we consider the four cases: $3\leq a\leq n-2, a=0, a=2,$ and $a=1.$

\medskip
\noindent(Case 1: $3\leq a\leq n-2$) 
     Applying Propositions \ref{prop:Low-to-High} and \ref{prop:High-to-Low} and Lemma \ref{lem:sc_High-and-Low}, we have the following orbit structures and signed cardinality statistic sums: $$\underbrace{[n-a-2,2,a:\varnothing]\xrightarrow[\ref{prop:Low-to-High}]{\row^{a+1}}}_{\delta_{a+2}(-1)^a}\underbrace{[\varnothing:a,n-a-1,1]\xrightarrow[\ref{prop:High-to-Low}b]{\row^3}}_{\delta_{n-a-1}}\underbrace{[1,a+1,n-a-2:\varnothing]\xrightarrow[\ref{prop:Low-to-High}]{\row^{n-a-1}}}_{\delta_{n-1}(-1)^{n-a-2}}$$
     $$\underbrace{[\varnothing:n-a-2,2,a]\xrightarrow[\ref{prop:High-to-Low}a]{\row^3}}_{0}[n-a+1,2,a-3:\varnothing].$$
    For each Low ICS appearing in this substructure, the cardinality is $ i=2$ or $a+1.$  Given $a\geq 3,$ no lower-chain singleton interval-closed sets appear.  Adding the partial sums of the signed cardinality statistic together, and using the assumption that $n$ is odd and the properties of $\delta$ to simplify, we find 
$\sum_{k=0}^{n+5}\SC(\row^k(I))=-\delta_a+\delta_a+0+0=0.$

\medskip
\noindent(Case 2: $a=0$) 
Consider $I=[n-2,2,0:\varnothing].$ Again applying Propositions \ref{prop:Low-to-High} and \ref{prop:High-to-Low} and Lemma \ref{lem:sc_High-and-Low}, we have
$$\underbrace{[n-2,2,0:\varnothing]\xrightarrow[\ref{prop:Low-to-High}]{\row}}_{\delta_{2}(-1)^0}\underbrace{[\varnothing:0,n-1,1]\xrightarrow[\ref{prop:High-to-Low}b]{\row^3}}_{\delta_{n-1}}[1,1,n-2:\varnothing]=\{(1,2)\}$$  with no other High or Low ICS appearing in this substructure.  Thus $\{(1,2)\}$ is the first singleton in the orbit of $I$ and the sum of the signed cardinality statistic over this substructure when $n$ is odd is 
$\sum_{k=0}^{3}\SC(\row^k(I))=0.$

\medskip

\noindent(Case 3: $a=2$)
Starting with $I=[n-4,2,2:\varnothing],$ we find

$$\underbrace{[n-4,2,2:\varnothing]\xrightarrow[\ref{prop:Low-to-High}]{\row^3}}_{\delta_{4}(-1)^2}\underbrace{[\varnothing:2,n-3,1]\xrightarrow[\ref{prop:High-to-Low}b]{\row^3}}_{\delta_{n-3}}\underbrace{[1,3,n-4:\varnothing]\xrightarrow[\ref{prop:Low-to-High}]{\row^{n-3}}}_{\delta_{n-1}(-1)^{n-4}}$$
$$\underbrace{[\varnothing:n-4,2,2]\xrightarrow[\ref{prop:High-to-Low}b]{\row^3}}_{0}\underbrace{[0,n-2,2:\varnothing]\xrightarrow[\ref{prop:Low-to-High}]{\row^3}}_{\delta_{n}(-1)^2}\underbrace{[\varnothing:2,1,n-3]=\{(2,3)\}\xrightarrow[\ref{prop:High-to-Low}a]{\row^3}}_{0}[5,1,n-6:\varnothing]=\{(1,6)\}.$$
By Propositions \ref{prop:Low-to-High} and \ref{prop:High-to-Low}, $\{(2,3)\}$ and $\{(1,6)\}$ are the only singletons appearing in this substructure.  Adding the partial sums of the signed cardinality statistic together for odd $n$ and simplifying, we find
$\sum_{k=0}^{11}\SC(\row^k(I))=0+0-0+0+1+0=1.$

\medskip
\noindent(Case 4: $a=1$)
Given $I=[n-3,2,1:\varnothing],$
$$\underbrace{[n-3,2,1:\varnothing]\xrightarrow[\ref{prop:Low-to-High}]{\row^2}}_{\delta_3(-1)}\underbrace{[\varnothing:1,n-2,1]\xrightarrow[\ref{prop:High-to-Low}b]{\row^3}}_{\delta_{n-2}}[1,2,n-3:\varnothing]$$
Since $n\geq 6,$ we can apply Case 1 of this proof $\lfloor\frac{n-3}{3}\rfloor$ times, 
$$[1,2,n-3:\varnothing]\xrightarrow[\ref{prop:big orbit 2phase}i]{\row^{\lfloor\frac{n-3}{3}\rfloor(n+6)}}[x,2,n\bmod{3}:\varnothing],$$ where $x=n-2-(n\bmod{3}).$  
Applying the signed cardinality results from Case 1 of this proof, we have that for odd $n,$
$$\sum_{k=0}^{\left\lfloor\frac{n-3}{3}\right\rfloor(n+6)+4}\SC(\row^k(I))=-\delta_3+\delta_{n-2}+\sum_{i=1}^{\left\lfloor\frac{n-3}{3}\right\rfloor} 0=-1+1+0=0.$$

\begin{itemize}
    \item If $n\equiv 1\pmod{3},$ $[x,2,n\bmod{3}:\varnothing]=[n-3,2,1:\varnothing]$ and we are back where we started, demonstrating $\O(I)$ has no singleton representatives, $$|\O(I)|=5+\left\lfloor\frac{n-3}{3}\right\rfloor(n+6)=5+\frac{n-4}{3}(n+6)=\frac{n^2+2n-9}{3},$$ 
    and when $n$ is odd $\sum_{J\in \O(I)}\SC(J)=0.$ 
    
\item If $n\equiv 0\pmod{3},$ $[x,2,n\bmod{3}:\varnothing]=[n-2,2,0:\varnothing]$ and by Case 2 of this proof,  
$$\underbrace{[n-2,2,0:\varnothing]\xrightarrow{\row^4}}_{0}\{(1,2)\}.$$  Thus, $\row^{m}(I)=\{(1,2)\}$ is the first appearance of a singleton in the orbit of $I$, where $$m=5+\left\lfloor\frac{n-3}{3}\right\rfloor(n+6)+4=\frac{n^2+3n+9}{3},$$
and the sum of the signed cardinality statistic when $n$ is odd is $\sum_{k=0}^{m-1}\SC(\row^k(I))=0+0=0.$

\item Finally, if $n\equiv 2\pmod{3},$ $[x,2,n\bmod{3}:\varnothing]=[n-4,2,2:\varnothing]$ and by Case 3 of this proof,  
$$\underbrace{[n-4,2,2:\varnothing]\xrightarrow{\row^{n+12}}}_{1}\{(1,6)\}.$$  
Therefore $\row^{l}(I)=\{(1,6)\}$ is the first appearance of a singleton in the orbit of $I$, where $$l=5+\left\lfloor\frac{n-3}{3}\right\rfloor(n+6)+(n+12)=\frac{n^2+4n+21}{3},$$ and when $n$ is odd, $\sum_{k=0}^{l-1}\SC(\row^k(I))=0+1=1.$
\qedhere
\end{itemize}
\end{proof}

We are now ready to complete the characterization started in Lemma \ref{prop:big orbits 2n+12}, by looking at the action of rowmotion on singletons of the form $\{(1,n-y)\}$ for $0\leq y\leq 6.$  As we will show, the complexity of the orbit substructure appearing between ICS of this form and the next lower-chain singleton ICS depends on whether an ICS consisting of two consecutive ranks in the lower-chain appears in between, as happens when $y=2$ and $y=4.$

\begin{lem}\label{prop:big orbits edge cases}  Given the interval-closed set $\{(1,n-y)\}$,  for $n\geq 10$ and $0\leq y\leq 5,$ the next lower-chain singleton appearing in the orbit of $I,$ $\row^m(I),$ has the form $\{(1,x)\}$ 
 where $2\leq x\leq 7.$  In particular,

 \begin{itemize}
     \item if $y\notin \{2,4\}$ then $x$ is the unique value satisfying both $2\leq x\leq 7$ and $x\equiv(10-y)\bmod{6}.$ Also,
     \begin{equation*}
\begin{cases}
        6 &  \text{for } y=0,\\
        \left(\frac{y-1}{2}\right)n+4\left\lfloor\frac{2y}{3}\right\rfloor+5 & \text{for } y\in\{1,3,5\},
    \end{cases}
    \end{equation*}
and, for odd values of $n,$ 
\begin{equation*}
        \sum_{k=0}^{m-1}\SC(\row^k(I))=
\begin{cases}
      1  & \text{for } y=0\\
        -1 & \text{for } y\in \{1,3\},\\
        0 & \text{for } y=5.
    \end{cases}
    \end{equation*}
\item if $y=2$ or $y=4,$ then $x$ is the unique value satisfying both $2\leq x\leq 7$ and $x\equiv
\begin{cases}
        (10-y)\bmod{6}  &  \text{if } n\equiv 1\bmod{3},\\
        (10+y)\bmod{6} & \text{if } n\equiv 0,2\bmod{3}.
    \end{cases}$
Also,        
\begin{equation*}
        m=
\begin{cases}
      \frac{n^2+(2y-1)n+6(y+1)}{3}  &  \text{if } (y,n \bmod{3})\in \{(2,0),(4,2)\},\\
        \frac{n^2+2^{y-1}n+6(3y-5)}{3} & \text{if } (y,n \bmod{3})\in \{(2,1), (4,1)\},\\
        \frac{2n^2+(2y+1)n+(6y-3)}{3} & \text{if } (y,n \bmod{3}) \in \{(2,2), (4,0)\},
    \end{cases} 
\end{equation*}
and, for odd values of $n,$ 
\begin{equation*}
        \sum_{k=0}^{m-1}\SC(\row^k(I))=
\begin{cases}
      0  & \text{if } (y,n\bmod{3}) \in \{(4,0), (2,1),  (4,2)\},\\
        1 & \text{if } (y,n\bmod{3}) \in \{(2,0), (4,1), (2,2)\}.
    \end{cases}
    \end{equation*}
 \end{itemize}
\end{lem}

\begin{proof} We break the proof up into cases based on the six possible values of $y,$ dispensing first with the cases in which no ICS consisting of two consecutive ranks appear between $I=\{(1,n-y)\}$ and the next lower-chain singleton in the orbit of $I.$

\noindent(Case 1: $y=0$)  Given $I=\{(1,n)\}=[n-1,1,0:\varnothing]$, we have that $\{(1,4)\}$ is the next appearance of a lower-chain singleton with 
$$\underbrace{[n-1,1,0:\varnothing]\xrightarrow[\ref{prop:Low-to-High}]{\row}}_{\delta_1(-1)^0}\underbrace{[\varnothing:0,n,0]\xrightarrow[\ref{Lem: CatalogRowmotion2xn}\textrm{ Type }6d]{\row}}_{-\delta_n}\underbrace{[0,n,0:\varnothing]\xrightarrow[\ref{prop:Low-to-High}]{\row}}_{\delta_n(-1)^0}\underbrace{[\varnothing:0,1,n-1]\xrightarrow[\ref{prop:High-to-Low}a]{\row^3}}_{0}[3,1,n-4:\varnothing]=\{(1,4)\}.$$ Here, we have used Lemma \ref{calc_sc(I)} and Lemma \ref{lem:sc_High-and-Low} to calculate the corresponding partial sums of the signed cardinality statistic when $n$ is odd.  Adding these partial sums together, we find that, when $n$ is odd, 
$\sum_{k=0}^5\SC(\row^k(\{(1,n)\}))=1-1+1+0=1.$

\medskip
\noindent For the remaining cases, we note that for $y\geq1,$ 
$$\{(1,n-y)\}=\underbrace{[n-y-1,1,y:\varnothing]\xrightarrow[\ref{prop:Low-to-High}]{\row^{y+1}}}_{\delta_{y+1}(-1)^y}\underbrace{[\varnothing:y,n-y,0]\xrightarrow[\ref{prop:High-to-Low}c(i)]{\row^3}}_{-1}[2,y,n-y-2:\varnothing],$$
and thus 
    \begin{equation}\label{eqn: Big Orbits Edge Cases}
    \row^{y+4}(\{(1,n-y)\})=[2,y,n-y-2:\varnothing], \mbox{ with }\sum_{k=0}^{y+3}\SC(\row^k(\{(1,n-y)\}))=\delta_{y+1}-1.\end{equation}

\noindent (Case 2: $y=1$)  By Equation \eqref{eqn: Big Orbits Edge Cases} we have $\row^{y+4}(\{(1,n-1)\})=\row^{5}(\{(1,n-1)\})=[2,1,n-3:\varnothing]=\{(1,3)\}$.  Moreover, by Propositions \ref{prop:Low-to-High} and \ref{prop:High-to-Low}, this is the next appearance of a lower-chain singleton in the orbit of $I=\{(1,n-1)\}.$  Summing the signed cardinality statistic over the sub-orbit from $\{(1,n-1)\}$ to $\{(1,3)\}$ gives $\sum_{k=0}^{4}\SC(\row^k(\{(1,n-1)\}))=\delta_2(-1)-1=-1$ when $n$ is odd.

\medskip

\noindent (Case 3: $y=3$) Similarly, when $y=3$ we have by Equation \eqref{eqn: Big Orbits Edge Cases}, $\row^{y+4}(\{(1,n-3)\})=\row^{7}(\{(1,n-3)\})=[2,3,n-5:\varnothing],$ and subsequent actions of rowmotion give $$\underbrace{\{(1,n-3)\}\xrightarrow[\textrm{Equation }\eqref{eqn: Big Orbits Edge Cases}]{\row^{7}}}_{\delta_{4}-1}\underbrace{[2,3,n-5:\varnothing]\xrightarrow[\ref{prop:Low-to-High}]{\row^{n-4}}}_{\delta_{n-2}(-1)^{n-5}}\underbrace{[\varnothing:n-5,3,2]\xrightarrow[\ref{prop:High-to-Low}b]{\row^3}}_{0}\underbrace{[0,n-3,3:\varnothing]\xrightarrow[\ref{prop:Low-to-High}]{\row^4}}_{\delta_n(-1)^3}$$ $$\underbrace{[\varnothing:3,1,n-4]\xrightarrow[\ref{prop:High-to-Low}a]{\row^3}}_{0}[6,1,n-7:\varnothing]=\{(1,7)\},$$ which, by Propositions \ref{prop:Low-to-High} and \ref{prop:High-to-Low}, is the next appearance of a lower-chain singleton in the orbit of $\{(1,n-3)\}.$ Adding the partial sums of the signed cardinality statistic and simplifying, we get $$\sum_{k=0}^{n+12}\SC(\row^k(\{(1,n-3)\}))=-\delta_4-1+\delta_{n-2}+0-\delta_n+0$$ which is also $-1$ for odd values of $n.$

\medskip

\noindent (Case 4: $y=5$)  Again by Equation \eqref{eqn: Big Orbits Edge Cases}, $\row^{y+4}(\{(1,n-5)\})=\row^{9}(\{(1,n-5)\})=[2,5,n-7:\varnothing]$ and subsequent actions of rowmotion give $$\underbrace{\{(1,n-5)\}\xrightarrow[\textrm{Equation }\eqref{eqn: Big Orbits Edge Cases}]{\row^{9}}}_{\delta_{6}-1}\underbrace{[2,5,n-7:\varnothing]\xrightarrow[\ref{prop:Low-to-High}]{\row^{n-6}}}_{\delta_{n-2}(-1)^{n-7}}\underbrace{[\varnothing:n-7,3,4]\xrightarrow[\ref{prop:High-to-Low}a]{\row^3}}_{0}\underbrace{[n-4,3,1:\varnothing]\xrightarrow[\ref{prop:Low-to-High}]{\row^2}}_{\delta_4(-1)}$$
$$\underbrace{[\varnothing:1,n-3,2]\xrightarrow[\ref{prop:High-to-Low}b]{\row^3}}_{0}\underbrace{[0,3,n-3:\varnothing]\xrightarrow[\ref{prop:Low-to-High}]{\row^{n-2}}}_{\delta_n(-1)^{n-3}}\underbrace{[\varnothing:n-3,1,2]\xrightarrow[\ref{prop:High-to-Low}b]{\row^3}}_{0}$$ $$\underbrace{[0,n-1,1:\varnothing]\xrightarrow[\ref{prop:Low-to-High}]{\row^2}}_{\delta_n(-1)}\underbrace{[\varnothing:1,1,n-2]\xrightarrow[\ref{prop:High-to-Low}a]{\row^3}}_{0}[4,1,n-5:\varnothing]=\{(1,5)\}.$$ 
Thus, $\{(1,5)\}$ is the next lower-chain singleton in the orbit of $I.$ Summing over the partial sums of the signed cardinality statistic, we find that when $n$ is odd $\sum_{k=0}^{2n+16}\SC(\row^k(\{(1,n-5)\}))=0.$

\medskip

Plugging $y\in \{1,3,5\}$ into the formula $\left(\frac{y-1}{2}\right)n+4\left\lfloor\frac{2y}{3}\right\rfloor+5$, we get respectively $5, n+13,$ and $2n+17,$ the calculated orbit substructure lengths between $\{(1,n-y)\}$ and the next lower-chain singleton.

\medskip

Cases 5 and 6 are more complicated as the orbit substructures examined include interval-closed sets consisting of two consecutive ranks in the lower chain.  As we showed in Lemma \ref{prop:big orbit 2phase}, the resulting orbit substructures depend on the value of $n\bmod{3}.$

\medskip

\noindent (Case 5: $y=2$)  Given $I=\{(1,n-2)\},$  we have by Equation \eqref{eqn: Big Orbits Edge Cases}, $\row^{y+4}(I)=\row^{6}(I)=[2,2,n-4:\varnothing]$  with $\sum_{k=0}^{5}\SC(\row^k(\{(1,n-2)\})=\delta_{3}-1=0$. Subsequent actions of rowmotion give
$$\underbrace{[2,2,n-4:\varnothing]\xrightarrow[\ref{prop:Low-to-High}]{\row^{n-3}}}_{\delta_{n-2}(-1)^{n-4}}\underbrace{[\varnothing:n-4,3,1]\xrightarrow[\ref{prop:High-to-Low}b]{\row^3}}_{\delta_3}\underbrace{[1,n-3,2:\varnothing]\xrightarrow[\ref{prop:Low-to-High}]{\row^3}}_{\delta_{n-1}(-1)^2}\underbrace{[\varnothing:2,2,n-4]\xrightarrow[\ref{prop:High-to-Low}a]{\row^3}}_{0}[5,2,n-7:\varnothing].$$
Since $n\geq 10$, we can apply Lemma \ref{prop:big orbit 2phase} to find
$$[5,2,n-7:\varnothing]\xrightarrow[\ref{prop:big orbit 2phase}i]{\row^{\lfloor\frac{n-7}{3}\rfloor(n+6)}}[x,2,n-1\bmod{3}:\varnothing],$$
where $x=n-2-(n-1\bmod{3}).$ 
Before considering the three possible values of $n-1\bmod{3},$ we sum the signed cardinality statistic over the orbit substructure from $I$ to $[x,2,n-1\bmod{3}:\varnothing]$ when $n$ is odd.  By Lemma \ref{prop:big orbit 2phase}\textit{i}, the interval-closed sets from $[5,2,n-7:\varnothing]$ to $[x,2,n-1\bmod{3}:\varnothing]$ contribute nothing to the sum of the signed cardinality statistic in this case.  Using this, and the assumption that $n$ is odd to simplify the sum, we find 
\begin{equation}\label{eqn: sc y=2 partial}
\sum_{k=0}^{n+12+\lfloor\frac{n-7}{3}\rfloor(n+6)-1}\SC(\row^k(I))=0+(-1)+1+0+0+0=0.
\end{equation}

\begin{itemize}
    \item If $n\equiv 0\pmod{3},$ 
$$[x,2,n-1\bmod{3}:\varnothing]=[n-4,2,2:\varnothing]\xrightarrow[\ref{prop:big orbit 2phase}iii]{\row^{n+12}}\{(1,6)\},$$
and we have $\row^m(I)=\{(1,6)\}$ is the next appearance of a lower-chain singleton in the orbit of $I=\{(1,n-2)\},$ where 
$$m=n+12+\lfloor\frac{n-7}{3}\rfloor(n+6)+(n+12)=\frac{n^2+3n+18}{3}=\frac{n^2+(2y-1)n+6(y+1)}{3}.$$ 
By Lemma \ref{prop:big orbit 2phase}, this last piece of the orbit substructure, from $[n-4,2,2:\varnothing]$ to $\{(1,6)\},$ contributes $1$ to the sum of the signed cardinality statistic.  Adding this to the partial sum from Equation \eqref{eqn: sc y=2 partial}, we find 
\begin{equation*}
    \sum_{k=0}^{m-1}\SC(\row^k(\{(1,n-2)\}))=1
\end{equation*}
for odd values of $n.$

\item If $n\equiv 1\pmod{3},$
$$[x,2,n-1\bmod{3}:\varnothing]=[n-2,2,0:\varnothing]\xrightarrow[\ref{prop:big orbit 2phase}iii]{\row^{4}}\{(1,2)\},$$
and the next lower-chain singleton in the orbit of $I$ is $\row^m(I)=\{(1,2)\},$ where $$m=n+12+\frac{n-7}{3}(n+6)+4=\frac{n^2+2n+6}{3}=\frac{n^2+2^{y-1}n+6(3y-5)}{3}.$$ Applying Equation \eqref{eqn: sc y=2 partial} and Lemma \ref{prop:big orbit 2phase}, we find that for odd $n,$
\begin{equation*}
    \sum_{k=0}^{m-1}\SC(\row^k(\{(1,n-2)\}))=0.
\end{equation*}

\item If $n\equiv 2\pmod{3},$
$$[x,2,n-1\bmod{3}:\varnothing]=[n-3,2,1:\varnothing]\xrightarrow[\ref{prop:big orbit 2phase}iii]{\row^{\frac{n^2+4n+21}{3}}}\{(1,6)\}$$ is the next lower-chain singleton appearing in the orbit of $I$ after $$m=n+12+
\frac{n-8}{3}(n+6)+\frac{n^2+4n+21}{3}=\frac{2n^2+5n+9}{3}=\frac{2n^2+(2y+1)n+(6y-3)}{3} $$ applications of rowmotion.  Again, by Lemma \ref{prop:big orbit 2phase}, when $n$ is odd, this last portion of the orbit substructure contributes $1$ to the sum of the signed cardinality statistic, and 
\begin{equation*}
    \sum_{k=0}^{m-1}\SC(\row^k(\{(1,n-2)\}))=1.
\end{equation*}
\end{itemize}

\medskip

\noindent (Case 6: $y=4$)
Given $I=\{(1,n-4)\},$  we have by Equation \eqref{eqn: Big Orbits Edge Cases}, $\row^{y+4}(I)=\row^{8}(I)=[2,4,n-6:\varnothing]$  with $\sum_{k=0}^{7}\SC(\row^k(\{(1,n-4)\})=\delta_{5}-1=0$. Subsequent actions of rowmotion give
$$\underbrace{[2,4,n-6:\varnothing]\xrightarrow[\ref{prop:Low-to-High}]{\row^{n-5}}}_{\delta_{n-2}(-1)^{n-6}}\underbrace{[\varnothing:n-6,3,3]\xrightarrow[\ref{prop:High-to-Low}a]{\row^3}}_{0}\underbrace{[n-3,3,0:\varnothing]\xrightarrow[\ref{prop:Low-to-High}]{\row}}_{\delta_3(-1)^0}\underbrace{[\varnothing:0,n-2,2]\xrightarrow[\ref{prop:High-to-Low}b]{\row^3}}_{0}$$
$$\underbrace{[0,2,n-2:\varnothing]\xrightarrow[\ref{prop:Low-to-High}]{\row^{n-1}}}_{\delta_n(-1)^{n-2}}\underbrace{[\varnothing:n-2,1,1]\xrightarrow[\ref{prop:High-to-Low}b]{\row^3}}_{\delta_1}\underbrace{[1,n-1,0:\varnothing]\xrightarrow[\ref{prop:Low-to-High}]{\row}}_{\delta_{n-1}(-1)^0}\underbrace{[\varnothing:0,2,n-2]\xrightarrow[\ref{prop:High-to-Low}a]{\row^3}}_{0}[3,2,n-5:\varnothing].$$
Thus, $\row^{2n+16}(I)=[3,2,n-5:\varnothing].$ 
As in Case 5, we apply Lemma \ref{prop:big orbit 2phase} to find
$$[3,2,n-5:\varnothing]\xrightarrow[\ref{prop:big orbit 2phase}i]{\row^{\lfloor\frac{n-5}{3}\rfloor(n+6)}}[x,2,n+1\bmod{3}:\varnothing],$$ where $x=n-2-(n+1\bmod{3}).$ Furthermore, by Lemma \ref{prop:big orbit 2phase}, the sum of the signed cardinality statistic over the orbit substructure from $[3,2,n-5:\varnothing]$ to $[x,2,n+1\bmod{3}:\varnothing]$ is $0$ when $n$ is odd.  Adding this partial sum to $\sum_{k=0}^{2n+15}\SC(\row^k(I)),$ we find that for odd values of $n$, 
\begin{equation*}
    \sum_{k=0}^{2n+16+\left\lfloor\frac{n-5}{3}\right\rfloor(n+6)-1}\SC(\row^k(I))=0.
\end{equation*}

\begin{itemize}
    \item If $n\equiv 0\pmod{3},$
$$[x,2,n+1\bmod{3}:\varnothing]=[n-3,2,1:\varnothing]\xrightarrow[\ref{prop:big orbit 2phase}iii]{\row^{\frac{n^2+3n+9}{3}}}\{(1,2)\},$$ 
and $\{(1,2)\}$ is the next lower-chain singleton appearing in the orbit of $I.$ 
Summing over the applications of rowmotion, and using the fact that $\lfloor\frac{n-5}{3}\rfloor=\frac{n-6}{3}$ when $n\equiv 0\pmod{3},$ we find $\row^m(I)=\{(1,2)\}$ for $$m=(2n+16)+\frac{n-6}{3}(n+6)+\frac{n^2+3n+9}{3}=\frac{2n^2+9n+21}{3}=\frac{2n^2+(2y+1)n+(6y-3)}{3} .$$ 
Also by Lemma \ref{prop:big orbit 2phase}, the orbit substructure from $[n-3,2,1:\varnothing]$ to $\{(1,2)\}$ contributes $0$ to the sum of the signed cardinality statistic when $n$ is odd.  Thus, for odd values of $n,$ 
\begin{equation*}
    \sum_{k=0}^{m-1}\SC(\row^k(\{(1,n-4)\}))=0+0=0.
\end{equation*}
\item If $n\equiv 1\pmod{3},$ $$[x,2,n+1\bmod{3}:\varnothing]=[n-4,2,2:\varnothing]\xrightarrow[\ref{prop:big orbit 2phase}iii]{\row^{n+12}}\{(1,6)\},$$ and $\row^m(I)=\{(1,6)\}$ is the next lower-chain singleton appearing in the orbit of $I,$ with $$m=(2n+16)+\frac{n-7}{3}(n+6)+(n+12)=\frac{n^2+8n+42}{3}=\frac{n^2+2^{y-1}n+6(3y-5)}{3}.$$
Here, we have used the fact that $\lfloor\frac{n-5}{3}\rfloor=\frac{n-7}{3}$ when $n\equiv 1\pmod{3}.$  By Lemma \ref{prop:big orbit 2phase}, when $n$ is odd, this last piece of the orbit substructure contributes $1$ to the sum of the signed cardinality statistic, and we have 
$$\sum_{k=0}^{m-1}\SC(\row^k(\{(1,n-4)\}))=0+1=1,$$
for odd $n.$
\item If $n\equiv 2\pmod{3},$
$$[x,2,n+1\bmod{3}:\varnothing]=[n-2,2,0:\varnothing]\xrightarrow[\ref{prop:big orbit 2phase}iii]{\row^{4}}\{(1,2)\},$$ and the next lower-chain singleton in the orbit of $I$ is $\row^m(I)=\{(1,2)\}$ with $$m=(2n+16)+\frac{n-5}{3}(n+6)+4=\frac{n^2+7n+30}{3}=\frac{n^2+(2y-1)n+6(y+1)}{3}.$$ 
Again by Lemma \ref{prop:big orbit 2phase}, this last piece of the substructure contributes $0$ to the sum of the signed cardinality statistic when $n$ is odd.  Thus, for odd values of $n,$
\[\sum_{k=0}^{m-1}\SC(\row^k(\{(1,n-4)\}))=0+0=0. \qedhere\]
\end{itemize}
\end{proof}

Lemmas \ref{prop:big orbits 2n+12} and \ref{prop:big orbits edge cases} allow us to study the orbits of lower-chain singleton ICS by examining just the lower-chain singleton elements in the orbit.  Starting with any $I=\{(1,x)\}$ for $x>1$, we either step up through lower-chain singletons of the form $\{(1,x+6k)\}$ until we get one of the form $\{(1,a)\}$ with $a>n-6$, or we started with $x>n-6.$ Either way, by Lemma \ref{prop:big orbits edge cases}, we eventually map to a lower-chain singleton of the form $\{(1,b)\}$ with $2\leq b\leq 7$.  Thus, the orbit of every lower-chain singleton has a representative of the form $\{(1,b)\}$ for $b\leq 7$, and knowing which lower-chain singletons of this form are in a particular orbit is enough to calculate the length of the orbit and all lower-chain singletons in the orbit for a given $n.$  This is the basis of the following Theorem. 

\begin{thm}\label{large_orbits} 
  For each $n \geq 5$, there are rowmotion orbits in $[2] \times [n]$ of the following sizes, for a total of $3n^2+5n-30$ interval-closed sets: 
\begin{itemize}
    \item for $n \equiv 0 \pmod{6}$, there is an orbit of size 
    \begin{itemize}
        \item $n^2 + 3n - 5$, containing all lower-chain singleton ICS of the form $\{(1, i)\}$ for $i \equiv 2 \pmod{6}$ 
        \item $n^2 + n - 12$, containing all lower-chain singleton ICS of the form $\{(1, i)\}$ for $i \equiv 0, 4 \pmod{6}$
        \item $n^2 + n - 13$, containing all lower-chain singleton ICS of the form $\{(1, i)\}$ for $i \equiv 1,3,5 \pmod{6}$, $i > 1$ 
    \end{itemize}
    \item for $n \equiv 1 \pmod{6}$, there is an orbit of size 
    \begin{itemize}
        \item $n^2 + 2n - 9$, containing all lower-chain singleton ICS of the form $\{(1, i)\}$ for $i \equiv 2, 5 \pmod{6}$ 
        \item $n^2 + 2n - 9$, containing all lower-chain singleton ICS of the form $\{(1, i)\}$ for $i \equiv 0, 3 \pmod{6}$
        \item $\frac{2n^2 + n - 27}{3}$, containing all lower-chain singleton ICS of the form $\{(1, i)\}$ for $i \equiv 1, 4 \pmod{6}$, $i > 1$
        \item $\frac{n^2+2n-9}{3}$, with no singleton ICS, and with representative $\{(1,n-2),(1,n-1)\}$
    \end{itemize}
    \item for $n \equiv 2 \pmod{6}$, there is an orbit of size 
    \begin{itemize}
        \item $n^2 + 3n - 4$, containing all lower-chain singleton ICS of the form $\{(1, i)\}$ for $i \equiv 2,4 \pmod{6}$ 
        \item $n^2 + n - 13$, containing all lower-chain singleton ICS of the form $\{(1, i)\}$ for $i \equiv 0 \pmod{6}$
        \item $n^2 + n - 13$, containing all lower-chain singleton ICS of the form $\{(1, i)\}$ for $i \equiv 1, 3, 5 \pmod{6}$, $i > 1$, 
    \end{itemize}
    \item for $n \equiv 3 \pmod{6}$, there is an orbit of size 
    \begin{itemize}
        \item $2n^2 + 5n - 13$, containing all lower-chain singleton ICS of the form $\{(1, i)\}$ for $i \equiv 2, 3, 4, 5 \pmod{6}$
        \item $n^2 - 17$, containing all lower-chain singleton ICS of the form $\{(1, i)\}$ for $i \equiv 0, 1 \pmod{6}$, $i > 1$ 
    \end{itemize}
    \item for $n \equiv 4 \pmod{6}$, there is an orbit of size 
    \begin{itemize}
        \item $\frac{2n^2 + 4n - 18}{3}$, containing all lower-chain singleton ICS of the form $\{(1, i)\}$ for $i \equiv 0 \pmod{6}$ 
        \item $\frac{2n^2 + 4n - 18}{3}$, containing all lower-chain singleton ICS of the form $\{(1, i)\}$ for $i \equiv 2 \pmod{6}$ 
        \item $\frac{n^2 + 2n - 6}{3}$, containing all lower-chain singleton ICS of the form $\{(1, i)\}$ for $i \equiv 4 \pmod{6}$ 
        \item $\frac{n^2 + 2n - 9}{3}$, containing all lower-chain singleton ICS of the form $\{(1, i)\}$ for $i \equiv 3 \pmod{6}$ 
        \item $\frac{n^2 + 2n - 9}{3}$, containing all lower-chain singleton ICS of the form $\{(1, i)\}$ for $i \equiv 5 \pmod{6}$ 
        \item $\frac{n^2 - n - 21}{3}$, containing all lower-chain singleton ICS of the form $\{(1, i)\}$ for $i \equiv 1 \pmod{6}$, $i > 1$ 
        \item $\frac{n^2+2n-9}{3}$, with no singleton ICS, and with representative $\{(1,n-2),(1,n-1)\}$
            \end{itemize}
    \item and for $n \equiv 5 \pmod{6}$, there is an orbit of size 
    \begin{itemize}
        \item $2n^2 + 3n - 21$, containing all lower-chain singleton ICS of the form $\{(1, i)\}$ for $i \equiv 0, 3, 4, 5 \pmod{6}$ 
               \item $n^2 + 2n - 9$, containing all lower-chain singleton ICS of the form $\{(1, i)\}$ for $i \equiv 1, 2 \pmod{6}$, $i > 1$. 
          \end{itemize}
    \end{itemize}
    \end{thm}

\begin{proof}
Consider first the posets $[2]\times[n]$ for $5 \leq n \leq 9$. The explicit orbit sizes are given in Table \ref{tab:orbit_lengths_small_posets}, allowing for the verification of the statement.

Now suppose that $n\geq 10$.
Consider the orbit, $\O(I),$ of $I=\{(1,i)\}$ for $1<i\leq 7.$  By Lemma \ref{prop:big orbits 2n+12}, $\O(I)$ contains all the interval-closed sets of the from $\{(1,a)\}$ where $a\equiv i\pmod{6}$ and $a>1.$  In particular, after $w=\lfloor \frac{n-i}{6}\rfloor (2n+12)$ applications of rowmotion, we have $\row^w(I)=\{(1,n-y)\}$ for $y=n-x\bmod{6}.$  As $n-y>n-6,$ the next lower-chain singleton appearing in the orbit of $I$ is determined by Lemma \ref{prop:big orbits edge cases}, and is of the form $\{(1,b)\}$ for some $1< b\leq 7.$ The process continues, sweeping in all lower-chain singleton ICS of the form $\{(1,t)\}$ with $t\equiv b\pmod{6}.$  In essence, Lemmas \ref{prop:big orbits 2n+12} and \ref{prop:big orbits edge cases} allow us to keep track of all lower-chain singleton elements in the orbit of an ICS by simply keeping track of the lower-chain singletons of the form $\{(1,i)\}$ for $2\leq i\leq 7$ and $\{(1,n-y)\}$ for $0\leq y\leq 5.$  As an example, consider the orbit of $\{(1,3)\}$ when $n\equiv 0\pmod{6}:$ $$\{(1,3)\}\xrightarrow[\ref{prop:big orbits 2n+12}]{\row^{\lfloor \frac{n-3}{6}\rfloor (2n+12)}}\{(1,n-(n-3\bmod{6}))\}=\{(1,n-3)\}\xrightarrow[\ref{prop:big orbits edge cases}]{\row^{n+13}}\{(1,7)\}$$
$$\xrightarrow[\ref{prop:big orbits 2n+12}]{\row^{\lfloor \frac{n-7}{6}\rfloor (2n+12)}}\{(1,n-(n-7\bmod{6}))\}=\{(1,n-5)\}\xrightarrow[\ref{prop:big orbits edge cases}]{\row^{2n+17}}\{(1,5)\}$$
$$\xrightarrow[\ref{prop:big orbits 2n+12}]{\row^{\lfloor \frac{n-5}{6}\rfloor (2n+12)}}\{(1,n-(n-5\bmod{6}))\}=\{(1,n-1)\}\xrightarrow[\ref{prop:big orbits edge cases}]{\row^{5}}\{(1,3)\}.$$ 
Using the fact that $\lfloor \frac{n-i}{6}\rfloor (2n+12)=\frac{(n-6)(2n+12)}{6}$ for all positive $i\leq 6$ and $\frac{(n-12)(2n+12)}{6}$ for $i=7,$ we sum over the applications of rowmotion to find that when $n\equiv 0\bmod{6}$, there is an orbit of size $$\frac{(n-6)(2n+12)}{6}+(n+13)+\frac{(n-12)(2n+12)}{6}+(2n+17)+\frac{(n-6)(2n+12)}{6}+5=n^2+n-13$$ containing all lower-chain singletons of the form $\{(1,i)\}$ for $i>1$ and equivalent to $1,3$ or $5\pmod{6}.$

We can use Lemmas \ref{prop:big orbits 2n+12} and \ref{prop:big orbits edge cases} to define a permutation on the elements $S=\{2,3,4,5,6,7,n,n-1,n-2,n-3,n-4,n-5\}$ for each residue class of $n\bmod{6}.$  From this perspective, the sequence of lower-chain representatives in the above argument, $3\rightarrow (n-3)\rightarrow 7\rightarrow (n-5)\rightarrow 5\rightarrow (n-1) \rightarrow 3,$ corresponds to a subcycle of the permutation.  Finding this permutation $\sigma_{(n\bmod{6})}$ and a corresponding function $w_{(n\bmod{6})},$ we can find the lengths of all orbits containing a lower-chain singleton $\{(1,x)\}$ for $x>1.$ 

To make this precise, for all $s\in S,$ define $\sn(s)$ to be the unique value in $S$ satisfying  \begin{equation}\label{eq:sigma}
        \sn(s)=
\begin{cases}
        n-(n-s\bmod{6})  &  \text{if } 2\leq s\leq 7,\\
        (10+(n-s))\bmod{6} & \text{if $s\in \{n-2,n-4\}$ and $n\equiv 0,2\bmod 3$},\\
        (10-(n-s))\bmod{6}& \text{otherwise},
\end{cases}
\end{equation}
and $\wn(s)$ to be the smallest positive value such that $\row^{\wn(s)}(\{(1,s)\})=\{(1,\sn(s))\}.$  By Lemmas \ref{prop:big orbits 2n+12} and \ref{prop:big orbits edge cases},
\begin{equation*}
        \wn(s)=
\begin{cases}
        \lfloor \frac{n-s}{6}\rfloor (2n+12)  &  \text{if } 2\leq s\leq 7,\\
        6& \text{if $s=n$},\\
        \left(\frac{y-1}{2}\right)n+4\left\lfloor\frac{2y}{3}\right\rfloor+5 & \text{if $s=n-y$ for $y\in\{1,3,5\}$},\\
        \frac{n^2+(2y-1)n+6(y+1)}{3}  &  \text{if $s=n-y$ and } (y,n \bmod{3})\in\{(2,0), (4,2)\},\\
        \frac{n^2+2^{y-1}n+6(3y-5)}{3} & \text{if $s=n-y$ and } (y,n \bmod{3})\in\{(2,1), (4,1)\},\\
        \frac{2n^2+(2y+1)n+(6y-3)}{3} & \text{if $s=n-y$ and } (y,n \bmod{3})\in\{(2,2), (4,0)\},
    \end{cases} 
\end{equation*}
and $\{(1,a)\}$ is in the orbit of some lower-chain ICS $\{(1,b)\}$ if and only if $a\equiv s_i\pmod{6}$ and $b\equiv s_j\pmod{6}$ for some $s_i,s_j$ in the same subcycle of $\sn.$  By our construction of $\wn,$ the size of this orbit
is $$|\O(\{(1,a)\})|=\sum_{s\in C_a}\wn(s),$$ where $C_a$ is the subcycle of $\sn$ containing an element congruent to $a\pmod{6}.$ 

We will calculate the orbit structures of lower-chain singleton ICS of the form $\{(1,i)\}$ with $2\leq i\leq 7$ in detail for $n\equiv 0\bmod{6}.$
Calculating the orbit structures for other values of $n$ follows similarly.

\medskip

\noindent (Case 0: $n\equiv 0\pmod{6}$) 
Writing $\sigma_0$ in cycle form, we have $$\sigma_0=(2,n-4)(3,n-3,7,n-5,5,n-1)(4,n-2,6,n).$$  
The first subcycle tells us $\row^{w_0(n-4)}(\{(1,n-4)\})=\row^{w_0(n-4)}(\row^{w_0(2)}(\{(1,2)\}))=\{(1,2)\}$ and we have an orbit of length
$$\sum_{s\in C_2}w_0(s)=w_0(2)+w_0(n-4)=\frac{(n-6)(2n+12)}{6}+\frac{2n^2+9n+21}{3}=n^2+3n-5$$ containing all lower-chain singletons of the form $\{(1,i)\}$ for $i\equiv 2\pmod{6}.$  The second subcycle gives us the orbit of length $n^2+n-13$ containing all lower-chain singletons of the form $\{(1,i)\}$ for $i>1$ and equivalent to $1,3$ or $5\pmod{6}$ examined at the beginning of this proof, and the last subcycle implies there is an orbit of length 
\begin{align*}
    \sum_{s\in C_4}w_0(s)&=w_0(4)+w_0(n-2)+w_0(6)+w_0(n)\\
    &=\frac{(n-6)(2n+12)}{6}+\frac{n^2+3n+18}{3}+\frac{(n-6)(2n+12)}{6}+6=n^2+n-12
\end{align*}
containing all lower-chain singletons of the form $\{(1,i)\}$ for $i\equiv 0,4\pmod{6}.$

\medskip

\noindent (Case $n\not\equiv 0\pmod{6}$) To find the orbit structures of lower-chain singletons when $n\not\equiv 0\pmod{6}$ we can calculate $\sn$ directly from Equation \ref{eq:sigma}, or from $\sigma_0$ using the following observations.  First, for $2\leq s\leq 7$, the image of $\sigma_r$ is the image of $\sigma_0$ cyclically shifted by $r.$  Thus, for $2\leq s\leq 7$, $$\sigma_r(s)=\left((n,n-1,n-2,n-3,n-4,n-5)^r\circ\sigma_0\right)(s).$$
Second, for $n-5\leq s\leq n,$ $\sigma_{(n\bmod{6})}(s)=\sigma_0(s)$ except when $s=n-2$ or $s=n-4.$  In these cases, if $n\equiv 1\bmod{3},$ $\sigma_{(n\bmod{6})}(s)=\left((2,6)\circ\sigma_0\right)(s).$
Putting these observations together, and letting $r=n\pmod{6},$ we find 
$$\sigma_r=(2,6)^{(r\bmod 3)}\circ(n,n-1,n-2,n-3,n-4,n-5)^r\circ\sigma_0.$$ 
In particular, 
\begin{align}
    \sigma_1&=(2,n-5,5,n-2)(3,n-4,6,n-1)(4,n-3,7,n),\nonumber\\
    \sigma_2&=(2,n,4,n-4)(3,n-5,5,n-3,7,n-1)(6,n-2),\nonumber\\
\sigma_3&=(6, n-3, 7, n-2) (2, n-1, 3, n, 4, n-5, 5, n-4)\label{eq:big-orbits sigmas}\\
\sigma_4&=(6, n-4) (7, n-3) (2, n-2)(3, n-1) (4, n)(5, n-5),\text{ and }\nonumber\\
\sigma_5&=(6, n-5, 5, n, 4, n-1, 3, n-2)(7, n-4, 2, n-3).\nonumber
\end{align}

The stated results for orbits with lower-chain singleton representatives then follow for each residue class of $n\bmod{6}$ by calculating $\sum_{s\in C_x}\wn(s)$ for each subcycle $C_x$ of $\sn$ and making note of the elements $s\in C_x$ such that $2\leq s\leq 7.$ 

By Lemma \ref{prop:big orbit 2phase}, for $n\equiv 1, 4\bmod{6},$ the orbit of $I=\{(1,n-2),(1,n-1)\}$ contains no lower-chain singleton ICS and has size $|\O(I)|=\frac{n^2+2n-9}{3},$ completing the proof.
\end{proof}

\begin{thm}\label{lem:sc-big-orbits}
Let $P = [2] \times [n]$ with $n$ odd. The average value of signed cardinality over the orbits with lower-chain singleton representatives is 0. 
\end{thm}

\begin{proof}
Let $n$ be odd.  By Theorem \ref{thm:sc_n+3}, $\sum_{J\in \O(\{(1,1)\})}\SC(J)=0.$ To prove $\sum_{J\in \O(I)}\SC(J)=0$ for ICS of the form $\{(1,i)\}$ with $i>1,$ we focus on orbit representatives of the form $\{(1,s)\}$ for $s\in S=\{2,3,4,5,6,7,n,n-1,n-2,n-3,n-4,n-5\}.$  Building on the proof of Theorem \ref{large_orbits}, we calculate the partial sums of the signed cardinality statistic between such representatives when $n$ is odd and then add these partial sums together to show the sum over the orbit is $0.$  We manually verified the result for $n=3,5,7,\text{ and } 9.$
    
    Recall the following from the proof of Theorem \ref{large_orbits},
    \begin{itemize}
        \item $\sn$ is the permutation of $S$ defined by setting $\sn(s)=x,$ where $\{(1,x)\}$ is the next lower-chain singleton in the orbit of $\{(1,s)\}$ having $x\in S,$
        \item $\wn(s)$ is the smallest positive number such that $\row^{\wn(s)}(\{(1,s)\})=\{(1,\sn(s))\},$ and
        \item each orbit having a lower-chain singleton representative $\{(1,i)\}$ for $i>1$ corresponds to a subcycle $C_i$ of $\sn$ containing an element $s\equiv i\pmod{6}.$ 
    \end{itemize}
    Given these definitions, the sum of the signed cardinality statistic over the orbit of $I=\{(1,i)\}$ for $i>1$ is given by 
    \begin{equation}\label{eq:sc-quad_orbit_formula}
      \sum_{J\in \O(I)}\SC(J)=\sum_{s\in C_i}\left(\sum_{k=0}^{\wn(s)-1}\SC(\row^k(\{(1,s)\}))\right).  
    \end{equation}
    We proceed by collecting the relevant partial sums of the signed cardinality statistic. 
    \begin{itemize}
        \item For $2\leq s\leq 7$, $\wn(s)=\left\lfloor \frac{n-s}{6}\right\rfloor(2n+12)$ and repeated applications of Lemma \ref{prop:big orbits 2n+12} give 
\begin{equation}\label{eq:sc-partial-2n+12}
    \sum_{k=0}^{\wn(s)-1}\SC(\row^k(\{(1,s)\}))=\sum_{t=0}^{\left\lfloor \frac{n-s}{6}\right\rfloor-1}(-1)^{s+6t}=\left\lfloor\frac{n-s}{6}\right\rfloor(-1)^s.
\end{equation}
\item For $n-5\leq s\leq n,$ let $y=n-s.$  By Lemma \ref{prop:big orbits edge cases}, 
\begin{equation}\label{eq:sc-partial-edge}
        \sum_{k=0}^{\wn(s)-1}\SC(\row^k(\{(1,s)\}))=
\begin{cases}           
        1 & \text{if $y=0$ or } (y,n\bmod{3})\in\{(2,0), (4,1), (2,2)\},\\
        -1 & \text{if } y\in\{1,3\},\\
      0  & \text{if $y=5$ or } (y,n\bmod{3})\in\{(4,0), (2,1), (4,2)\}.
\end{cases}
    \end{equation}
    \end{itemize}
We now prove the result by applying Equation \eqref{eq:sc-quad_orbit_formula} to the subcycles $C_i$ of $\sn$ for $n$ odd.

\medskip

\noindent (Case 1: $n\equiv 1\bmod{6}$) By Equation \eqref{eq:big-orbits sigmas}, $$\sigma_1=(2,n-5,5,n-2)(3,n-4,6,n-1)(4,n-3,7,n).$$
Calculating the sum of the signed cardinality statistic over the orbit of $\{(1,i)\}$ by first calculating the partial sums between consecutive elements of the subcycle $C_i$ for $2\leq i\leq 7,$ we find  
$$C_2=C_5:\underbrace{\{(1,2)\}\xrightarrow[]{\row^{w_1(2)}}}_{\left\lfloor\frac{n-2}{6}\right\rfloor(-1)^2}\underbrace{\{(1,n-5)\}\xrightarrow[]{\row^{w_1(n-5)}}}_{0}\underbrace{\{(1,5)\}\xrightarrow[]{\row^{w_1(5)}}}_{\left\lfloor\frac{n-5}{6}\right\rfloor(-1)^5}\underbrace{\{(1,n-2)\}\xrightarrow[]{\row^{w_1(n-2)}}}_{0}\{(1,2)\},$$

$$C_3=C_6: \underbrace{\{(1,3)\}\xrightarrow[]{\row^{w_1(3)}}}_{\left\lfloor\frac{n-3}{6}\right\rfloor(-1)^3}\underbrace{\{(1,n-4)\}\xrightarrow[]{\row^{w_1(n-4)}}}_{1}\underbrace{\{(1,6)\}\xrightarrow[]{\row^{w_1(6)}}}_{\left\lfloor\frac{n-6}{6}\right\rfloor(-1)^6}\underbrace{\{(1,n-1)\}\xrightarrow[]{\row^{w_1(n-1)}}}_{-1}\{(1,3)\},$$

$$\text{and }C_4=C_7: \underbrace{\{(1,4)\}\xrightarrow[]{\row^{w_1(4)}}}_{\left\lfloor\frac{n-4}{6}\right\rfloor(-1)^4}\underbrace{\{(1,n-3)\}\xrightarrow[]{\row^{w_1(n-3)}}}_{-1}\underbrace{\{(1,7)\}\xrightarrow[]{\row^{w_1(7)}}}_{\left\lfloor\frac{n-7}{6}\right\rfloor(-1)^7}\underbrace{\{(1,n)\}\xrightarrow[]{\row^{w_1(n)}}}_{1}\{(1,4)\}.$$
By assumption, $n=6q+1$ for some integer $q\geq0.$ 
 Thus, $\left\lfloor\frac{n-y}{6}\right\rfloor=q-1$ for $2\leq y\leq7,$ and 
 \begin{align*}
      \sum_{J\in \O(\{(1,i)\})}\SC(J)&=\sum_{s\in C_i}\left(\sum_{k=0}^{\wn(s)-1}\SC(\row^k(\{(1,s)\}))\right)=0
    \end{align*}
for $2\leq i\leq 7.$

\medskip

\noindent (Case 2: $n\equiv 3\bmod{6}$)
Similarly, using the subcycles of 
$$\sigma_3=(6, n-3, 7, n-2) (2, n-1, 3, n, 4, n-5, 5, n-4)$$
and Equations \eqref{eq:sc-quad_orbit_formula}, \eqref{eq:sc-partial-2n+12}, and \eqref{eq:sc-partial-edge} to calculate the sum of the signed cardinality statistic when $n\equiv3\bmod{6},$ we find 
$$\sum_{J\in \O(\{(1,i)\})}\SC(J)=\left\lfloor\frac{n-6}{6}\right\rfloor(-1)^6+(-1)+\left\lfloor\frac{n-7}{6}\right\rfloor(-1)^7+1 $$
for the orbit containing lower-chain singletons of the form $\{(1,i)\}$ for $i\equiv 6, 7\bmod{6},$ and 
$$\sum_{J\in \O(\{(1,i)\})}\SC(J)=\left\lfloor\frac{n-2}{6}\right\rfloor(-1)^2+(-1)+\left\lfloor\frac{n-3}{6}\right\rfloor(-1)^3+1+\left\lfloor\frac{n-4}{6}\right\rfloor(-1)^4+0+\left\lfloor\frac{n-5}{6}\right\rfloor(-1)^5+0$$
for the orbit containing lower-chain singletons of the form $\{(1,i)\}$ for $i\equiv 2,3,4,5\bmod{6}.$  Since $n=6q+3$ for some $q\geq 0,$
$\left\lfloor\frac{n-6}{6}\right\rfloor=\left\lfloor\frac{n-7}{6}\right\rfloor=\left\lfloor\frac{n-4}{6}\right\rfloor=\left\lfloor\frac{n-5}{6}\right\rfloor=q-1$, $\left\lfloor\frac{n-2}{6}\right\rfloor=\left\lfloor\frac{n-3}{6}\right\rfloor=q,$ and the sum of the signed cardinality statistic over each orbit is zero for $1<x\leq 7.$

\medskip

\noindent (Case 3: $n\equiv 5\bmod{6}$) Finally, by Equation \eqref{eq:big-orbits sigmas}, 
$$\sigma_5=(6, n-5, 5, n, 4, n-1, 3, n-2)(7, n-4, 2, n-3)$$
and the sum of the signed cardinality statistic is 
$$\sum_{J\in \O(\{(1,i)\})}\SC(J)=\left\lfloor\frac{n-6}{6}\right\rfloor(-1)^6+0+\left\lfloor\frac{n-5}{6}\right\rfloor(-1)^5+1+\left\lfloor\frac{n-4}{6}\right\rfloor(-1)^4+-1+\left\lfloor\frac{n-3}{6}\right\rfloor(-1)^3+1$$
for the orbit containing lower-chain singletons of the form $\{(1,i)\}$ for $i\equiv 3,4,5,6\bmod{6},$ and
$$\sum_{J\in \O(\{(1,i)\})}\SC(J)=\left\lfloor\frac{n-7}{6}\right\rfloor(-1)^7+0+\left\lfloor\frac{n-2}{6}\right\rfloor(-1)^2+(-1) $$
for the orbit containing lower-chain singletons of the form $\{(1,i)\}$ for $i\equiv 2,7\bmod{6}.$ Since $n=6q+5$ for some $q\geq 0,$
$\left\lfloor\frac{n-6}{6}\right\rfloor=\left\lfloor\frac{n-7}{6}\right\rfloor=q-1$, while $\left\lfloor\frac{n-5}{6}\right\rfloor=\left\lfloor\frac{n-4}{6}\right\rfloor=\left\lfloor\frac{n-3}{6}\right\rfloor=\left\lfloor\frac{n-2}{6}\right\rfloor=q.$ As before, it follows that for $1<x\leq 7,$ the sum of the signed cardinality statistic over each orbit is zero. 
\end{proof}

\subsection{Putting it all together: Proofs of Theorems \ref{thm:thmorbits} and \ref{thm:homomesy}}\label{ssec:all_together}
We are now ready to prove that the orbits of rowmotion described above are \textit{all} the orbits of interval-closed sets of the poset $[2]\times[n]$. To do so, we compare the enumeration of interval-closed sets for this poset obtained in \cite{ELMSW} with the number of interval-closed sets in each of the orbits described above. This is Theorem \ref{thm:thmorbits}, that we recall here:

\begin{thmorbits}
    The orbits of rowmotion on $\IC([2]\times [n])$ 
    are of sizes dividing $n+3$ or $n+5$, with the exception of a single orbit of size $2$ and the quadratic orbits described in Theorem \ref{large_orbits} and Table \ref{tab:orbit_lengths_small_posets}. In particular, the counts of orbits of each size are given in Tables \ref{tab:all_orbits} and \ref{tab:orbit_lengths_small_posets}.
\end{thmorbits}

\begin{table}[]
    \[\def\arraystretch{2.5}
    \begin{array}{c||c|c|c|c|c|c|c}
        n \mod 6 & \#\ 2& \#\ n+3 & \# \dfrac{n + 3}{2} & \# \ \dfrac{n + 3}{3} & \# \ n + 5 & \# \dfrac{n + 5}{3} & \text{other sizes}\\
        \hline
        0 & 1& \dfrac{n^3 - n^2 - 4n + 12}{12} & 0 & 1 & \dfrac{n^2 - 11n + 30}{6} & 0 & 3\\
            1 & 1 & \dfrac{n^3 - n^2 - 7n + 19}{12} & \dfrac{n - 1}{2} & 0 & \dfrac{n^2 - 11n + 28}{6} & 1& 4 \\
        2 & 1& \dfrac{n^3 - n^2 - 4n + 16}{12} & 0 & 0 & \dfrac{n^2 - 11n + 30}{6} & 0 & 3\\
        3 & 1& \dfrac{n^3 - n^2 - 7n + 15}{12} & \dfrac{n - 1}{2} & 1 & \dfrac{n^2 - 11n + 30}{6} & 0 & 2 \\
        4 & 1& \dfrac{n^3 - n^2 - 4n + 16}{12} & 0 & 0 & \dfrac{n^2 - 11n + 28}{6} & 1 & 7\\
        5 & 1& \dfrac{n^3 - n^2 - 7n + 19}{12} & \dfrac{n - 1}{2} & 0 & \dfrac{n^2 - 11n + 30}{6} & 0 & 2\\
    \end{array}
    \]
    \caption{The number of orbits of rowmotion on ICS of $[2]\times[n]$ for $n\geq 5$, by size.}
    \label{tab:all_orbits}
\end{table}

Before proving this theorem, it is useful  to recall the enumeration of ICS for $[2]\times[n]$:
\begin{thm}[Theorem 4.2 in \cite{ELMSW}]
The number of interval-closed sets of $[2]\times[n]$ is $$1+n+n^2+\frac{n+1}{2}\binom{n+2}{3} = \frac{n^4+4n^3+17n^2+14n+12}{12}.$$
\end{thm}

\begin{proof}[Proof of Theorem \ref{thm:thmorbits}] We first prove the theorem for $n$ sufficiently large, and then look at explicit orbit sizes of orbits for smaller posets.
Suppose $n \geq 5$ and let $P = [2]\times [n]$.
    From \cite[Corollary 2.11]{ELMSW}, every non-empty poset has one orbit of size $2$ containing the empty ICS and the full poset.
    Also, the number of interval-closed sets of $[2]\times[n]$ in orbits of sizes larger than $n+5$ are given by Theorem \ref{large_orbits}. There are six cases for the residue modulo $6$ of $n$, and each case yields between two and seven orbits. Adding their sizes (as they are given in  Theorem \ref{large_orbits}), one gets that the number of interval-closed sets in those large orbits is  $    3n^2+5n-30$ in each of these cases.
    
    For orbits of sizes dividing $n+3$ or $n+5$, their numbers are given by Theorems \ref{thm:n+3} and \ref{thm:n+5} and Corollary \ref{cor_n+3/2}. These are reported in Table \ref{tab:all_orbits}.
    
    To verify that those are all the orbits, we compare the number of interval-closed sets listed for each orbit sizes with the total number of them. We proceed with six cases.\\

\noindent
    (Case 1: $n \equiv 0 \pmod 6$) Let $n\equiv 0 \pmod 6$ and $n\geq 5$. Then, $[2]\times [n]$ has the following orbits:
    \begin{itemize}
        \item One orbit of size $2$, for a total of $2$ ICS
        \item $\frac{n^3-n^2-4n+12}{12}$ orbits of size $n+3$ for a total of $\frac{n^4+2n^3-7n^2+36}{12}$ ICS
        \item One orbit of size $\frac{n+3}{3}$, for a total of $\frac{n+3}{3}$ ICS
        \item $\frac{n^2-11n+30}{6}$ orbits of size $n+5$, for a total of $\frac{n^3-6n^2-25n+150}{6}$ ICS
        \item One orbit of each sizes $n^2+3n-5$, $n^2+n-12$ and $n^2+n-13$, for a total of $3n^2+5n-30$ ICS.
    \end{itemize}
    Summing all these values, one gets a total count of $\frac{n^4+4n^3+17n^2+14n+12}{12}$ ICS, matching the total number of interval-closed sets of the poset.\\

\noindent
    (Case 2: $n \equiv 1 \pmod 6$) Let $n\equiv 1 \pmod 6$ and $n\geq 5$. Then, $[2]\times [n]$ has the following orbits:
    \begin{itemize}
        \item One orbit of size $2$, for a total of $2$ ICS
        \item $\frac{n^3-n^2-7n+19}{12}$ orbits of size $n+3$ for a total of $\frac{n^4+2n^3-10n^2-2n+57}{12}$ ICS
        \item $\frac{n-1}{2}$ orbits of size $\frac{n+3}{2}$, for a total of $\frac{n^2+2n-3}{4}$ ICS
        \item $\frac{n^2-11n+28}{6}$ orbits of size $n+5$, for a total of $\frac{n^3-6n^2-27n+140}{6}$ ICS
        \item One orbit of size $\frac{n+5}{3}$ for a total of $\frac{n+5}{3}$ ICS
        \item Two orbits of size $n^2+2n-9$, one orbit of size $\frac{2n^2+n-27}{3}$ and one orbit of size $\frac{n^2+2n-9}{3}$ for a total of $3n^2+5n-30$ ICS. 
    \end{itemize}
    Summing all these values, one gets a total count of $\frac{n^4+4n^3+17n^2+14n+12}{12}$ ICS, matching the total number of interval-closed sets of the poset.\\

\noindent
    (Case 3: $n \equiv 2 \pmod 6$) Let $n\equiv 2 \pmod 6$ and $n\geq 5$. Then, $[2]\times [n]$ has the following orbits:
    \begin{itemize}
        \item One orbit of size $2$, for a total of $2$ ICS
        \item $\frac{n^3-n^2-4n+16}{12}$ orbits of size $n+3$ for a total of $\frac{n^4+2n^3-7n^2+4n+48}{12}$ ICS
        \item $\frac{n^2-11n+30}{6}$ orbits of size $n+5$, for a total of $\frac{n^3-6n^2-25n+150}{6}$ ICS
        \item One orbit of size $n^2+3n-4$ and two of size $n^2+n-13$ for a total of $3n^2+5n-30$ ICS.
    \end{itemize}
    Summing all these values, one gets a total count of $\frac{n^4+4n^3+17n^2+14n+12}{12}$ ICS, matching the total number of interval-closed sets of the poset.\\

\noindent
(Case 4: $n \equiv 3 \pmod 6$) Let $n\equiv 3 \pmod 6$ and $n\geq 5$. Then, $[2]\times [n]$ has the following orbits:
    \begin{itemize}
        \item One orbit of size $2$, for a total of $2$ ICS
        \item $\frac{n^3-n^2-7n+15}{12}$ orbits of size $n+3$ for a total of $\frac{n^4+2n^3-10n^2-10n+45}{12}$ ICS
        \item $\frac{n-1}{2}$ orbits of size $\frac{n+3}{2}$, for a total of $\frac{n^2+2n-3}{4}$ ICS
        \item One orbit of size $\frac{n+3}{3}$, for a total of $\frac{n+3}{3}$ ICS
        \item $\frac{n^2-11n+30}{6}$ orbits of size $n+5$, for a total of $\frac{n^3-6n^2-25n+150}{6}$ ICS
        \item One orbit of each size $2n^2+5n-13$ and $n^2-17$ for a total of $3n^2+5n-30$ ICS. 
    \end{itemize}
    Summing all these values, one gets a total count of $\frac{n^4+4n^3+17n^2+14n+12}{12}$ ICS, matching the total number of interval-closed sets of the poset.\\

\noindent
(Case 5: $n \equiv 4 \pmod 6$) Let $n\equiv 4 \pmod 6$ and $n\geq 5$. Then, $[2]\times [n]$ has the following orbits:
    \begin{itemize}
        \item One orbit of size $2$, for a total of $2$ ICS
        \item $\frac{n^3-n^2-4n+16}{12}$ orbits of size $n+3$ for a total of $\frac{n^4+2n^3-7n^2+4n+48}{12}$ ICS
        \item $\frac{n^2-11n+28}{6}$ orbits of size $n+5$, for a total of $\frac{n^3-6n^2-27n+140}{6}$ ICS
        \item One orbit of size $\frac{n+5}{3}$ for a total of $\frac{n+5}{3}$ ICS
        \item  Three orbits of size $\frac{n^2+2n-9}{3}$, two orbits of size $\frac{2n^2+4n-18}{3}$, one orbit of size $\frac{n^2+2n-6}{3}$ and one orbit of size $\frac{n^2-n-21}{3}$, for a total of $3n^2+5n-30$ ICS.
    \end{itemize}
    Summing all these values, one gets a total count of $\frac{n^4+4n^3+17n^2+14n+12}{12}$ ICS, matching the total number of interval-closed sets of the poset.\\

\noindent
(Case 6: $n \equiv 5 \pmod 6$) Let $n\equiv 5 \pmod 6$. Then, $[2]\times [n]$ has the following orbits:
    \begin{itemize}
        \item One orbit of size $2$, for a total of $2$ ICS
        \item $\frac{n^3-n^2-7n+19}{12}$ orbits of size $n+3$ for a total of $\frac{n^4+2n^3-10n^2-2n+57}{12}$ ICS
        \item $\frac{n-1}{2}$ orbits of size $\frac{n+3}{2}$, for a total of $\frac{n^2+2n-3}{4}$ ICS
        \item $\frac{n^2-11n+30}{6}$ orbits of size $n+5$, for a total of $\frac{n^3-6n^2-25n+150}{6}$ ICS
        \item One orbit of each size $2n^2+3n-21$ and $n^2+2n-9$, for a total of $3n^2+5n-30$ ICS. 
    \end{itemize}
    Summing all these values, one gets a total count of $\frac{n^4+4n^3+17n^2+14n+12}{12}$ ICS, matching the total number of interval-closed sets of the poset.

    Also, the orbits listed above are all distinct when $n\geq 5$, since all the orbits of quadratic size have size larger than $n+5$. For each group of orbits (size dividing $n+3$, size dividing $n+5$ and quadratic size), we described all the representatives of a given type in each orbit, ensuring that they are distinct.

    For $1\leq n \leq 4$, the hypothesis of Theorem \ref{large_orbits} is not met. However, we computed explicitly the orbit sizes and we provide them in Table \ref{tab:orbit_lengths_small_posets}.  Notice that all the orbits are of one of the sizes described above, but there are too few ICS for all the orbits listed above to exist. 
    
    This completes the verification of the statement of Theorem \ref{thm:thmorbits} for small posets.
    \begin{table}[htbp]
        \centering
        \begin{minipage}{0.45\linewidth}{\renewcommand{\arraystretch}{1.2}
            \begin{tabular}{c|c}
            $n$ & Orbit sizes \\
            \hline
            1 & 2 orbits of size 2 ($2$ and $\frac{n+5}{3}$)\\
            \hline
            2 & 1 orbit of size 2\\
            & 1 orbit of size 5 ($n+3$)\\
            & 1 orbit of size 6 ($n^2+3n-4$)\\
            \hline
            3 &  2 orbits of size 2 ($2$ and $\frac{n+3}{3}$)\\
            & 1 orbit of size 3 ($\frac{n+3}{2}$)\\
            & 1 orbit of size 6 ($n+3$)\\
            & 1 orbit of size 20 ($2n^2+5n-13$)\\
            \hline 
            4 & 1 orbit of size 2\\
            & 3 orbits of size 5 ($\frac{n^2+2n-9}{3}$)\\
            & 1 orbit of size 6 ($\frac{n^2+2n-6}{3}$)\\
            & 4 orbits of size 7 ($n+3$)\\
            & 2 orbits of size 10 ($\frac{2n^2+4n-18}{3}$)\\
            \hline
            5 & 1 orbit of size $2$\\
            & 2 orbits of size 4 ($\frac{n+3}{2}$)\\
            & 7 orbits of size 8 ($n+3$)\\
            & 1 orbit of size 26 ($n^2+2n-9$)\\
            & 1 orbit of size 44 ($2n^2+3n-21$)\\
            \hline
            6 & 1 orbit of size $2$\\
            & 1 orbit of size 3 ($\frac{n+3}{3}$)\\
            & 14 orbits of size 9 ($n+3$)\\
            & 1 orbit of size 29 ($n^2+n-13$)\\
            & 1 orbit of size 30 ($n^2+n-12$)\\
            & 1 orbit of size 49 ($n^2+3n-5$)\\
        \end{tabular}}
        \end{minipage}\quad
        \begin{minipage}{0.45\linewidth}{\renewcommand{\arraystretch}{1.2}
            \begin{tabular}{c|c}
            $n$ & Orbit sizes \\
            \hline
            7 & 1 orbit of size $2$\\
            & 1 orbit of size 4 ($\frac{n+5}{3}$)\\
            & 3 orbits of size 5 ($\frac{n+3}{2}$)\\
            & 22 orbits of size 10 ($n+3$)\\
            & 1 orbit of size 18 ($\frac{n^2+2n-9}{3}$)\\
            & 1 orbit of size 26 ($\frac{2n^2+n-27}{3}$)\\
            & 2 orbits of size 54 ($n^2+2n-9$)\\
            \hline 
            8 & 1 orbit of size $2$\\
            & 36 orbits of size 11 ($n+3$)\\
            & 1 orbit of size 13 ($n+5$)\\
            & 2 orbits of size 59 ($n^2+n-13$)\\
            & 1 orbit of size 84 ($n^2+3n-4$)\\
            \hline 
            9 & 1 orbit of size $2$\\
            & 1 orbit of size 4 ($\frac{
            n+3}{3}$)\\
            & 4 orbits of size 6 ($\frac{n+3}{2}$)\\
            & 50 orbits of size 12 ($n+3$)\\
            & 2 orbits of size 14 ($n+5$)\\
            & 1 orbits of size 64 ($n^2-17$)\\
            & 1 orbit of size 194 ($2n^2+5n-13$)\\
        \end{tabular}}
        \end{minipage}
        \caption{Explicit orbit sizes of rowmotion for $[2]\times [n]$, $1\leq n\leq 9$.}
        \label{tab:orbit_lengths_small_posets}
    \end{table}
\end{proof}

\begin{remark}\label{rem:orbits_description} This characterization of the orbit lengths includes  representatives for all orbits (see the theorems of this section cited in the proof). Our characterization shows that if an orbit contains no low interval-closed sets (those that only have elements in the first chain) or high interval-closed sets (those that have only elements in the second chain), then the orbit is of size dividing $n+3$. The vast majority of orbits are of this type.  All orbits of other sizes contain either the empty interval-closed set or an interval on a single chain. In particular, a given single-chain ICS belongs to an orbit of the following size:
\begin{itemize}
    \item $n+3$, if the ICS is either of the singleton sets $\{(1,1)\}$ or $\{(2,n)\}$, i.e., the minimum element or the maximum element of the poset;
    \item $n+5$ or $\frac{n+5}{3}$, if the ICS is an interval of size at least $4$ and either on the first chain and not including $(1,1)$, $(1,2)$ or $(1,3)$, or on the second chain and not including $(2,n)$, $(2,n-1)$ or $(2,n-2)$;
    \item quadratic size, if the ICS is an interval of size at least $4$ and includes any of $(1,1)$, $(1,2)$, $(1,3)$, $(2,n)$, $(2,n-1)$ or $(2,n-2)$;
    \item quadratic size, if the ICS has size at most $3$, and it is neither $\{(1,1)\}$ nor $\{(2,n)\}$. (In particular, all the orbits of quadratic size contain a singleton representative, except for one orbit with no singleton when $n\equiv1 \pmod3$. The latter orbit has a lower-chain representative of size $2$.)
\end{itemize}
\end{remark}

With the complete description of the orbits from Theorem~\ref{thm:thmorbits}, we can now prove Theorem \ref{thm:homomesy}.
\begin{thmhomomesy} 
    Suppose that $n$ is odd. Then the signed cardinality statistic on $\IC([2] \times [n])$ is $0$-mesic with respect to rowmotion. 
\end{thmhomomesy}

\begin{proof}
We verified with SageMath \cite{sage} that all the orbits of $[2]\times [n]$ have an average signed cardinality of $0$ for $n \in \{1,3,5,7,9\}$.
Now assume that $n\geq 11$ and that $n$ is odd.
    The proof of Theorem \ref{thm:thmorbits} gives a complete description of all the orbits. Those are one orbit of size 2, orbits of size dividing $n+3$ and $n+5$ and orbits of quadratic sizes. The orbit of size $2$ contains two ICS for which the signed cardinality is $0$, which are the full poset and the empty poset.
    Following Theorem \ref{thm:sc_n+3} and Corollary \ref{cor:sc_n+3}, the signed cardinality also has an average of $0$ over orbits of size dividing $n+3$. Theorem \ref{thm:sc_n+5} shows that the average value of signed cardinality over the orbits of size dividing $n+5$ is also $0$. As for the quadratic orbits, all of them contain singleton representatives in the lower chain, except when $n \equiv 1 \pmod{3}$, in which case there is one orbit with representative $\{(1, n-2),(1,n-1)\}$.
    Theorem \ref{lem:sc-big-orbits} says that the average value of signed cardinality over the orbits with lower-chain singleton representatives is 0. Lemma \ref{prop:big orbit 2phase} means that the signed cardinality sums to $0$ for the one orbit with no singleton when $n\equiv1 \pmod3$. Together, the last two statements mean that all the orbits of quadratic size also have an average signed cardinality of $0$, from which we conclude that all orbits of $[2]\times[n]$ under rowmotion have an average signed cardinality of $0$. Hence, the signed cardinality  on $\IC([2]\times[n]$ is $0$-mesic with respect to rowmotion when $n$ is odd.
\end{proof}
It is worth noting that homomesy only occurs for the $[2]\times[n]$ poset when $n$ is odd, as shown in the proposition below. 
\begin{prop}\label{prop:no_homomesy_n_even} Let $n$ be positive and even. Then, the signed cardinality statistic is not homomesic under rowmotion for the poset $[2]\times[n]$.
\end{prop}
\begin{proof}
    For homomesy, we need the global average of signed cardinality to be the same as the average over every orbit. Under rowmotion, every nonempty poset has an orbit of size $2$ containing the whole poset and the empty ICS, which both have signed cardinality $0$. However, the global average of signed cardinality is not $0$ for the poset $[2]\times[n]$ when $n$ is even. Indeed, the average would be $0$ only if the sum of the signed cardinality over all ICS would be $0$. We show here that the total signed cardinality is $\frac{-n^2(n+1)}{4}$ over ICS of $[2]\times[n]$ with $n$ even.

Interval-closed sets of $[2]\times[n]$ are of one of the following types, which are disjoint:
\begin{itemize}
    \item an interval-closed set of $[2]\times[n-1]$ embedded in the $[2]\times[n]$ poset. Since rowmotion is $0$-mesic for ICS of $[2]\times[n-1]$ when $n$ is even, the sum of the signed cardinality for these ICS is $0$.
    \item a low ICS $[b,n-b,0:\varnothing]$ or a high ICS $[\varnothing:b,n-b,0]$, for $0\leq b \leq n-1$. For any given value of $b$, the low ICS $[b,n-b,0:\varnothing]$ has opposite signed cardinality value of that of the high ICS $[\varnothing: b,n-b,0]$, so the total sum of the signed cardinality for these ICS is $0$.
    \item an interval-closed set with elements in both the lower and upper chains, and containing $(1,n)$. This is the ICS $[b_1, n-b_1, 0: b_2, i_2, n-i_2-b_2]$ with $1\leq i_2\leq n$, $0\leq b_1\leq n-1$ and $0\leq b_2 \leq n-i_2$. Following Lemma \ref{calc_sc(I)}, the total signed cardinality for these ICS is 
    \begin{align*}
        \sum_{i_2=1}^{n} \sum_{b_1=0}^{n-1}\sum_{b_2=0}^{n-i_2} \left(\delta_{b_1}(-1)^{b_1}+\delta_{i_2}(-1)^{b_2+1}\right) &= \sum_{i_2=1}^{n} \sum_{b_1=0}^{n-1}\sum_{b_2=0}^{n-i_2} -\delta_{b_1}+ n\sum_{i_2=1}^{n} \sum_{b_2=0}^{n-i_2} \delta_{i_2}(-1)^{b_2+1}\\ &= -\frac{n}{2}\sum_{i_2=1}^n (n-i_2+1) +0\\ & = -\frac{n^2(n+1)}{4},
    \end{align*}
    where the second-to-last equality comes from parity analysis: when $n$ is even and $i_2$ is odd, there is an even number of values of $b_2$ between $0$ and $n-i_2$, meaning that the sum of $(-1)^{b_2+1}$ over these values of $b_2$ is $0$.
\end{itemize}
Therefore, signed cardinality cannot be homomesic under rowmotion for $[2]\times[n]$ with $n$ even.
\end{proof}

It is natural to ask whether signed cardinality is homomesic for other posets that can be expressed as a product of chains.
\begin{remark}
    As noted in \cite[Conjecture 4.12]{ELMSW}, the signed cardinality statistics seems to be homomesic under rowmotion for $[3]\times[n]$ when $n$ is even. For reasons analogous to those behind Proposition \ref{prop:no_homomesy_n_even},  the statistic is not homomesic for $[3]\times[n]$ when $n$ is odd.

    The chain poset, $[n]$, also exhibits homomesy for the signed cardinality poset, as shown in \cite[Proposition 3.18]{ELMSW}.
    
    However, hopes for a general statement that signed cardinality is homomesic under rowmotion for products of chains are limited by the following counter-examples, among many others: $[4 ]\times[4]$, $[4 ]\times[5]$, $[2]\times[2]\times[3]$, $[2]\times[2]\times[4]$, and $[2]\times[2]\times[2]\times[2]$.
\end{remark}

\section{Conjectures and future work}
In this section, we discuss some conjectures and directions for future work.
\label{sec:conj}
\subsection{Rowmotion orbits of interval-closed sets of $[m]\times[n]$}

Experimental evidence suggests that it will not be possible to fully catalogue rowmotion orbits in $[m] \times [n]$ for $m > 2$ and arbitrary $n$.  Nevertheless, some steps in this direction may be possible. It follows from Theorem~\ref{thm:thmorbits} that for large $n$, the vast majority of ICS of $[2] \times [n]$ belong to orbits of size $n + 3$.  Experimental evidence suggests that this extends to $[m] \times [n]$, in the following sense. 

\begin{conjecture}\label{conj:m + n + 1}
    Fix any positive integer $m$.  Let $r(n)$ be the fraction of interval-closed sets of $[m] \times [n]$ that belong to orbits of size $m + n + 1$.  Then $r(n) \to 1$ as $n \to \infty$.
\end{conjecture}

One may view this conjecture as asserting that, although the overall behavior of rowmotion on ICS is complicated and of large order, its behavior on \emph{typical} ICS is very nice. It would be possible in principle to prove Conjecture~\ref{conj:m + n + 1} by identifying a sufficiently large family of orbits of size $m + n + 1$, and comparing with the enumeration of all ICS. We outline a strategy in this direction here.

Say that an ICS $I$ in $[m] \times [n]$ is \emph{good} if 
\[
I = [(1, a_1), (1, b_1)] \cup [(2, a_2), (2, b_2)] \cup \cdots \cup [(m, a_m), (m, b_m)]
\]
such that no two of $\{a_1, \ldots, a_m, b_1, \ldots, b_m\}$ differ by less than $2$.  The number of good ICS is $C_m \cdot \binom{n - 2m + 1}{2m}$, where $C_m$ is the $m$th Catalan number $C_m=\frac{1}{m+1}\binom{2m}{m}$. The binomial coefficient $\binom{n - 2m + 1}{2m}$ counts the number of ways to choose $2m$ distinct elements from $\{1, \ldots, n\}$ such that no two differ by $1$, and there are $C_m$ ways to assign these $2m$ distinct values to be $a_1, \ldots, a_m, b_1, \ldots, b_m$ (the conditions $a_1 > a_2 > \cdots > a_m$, $b_1 > \cdots > b_m$, and $a_i < b_i$ on the assignment are essentially those of a standard Young tableau of shape $\langle m, m\rangle$).  If $m$ is fixed, then as $n \to \infty$, this number grows like $\frac{n^{2m}}{m!(m + 1)!} + O(n^{2m - 1})$.  Then \emph{if} one could show that all such elements belong to orbits of size $m + n + 1$, \emph{and} one could show that the total number of ICS of $[m] \times [n]$ for fixed $m$ and $n \to \infty$ is asymptotic to $\frac{n^{2m}}{m!(m + 1)!}$, that would be sufficient to prove the conjecture.  One should be able to extract the latter (enumerative) piece from the formulas in \cite{ELLMSW}.

\subsection{Max-minus-min homomesy}
Our second main result, Theorem~\ref{thm:homomesy}, proved a homomesy conjecture of \cite{ELMSW}. We state below another conjecture of that paper and give some remarks. The $m=2$ case was proved in \cite[Thm.\ 4.7]{ELMSW}.
\begin{conjecture}[\protect{\cite[Conj.~4.9]{ELMSW}}]\label{conj:max-min}
    The number of maximal elements minus the number of minimal elements is $0$-mesic under rowmotion on $\IC(P)$, for $P=[m] \times [n]$.
\end{conjecture}
\begin{remark}
    The proof in \cite[Thm.\ 4.7]{ELMSW} of this statement in the case $m=2$ proceeded by matching each interval-closed set that contributes $+1$ with an interval-closed set in the same orbit contributing $-1$. In that case, $+1$, $-1$ and $0$ are the only possible values for the statistics, since all nonempty antichains have either size $1$ or $2$ in $[2]\times [n]$. For larger products of chains, it is not possible to do the same. For instance, in $[3]\times [5]$, nonempty antichains have size $1$, $2$ or $3$, so the max-minus-min statistic takes values between $-2$ and $+2$. However, it is not always possible to pair, over a given orbit, interval-closed sets with values $i$ and $-i$. For example, the orbit depicted in Figure~\ref{fig:ICS-with-max-minus-min} has two ICS for which the statistic has value $-2$ and only one for which the statistic has value $+2$.
\begin{figure}[htbp]
\begin{center}
        \begin{tabular}{c|c|c|c|c|c|c|c|c}
\hspace{-.75em}\rotatebox{-15}{\begin{tikzpicture}[scale=.2]
\foreach \x in {1,2,3}
	{\foreach \y in {1,...,5}
		{\ifthenelse{\x < 3}
			{\draw (\x - \y, \x + \y) -- (\x - \y + 1, \x + \y + 1);}{}
		\ifthenelse{\y < 5}
			{\draw (\x - \y, \x + \y) -- (\x - \y - 1, \x + \y+1);}{}
   \draw[fill=white, radius = .3] (\x - \y, \x + \y) circle;
		}
	}
\draw[fill=red, radius = .3] (1-5,1+5) circle;
\draw[fill=red, radius = .3] (2-5,2+5) circle;
\draw[fill=red, radius = .3] (2-4,2+4) circle;
\draw[fill=red, radius = .3] (3-5,3+5) circle;
\draw[fill=red, radius = .3] (3-4,3+4) circle;
\draw[fill=red, radius = .3] (3-3,3+3) circle;
\draw[fill=red, radius = .3] (3-2,3+2) circle;
\end{tikzpicture}}\hspace{-.75em}  &
\hspace{-.75em}\rotatebox{-15}{\begin{tikzpicture}[scale=.2]
\foreach \x in {1,2,3}
	{\foreach \y in {1,...,5}
		{\ifthenelse{\x < 3}
			{\draw (\x - \y, \x + \y) -- (\x - \y + 1, \x + \y + 1);}{}
		\ifthenelse{\y < 5}
			{\draw (\x - \y, \x + \y) -- (\x - \y - 1, \x + \y+1);}{}
   \draw[fill=white, radius = .3] (\x - \y, \x + \y) circle;
		}
	}
\draw[fill=red, radius = .3] (1-1,1+1) circle;
\draw[fill=red, radius = .3] (1-2,1+2) circle;
\draw[fill=red, radius = .3] (1-3,1+3) circle;
\draw[fill=red, radius = .3] (1-4,1+4) circle;
\draw[fill=red, radius = .3] (2-1,2+1) circle;
\draw[fill=red, radius = .3] (2-2,2+2) circle;
\draw[fill=red, radius = .3] (2-3,2+3) circle;
\draw[fill=red, radius = .3] (3-1,3+1) circle;
\end{tikzpicture}}\hspace{-.75em}  &
\hspace{-.75em}\rotatebox{-15}{\begin{tikzpicture}[scale=.2]
\foreach \x in {1,2,3}
	{\foreach \y in {1,...,5}
		{\ifthenelse{\x < 3}
			{\draw (\x - \y, \x + \y) -- (\x - \y + 1, \x + \y + 1);}{}
		\ifthenelse{\y < 5}
			{\draw (\x - \y, \x + \y) -- (\x - \y - 1, \x + \y+1);}{}
   \draw[fill=white, radius = .3] (\x - \y, \x + \y) circle;
		}
	}
\draw[fill=red, radius = .3] (1-2,1+2) circle;
\draw[fill=red, radius = .3] (1-3,1+3) circle;
\draw[fill=red, radius = .3] (1-4,1+4) circle;
\draw[fill=red, radius = .3] (1-5,1+5) circle;
\draw[fill=red, radius = .3] (2-1,2+1) circle;
\draw[fill=red, radius = .3] (2-2,2+2) circle;
\draw[fill=red, radius = .3] (2-3,2+3) circle;
\draw[fill=red, radius = .3] (2-4,2+4) circle;
\draw[fill=red, radius = .3] (3-1,3+1) circle;
\draw[fill=red, radius = .3] (3-2,3+2) circle;
\end{tikzpicture}}\hspace{-.75em}  &
\hspace{-.75em}\rotatebox{-15}{\begin{tikzpicture}[scale=.2]
\foreach \x in {1,2,3}
	{\foreach \y in {1,...,5}
		{\ifthenelse{\x < 3}
			{\draw (\x - \y, \x + \y) -- (\x - \y + 1, \x + \y + 1);}{}
		\ifthenelse{\y < 5}
			{\draw (\x - \y, \x + \y) -- (\x - \y - 1, \x + \y+1);}{}
   \draw[fill=white, radius = .3] (\x - \y, \x + \y) circle;
		}
	}
\draw[fill=red, radius = .3] (1-3,1+3) circle;
\draw[fill=red, radius = .3] (1-4,1+4) circle;
\draw[fill=red, radius = .3] (1-5,1+5) circle;
\draw[fill=red, radius = .3] (2-2,2+2) circle;
\draw[fill=red, radius = .3] (2-3,2+3) circle;
\draw[fill=red, radius = .3] (2-4,2+4) circle;
\draw[fill=red, radius = .3] (2-5,2+5) circle;
\draw[fill=red, radius = .3] (3-1,3+1) circle;
\draw[fill=red, radius = .3] (3-2,3+2) circle;
\draw[fill=red, radius = .3] (3-3,3+3) circle;
\end{tikzpicture}}\hspace{-.75em}  &
\hspace{-.75em}\rotatebox{-15}{\begin{tikzpicture}[scale=.2]
\foreach \x in {1,2,3}
	{\foreach \y in {1,...,5}
		{\ifthenelse{\x < 3}
			{\draw (\x - \y, \x + \y) -- (\x - \y + 1, \x + \y + 1);}{}
		\ifthenelse{\y < 5}
			{\draw (\x - \y, \x + \y) -- (\x - \y - 1, \x + \y+1);}{}
   \draw[fill=white, radius = .3] (\x - \y, \x + \y) circle;
		}
	}
\draw[fill=red, radius = .3] (1-4,1+4) circle;
\draw[fill=red, radius = .3] (2-3,2+3) circle;
\draw[fill=red, radius = .3] (2-4,2+4) circle;
\draw[fill=red, radius = .3] (3-2,3+2) circle;
\draw[fill=red, radius = .3] (3-3,3+3) circle;
\draw[fill=red, radius = .3] (3-4,3+4) circle;
\end{tikzpicture}}\hspace{-.75em}  &
\hspace{-.75em}\rotatebox{-15}{\begin{tikzpicture}[scale=.2]
\foreach \x in {1,2,3}
	{\foreach \y in {1,...,5}
		{\ifthenelse{\x < 3}
			{\draw (\x - \y, \x + \y) -- (\x - \y + 1, \x + \y + 1);}{}
		\ifthenelse{\y < 5}
			{\draw (\x - \y, \x + \y) -- (\x - \y - 1, \x + \y+1);}{}
   \draw[fill=white, radius = .3] (\x - \y, \x + \y) circle;
		}
	}
\draw[fill=red, radius = .3] (1-5,1+5) circle;
\draw[fill=red, radius = .3] (2-1,2+1) circle;
\draw[fill=red, radius = .3] (2-2,2+2) circle;
\draw[fill=red, radius = .3] (3-1,3+1) circle;
\end{tikzpicture}}\hspace{-.75em}  &
\hspace{-.75em}\rotatebox{-15}{\begin{tikzpicture}[scale=.2]
\foreach \x in {1,2,3}
	{\foreach \y in {1,...,5}
		{\ifthenelse{\x < 3}
			{\draw (\x - \y, \x + \y) -- (\x - \y + 1, \x + \y + 1);}{}
		\ifthenelse{\y < 5}
			{\draw (\x - \y, \x + \y) -- (\x - \y - 1, \x + \y+1);}{}
   \draw[fill=white, radius = .3] (\x - \y, \x + \y) circle;
		}
	}
\draw[fill=red, radius = .3] (1-2,1+2) circle;
\draw[fill=red, radius = .3] (1-3,1+3) circle;
\draw[fill=red, radius = .3] (1-4,1+4) circle;
\draw[fill=red, radius = .3] (2-2,2+2) circle;
\draw[fill=red, radius = .3] (2-3,2+3) circle;
\draw[fill=red, radius = .3] (3-1,3+1) circle;
\draw[fill=red, radius = .3] (3-2,3+2) circle;
\end{tikzpicture}}\hspace{-.75em}  &
\hspace{-.75em}\rotatebox{-15}{\begin{tikzpicture}[scale=.2]
\foreach \x in {1,2,3}
	{\foreach \y in {1,...,5}
		{\ifthenelse{\x < 3}
			{\draw (\x - \y, \x + \y) -- (\x - \y + 1, \x + \y + 1);}{}
		\ifthenelse{\y < 5}
			{\draw (\x - \y, \x + \y) -- (\x - \y - 1, \x + \y+1);}{}
   \draw[fill=white, radius = .3] (\x - \y, \x + \y) circle;
		}
	}
\draw[fill=red, radius = .3] (1-3,1+3) circle;
\draw[fill=red, radius = .3] (1-4,1+4) circle;
\draw[fill=red, radius = .3] (1-5,1+5) circle;
\draw[fill=red, radius = .3] (2-2,2+2) circle;
\draw[fill=red, radius = .3] (2-3,2+3) circle;
\draw[fill=red, radius = .3] (2-4,2+4) circle;
\draw[fill=red, radius = .3] (3-2,3+2) circle;
\draw[fill=red, radius = .3] (3-3,3+3) circle;
\end{tikzpicture}}\hspace{-.75em} &
\hspace{-.75em}\rotatebox{-15}{\begin{tikzpicture}[scale=.2]
\foreach \x in {1,2,3}
	{\foreach \y in {1,...,5}
		{\ifthenelse{\x < 3}
			{\draw (\x - \y, \x + \y) -- (\x - \y + 1, \x + \y + 1);}{}
		\ifthenelse{\y < 5}
			{\draw (\x - \y, \x + \y) -- (\x - \y - 1, \x + \y+1);}{}
   \draw[fill=white, radius = .3] (\x - \y, \x + \y) circle;
		}
	}
\draw[fill=red, radius = .3] (1-4,1+4) circle;
\draw[fill=red, radius = .3] (1-5,1+5) circle;
\draw[fill=red, radius = .3] (2-3,2+3) circle;
\draw[fill=red, radius = .3] (2-4,2+4) circle;
\draw[fill=red, radius = .3] (2-5,2+5) circle;
\draw[fill=red, radius = .3] (3-1,3+1) circle;
\draw[fill=red, radius = .3] (3-2,3+2) circle;
\draw[fill=red, radius = .3] (3-3,3+3) circle;
\draw[fill=red, radius = .3] (3-4,3+4) circle;
\end{tikzpicture}}\hspace{-.75em} \\\hline
              $-2$ & $2$ & $1$ & $-1$ & $-2$& $1$& $1$& $1$& $-1$ 
        \end{tabular}
    \end{center}
\caption{A rowmotion orbit of ICS with the statistic $\max - \min$ as indicated}
\label{fig:ICS-with-max-minus-min}
\end{figure}
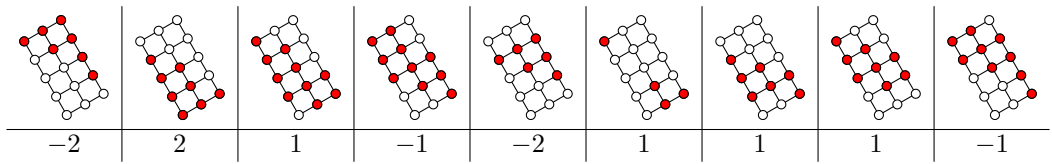
\end{remark}

\section*{Acknowledgements}
The authors thank Sergi Elizalde and Corey Vorland for helpful conversations and the developers of SageMath~\cite{sage}, which was useful in this research.
This work benefited from the opportunity to present and collaborate at the Fall 2023 AMS Central Sectional Meeting.
Striker was supported by Simons Foundation gift MP-TSM-00002802 and NSF grant DMS-2247089.  Lewis was supported by Simons Foundation gift MPS-TSM-00006960.

\bibliographystyle{abbrv}
\bibliography{master}

\begin{thebibliography}{10}

\bibitem{AST2013}
D.~Armstrong, C.~Stump, and H.~Thomas.
\newblock A uniform bijection between nonnesting and noncrossing partitions.
\newblock {\em Trans. Amer. Math. Soc.}, 365(8):4121--4151, 2013.

\bibitem{Barnard2021}
E.~Barnard, G.~Todorov, and S.~Zhu.
\newblock Dynamical combinatorics and torsion classes.
\newblock {\em J. Pure Appl. Algebra}, 225(9):Paper No. 106642, 25, 2021.

\bibitem{BSV21}
J.~Bernstein, J.~Striker, and C.~Vorland.
\newblock {$P$}-strict promotion and {$B$}-bounded rowmotion, with applications to tableaux of many flavors.
\newblock {\em Comb. Theory}, 1:Paper No. 8, 46, 2021.

\bibitem{Brouwer1975}
A.~Brouwer.
\newblock On dual pairs of antichains.
\newblock {\em Stichting Mathematisch Centrum. Zuivere Wiskunde}, (ZW 40/75):1--8, 1975.

\bibitem{CF1995}
P.~J. Cameron and D.~G. Fon-Der-Flaass.
\newblock Orbits of antichains revisited.
\newblock {\em European J. Combin.}, 16(6):545--554, 1995.

\bibitem{DHPP23}
C.~Defant, S.~Hopkins, S.~Poznanovi\'c, and J.~Propp.
\newblock Homomesy via toggleability statistics.
\newblock {\em Comb. Theory}, 3(2):Paper No. 14, 61, 2023.

\bibitem{Semidistrim}
C.~Defant and N.~Williams.
\newblock Semidistrim lattices.
\newblock {\em Forum Math. Sigma}, 11:Paper No. e50, 35, 2023.

\bibitem{Duchet1974}
P.~Duchet.
\newblock Sur les hypergraphes invariants.
\newblock {\em Discrete Math.}, 8:269--280, 1974.

\bibitem{EinsteinPropp21}
D.~Einstein and J.~Propp.
\newblock Combinatorial, piecewise-linear, and birational homomesy for products of two chains.
\newblock {\em Algebr. Comb.}, 4(2):201--224, 2021.

\bibitem{EinsteinFGJMPR16}
D.~M. Einstein, M.~Farber, E.~Gunawan, M.~Joseph, M.~Macauley, J.~Propp, and S.~Rubinstein{-}Salzedo.
\newblock Noncrossing partitions, toggles, and homomesies.
\newblock {\em Electron. J. Comb.}, 23(3):Paper 52, 26 pages, 2016.

\bibitem{ELMSW}
J.~Elder, N.~Lafreni{\`e}re, E.~McNicholas, J.~Striker, and A.~Welch.
\newblock Toggling, rowmotion, and homomesy on interval-closed sets.
\newblock {\em Journal of Combinatorics}, 15(4):479--528, 2024.

\bibitem{ELLMSW}
S.~Elizalde, N.~Lafrenière, J.~B. Lewis, E.~McNicholas, J.~Striker, and A.~Welch.
\newblock Motzkin paths, quarter-plane walks, and interval-closed set enumeration.
\newblock {\em Preprint}.
\newblock \url{https://arxiv.org/abs/2412.16368}.

\bibitem{HopkinsOPAC}
S.~Hopkins.
\newblock Order polynomial product formulas and poset dynamics.
\newblock In {\em Open problems in algebraic combinatorics}, volume 110 of {\em Proc. Sympos. Pure Math.}, pages 135--157. Amer. Math. Soc., Providence, RI, 2024.

\bibitem{IyamaMarczinzik22}
O.~Iyama and R.~Marczinzik.
\newblock Distributive lattices and {A}uslander regular algebras.
\newblock {\em Adv. Math.}, 398:Paper No. 108233, 27, 2022.

\bibitem{Joseph19}
M.~Joseph.
\newblock Antichain toggling and rowmotion.
\newblock {\em Electron. J. Combin.}, 26(1):Paper No. 1.29, 43, 2019.

\bibitem{PR2015}
J.~Propp and T.~Roby.
\newblock Homomesy in products of two chains.
\newblock {\em Electron. J. Combin.}, 22(3):Paper 3.4, 29 pages, 2015.

\bibitem{Roby2016}
T.~Roby.
\newblock Dynamical algebraic combinatorics and the homomesy phenomenon.
\newblock In {\em Recent Trends in Combinatorics}, pages 619--652, Cham, 2016.

\bibitem{Stanley2009}
R.~P. Stanley.
\newblock Promotion and evacuation.
\newblock {\em Electron. J. Combin.}, 16(2, Special volume in honor of Anders Bj\"orner):Research Paper 9, 24, 2009.

\bibitem{Stanley2011}
R.~P. Stanley.
\newblock {\em Enumerative Combinatorics: Volume 1}.
\newblock Cambridge University Press, New York, NY, USA, 2nd edition, 2011.

\bibitem{sage}
W.~Stein et~al.
\newblock {\em {S}age {M}athematics {S}oftware ({V}ersion 9.4)}.
\newblock The Sage Development Team, 2022.
\newblock \url{http://www.sagemath.org}.

\bibitem{Striker2017}
J.~Striker.
\newblock Dynamical algebraic combinatorics: {P}romotion, rowmotion, and resonance.
\newblock {\em Notices Amer. Math. Soc.}, 64(6):543--549, 2017.

\bibitem{Striker2018}
J.~Striker.
\newblock Rowmotion and generalized toggle groups.
\newblock {\em Discrete Math. Theor. Comput. Sci.}, 20(1):Paper No. 17, 26, 2018.

\bibitem{SW2012}
J.~Striker and N.~Williams.
\newblock Promotion and rowmotion.
\newblock {\em European J. Combin.}, 33(8):1919--1942, 2012.

\bibitem{RowmotionSlowmotion}
H.~Thomas and N.~Williams.
\newblock Rowmotion in slow motion.
\newblock {\em Proc. Lond. Math. Soc. (3)}, 119(5):1149--1178, 2019.

\end{thebibliography}

\end{document}